\documentclass[reqno]{amsart}

\usepackage[english]{babel}
\usepackage[utf8]{inputenc} 

\usepackage{amsfonts,amssymb,amsmath,amsthm}
\numberwithin{equation}{section}
\usepackage{dsfont}
\usepackage{multirow}
\usepackage{url}
\usepackage{xcolor} 
\usepackage{graphics}
\usepackage{booktabs}
\usepackage{epsfig}
\usepackage{epstopdf}
\usepackage{psfrag}                
\usepackage{mathrsfs}
\usepackage{mathtools}
\mathtoolsset{showonlyrefs}    
\usepackage{subfigure}
\usepackage{nicefrac}
\usepackage{bm}
\usepackage{diagbox}
\usepackage{caption}
\usepackage{paralist}
\usepackage[shortlabels,inline]{enumitem}

\usepackage[pdftex,
	pdftitle={Multiple and weak Markov properties in Hilbert spaces},
	pdfauthor={K.~Kirchner and J.~Willems},
	bookmarksopen,
	colorlinks,
	linkcolor=black,
	urlcolor=black,
	citecolor=black
	]{hyperref}

\newcommand{\bbN}{\mathbb{N}}
\newcommand{\bbP}{\mathbb{P}}
\newcommand{\bbR}{\mathbb{R}}

\newcommand{\bbT}{\mathbb{T}}

\newcommand{\R}{\mathbb{R}}

\newcommand{\N}{\mathbb{N}}

\newcommand{\E}{\mathbb{E}}

\newcommand{\cA}{\mathcal{A}}
\newcommand{\cB}{\mathcal{B}}

\newcommand{\cD}{\mathcal{D}}

\newcommand{\cF}{\mathcal{F}}
\newcommand{\cG}{\mathcal{G}}
\newcommand{\cH}{\mathcal{H}}
\newcommand{\cI}{\mathcal{I}}
\newcommand{\cJ}{\mathcal{J}}
\newcommand{\cK}{\mathcal{K}}
\newcommand{\cL}{\mathcal{L}}

\newcommand{\cN}{\mathcal{N}}
\newcommand{\cO}{\mathcal{O}}
\newcommand{\cP}{\mathcal{P}}
\newcommand{\cQ}{\mathcal{Q}}

\newcommand{\cS}{\mathcal{S}}
\newcommand{\cT}{\mathcal{T}}

\newcommand{\cW}{\mathcal{W}}
\newcommand{\cX}{\mathcal{X}}

\newcommand{\norm}[2]{     \| #1       \|_{ #2 }}

\newcommand{\Norm}[2]{\left\| #1 \right\|_{ #2 }}
\newcommand{\scalar}[2]{     \langle #1       \rangle_{ #2 }}
\newcommand{\Scalar}[2]{\left\langle #1 \right\rangle_{ #2 }}

\newcommand{\tr}{\operatorname{tr}}
 
\newcommand{\rd}{\mathop{}\!\mathrm{d}}
\newcommand{\from}{\colon}

\newcommand{\clos}[1]{\overline{ #1 }} 
\newcommand{\id}{\operatorname{Id}}
\newcommand{\dom}[1]{\mathsf D(#1)}
\newcommand{\range}[1]{\mathsf R(#1)}

\newcommand{\dual}[1]{#1^*}

\newcommand{\ind}{\perp\!\!\!\perp}

\newcommand{\supp}{\operatorname{supp}}

\newcommand{\LO}{\mathscr L}

\newtheorem{lemma}{Lemma}[section]
\newtheorem{proposition}[lemma]{Proposition}
\newtheorem{theorem}[lemma]{Theorem}
\newtheorem{corollary}[lemma]{Corollary}

\theoremstyle{remark}
\newtheorem{remark}[lemma]{Remark}

\theoremstyle{definition}
\newtheorem{definition}[lemma]{Definition}
\newtheorem{assumption}[lemma]{Assumption}
\newtheorem{example}[lemma]{Example}

\begin{document}

\author{Kristin Kirchner \and Joshua Willems}

\address[Kristin Kirchner]{Delft Institute of Applied Mathematics\\
	Delft University of Technology\\
	P.O.~Box 5031 \\ 
	2600 GA Delft \\
	The Netherlands.}

\email{k.kirchner@tudelft.nl}

\address[Joshua Willems]{Delft Institute of Applied Mathematics\\
	Delft University of Technology\\
	P.O.~Box 5031 \\ 
	2600 GA Delft \\
	The Netherlands.}

\email{j.willems@tudelft.nl}


\title[Multiple and weak Markov properties in Hilbert spaces]{%
	Multiple and weak Markov properties in Hilbert spaces with applications to fractional stochastic evolution equations}




\date{}


\begin{abstract}
	We define various higher-order Markov properties 
	for stochastic processes
	$(X(t))_{t\in \mathbb{T}}$,
	indexed by an interval $\mathbb{T} \subseteq \mathbb{R}$ and
	taking values in a real and 
	separable Hilbert space~$U$.
	We furthermore investigate 
	the relations between them. 
	In particular, for solutions to the 
	stochastic evolution equation $\mathcal{L} X = \dot W^Q\!$, where
	$\mathcal{L}$ is a linear operator acting on functions mapping
	from $\mathbb{T}$ to $U$ and 
	$(\dot W^Q(t))_{t\in\mathbb{T}}$ is the formal derivative of a $U$-valued (cylindrical) $Q$-Wiener process,
	we prove necessary and sufficient conditions for the weakest Markov property
	via locality of the precision operator $\mathcal{L}^*\! \mathcal{L}$.

	As an application, we consider the space--time 
	fractional parabolic operator $\mathcal{L} = {(\partial_t + A)^\gamma}$
	of order ${\gamma \in (1/2,\infty)}$, where
	$-A$ is a linear operator generating a $C_0$-semigroup on~$U$.
	We prove that the resulting solution process satisfies an $N$th order Markov property
	if $\gamma = N \in \mathbb{N}$ and show that a necessary condition for the weakest Markov property
	is generally not satisfied if $\gamma \notin \mathbb{N}$.
	The relevance of this class of processes
	is twofold: Firstly, it can be seen as a 
	spatiotemporal generalization of 
	Whittle--Mat\'ern Gaussian random fields 
	if $U = L^2(\mathcal{D})$ for a spatial domain $\mathcal{D}\subseteq\mathbb{R}^d\!$. 
	Secondly, we show that
	a $U$-valued analog to the fractional Brownian motion 
	with Hurst parameter $H \in (0,1)$ can be obtained 
	as the limiting case of
	$\mathcal{L} = (\partial_t + \varepsilon \, \mathrm{Id}_U)^{H + \frac{1}{2}}$
	for $\varepsilon \downarrow 0$.
\end{abstract}

\maketitle

\section{Introduction}
\label{section:intro}
\subsection{Background and motivation}
Gaussian Markov random fields play an important role
for various applications, such as
the analysis of time series or longitudinal data, image processing and spatial statistics, 
see e.g.~\cite[Section~1.3]{Rue2005}.
The latter
focuses on the statistical modeling
of spatial or spatiotemporal
dependence in data
collected from phenomena encountered in disciplines 
such as 
climatology~\cite{Alexeeff2018}, 
epidemiology~\cite{Lawson2014} and 
neuroimaging~\cite{MejiaEtAl2020}. 
The popularity of Gaussian Markov random fields among
the larger class of Gaussian random fields is a consequence
of their additional conditional independence properties,
which entail a sparse precision structure and 
facilitate efficient computational methods for statistical inference.
In particular, hierarchical models based on Gaussian Markov random fields allow for efficient
Bayesian inference using Markov chain Monte Carlo methods, see e.g.~\cite[Section~4.1]{Rue2005}.

Since a Gaussian process is fully characterized 
by its second-order structure, i.e., 
the mean and covariance function, 
a natural way to specify its distribution is 
to choose a suitable second-order structure.  
Alternatively, the dynamics of Gaussian random fields defined on 
a Euclidean domain $\cD \subseteq \R^d$ can be specified
by means of stochastic partial differential equations 
(SPDEs), such as the white noise $(\cW(x))_{x\in\cD}$ driven equation 
\begin{equation}\label{eq:abstract-spatial-SPDE}
	L X(x) = \cW(x), \quad x \in \cD. 
\end{equation}
Here, $L$ is 
a linear operator 
acting on real-valued functions defined on~$\cD$.
A spatial
Gaussian random field $(X(x))_{x \in \cD}$ is said to have the Markov property
if the subcollections $(X(x))_{x \in \cD_{1}}$
and $(X(x))_{x \in \cD_{2}}$
corresponding to
pairs of disjoint subdomains $\cD_1, \cD_2 \subseteq \cD$
are independent
conditional on $(X(x))_{x \in \cD'}$ for 
some non-trivial `splitting' set $\cD' \subseteq \cD$ separating the two.
The precise specification of these sets, 
which respectively carry  
the intuitive interpretations
of \emph{past}, \emph{future} and \emph{present}, leads to 
various definitions of the Markov property.
According to the theory of Rozanov~\cite{Rozanov1982},    
a real-valued Gaussian random field 
satisfying~\eqref{eq:abstract-spatial-SPDE}
has such a Markov property
if and only if its
precision operator $L^*\! L$ is local, where $L^*$ denotes the adjoint of $L$.

An important example 
in spatial statistics is the choice of a 
fractional-order differential operator
$L := \tau (\kappa^2 - \Delta)^\beta$ in~\eqref{eq:abstract-spatial-SPDE},
where $\Delta$ is the Laplacian, $\cW$ is Gaussian white noise
and $\tau, \kappa, \beta \in (0,\infty)$.
Whittle~\cite{Whittle1963} observed that the covariance function of the  
stationary solution $(X(x))_{x\in \cD}$ to \eqref{eq:abstract-spatial-SPDE}  
with $\cD = \R^d$ 
then belongs to the widely used
\emph{Mat\'ern 
	covariance class}~\cite{Matern1960}. 
This observation motivated the \emph{SPDE approach}
for spatial statistical modeling proposed by
Lindgren, Rue and Lindstr\"om~\cite{LindgrenRueLindstroem2011}. Here, one considers~\eqref{eq:abstract-spatial-SPDE} 
on a bounded Euclidean
domain $\cD\subsetneq \R^d$, augmented with boundary conditions,
and approximates the resulting \emph{Whittle--Mat\'ern fields} by means of
efficient numerical methods available for (S)PDEs.
Owing to its ease of generalization and 
its computational efficiency as 
compared to covariance-based techniques, 
this approach
has  gained widespread popularity, 
see e.g.~\cite{BolinKirchnerJCGS2020,BKKBIT2018,BKKIMA2020,HKS2020,LindgrenBolinRue2022,SanzAlonsoYang2022}.
Since  in this case
the precision operator is given by 
$L^*\! L = \tau^2(\kappa^2 - \Delta)^{2\beta}\!$, 
we find that Whittle--Mat\'ern fields are Gaussian \emph{Markov} random fields 
in the sense of Rozanov~\cite{Rozanov1982}
precisely when $2\beta \in \N$.

Recently, extensions of the SPDE approach 
incorporating time dependence 
have been discussed. 
A class of space--time equations 
which has been proposed in this context is
\begin{equation}\label{eq:bla}
	(\partial_t + L)^\gamma X(t,x) = \dot \cW^Q(t,x),
	\quad 
	(t,x) \in \bbT \times \cD,
\end{equation}
where $\bbT \subseteq \bbR$ represents a time interval and
$\dot \cW^Q$ is spatiotemporal Gaussian noise, 
which is 
spatially colored by an operator $Q$, see~\cite{KW2022, Bakka2020}.
In particular,
it has been shown in~\cite{KW2022} that
equation \eqref{eq:bla} extends the Mat\'ern model in terms of spatial marginal covariance,
and that the interplay of its parameters governs
smoothness in space and time as well as
the degree of separability. 

Spatiotemporal random fields 
can be viewed as 
$U$-valued stochastic processes by 
letting a Hilbert space $U$ encode the spatial variable,
so that~\eqref{eq:bla} corresponds to
a stochastic fractional evolution equation 
of the form 
\begin{equation}\label{eq:the-SPDE-intro}
	(\partial_t + A)^\gamma X(t) = \dot W^Q(t), \quad t \in \bbT.
\end{equation}
The (temporal) Markov property of solutions to~\eqref{eq:bla} is then 
equivalent to that of the $U$-valued solution process $(X(t))_{t\in\bbT}$, 
where the Markov behavior is considered 
with respect to the index set~$\bbT$. 
Moreover, viewing~\eqref{eq:the-SPDE-intro} as a special case
of
\begin{equation}\label{eq:abstract-spatiotemporal-SPDE}
	\cL X(t) = \dot W^Q(t), 
	\quad 
	t \in \bbT,
\end{equation}
where $\cL$ is now a linear operator acting on functions from $\bbT$ to $U$,
the theory of Rozanov~\cite{Rozanov1982} suggests that 
locality of the precision operator $\cL^*\! \cL$,
also acting on functions $f \from \bbT \to U$, can be used to
characterize
temporal Markov behavior of the solution~$X$.

\subsection{Contributions}
In this work we define 
\emph{simple}, \emph{multiple} (\emph{$N$-ple} for $N \in \N$) and \emph{weak} Markov properties 
for stochastic processes which take values in a Hilbert space~$U$. 
These definitions generalize   
those appearing for instance in~\cite{HidaHitsuda1993, RevuzYor1999, Rozanov1982} 
for real-valued processes to infinite dimensions,  
see Definitions~\ref{def:simple-markov}, \ref{def:Nple-markov-Doob} and~\ref{def:weak-markov},
respectively.
Besides gathering them in once place, we establish their interrelations,
see Proposition~\ref{prop:Markov-relations} and Remark~\ref{rem:Markov-relations-weak}.
The main results are Theorems~\ref{thm:weakMarkov-Gaussian-necessity} and~\ref{thm:weakMarkov-Gaussian-sufficiency},
which give necessary and sufficient conditions, 
in terms of the precision operator $\cL^*\! \cL$, for the weakest notion of Markovianity for a $U$-valued
Gaussian process defined via~\eqref{eq:abstract-spatiotemporal-SPDE}.
These results are proven by a non-trivial extension of the theory by Rozanov~\cite[Chapters~2 and~3]{Rozanov1982} 
from the real-valued to the $U$-valued setting.

In order to consider the more concrete class of processes defined via~\eqref{eq:the-SPDE-intro},
we construct
a stochastic integral for deterministic operator-valued integrands defined on the
whole of $\bbR$ with respect to a two-sided (cylindrical)
$Q$-Wiener process $(W^Q(t))_{t\in\bbR}$, see Subsection~\ref{subsec:prelims:stochastic-int-twosided-Wiener}.
We employ this stochastic integral to
define the mild solution process $Z_\gamma = (Z_\gamma(t))_{t\in\R}$ to~\eqref{eq:the-SPDE-intro}
on $\bbT = \bbR$, see~Definition~\ref{def:mildsol}.
Our rigorous definition of the
fractional space--time operator $(\partial_t + A)^\gamma$ for $\gamma \in \R$,
see Definition~\ref{def:Weyl-fractional-parabolic-integral},
extends the Weyl fractional calculus in the sense that one recovers the Weyl fractional derivatives and integrals
defined in~\cite[Section~2.3]{KST2006} upon specializing to $U = \R$ and $A = 0$.

We show that the mild solution $Z_\gamma$ to~\eqref{eq:the-SPDE-intro}
satisfies the $N$-ple Markov property if $\gamma = N \in \N$, see~Theorem~\ref{thm:Nple-Markov-ZN}.
Conversely, we use Theorem~\ref{thm:weakMarkov-Gaussian-necessity} to show that
$Z_\gamma$ is in general not weakly Markov for $\gamma\not\in\N$. 
This complements~\cite[Theorem~2.7]{GoldysatZabczyk2016}, where the authors show that any 
time-homogeneous $U$-valued Gaussian simple Markov process is the solution to a first-order stochastic evolution equation.

Finally, we discuss another interesting aspect of
the SPDE~\eqref{eq:the-SPDE-intro}: A fractional $Q$-Wiener process $(W_H^Q(t))_{t\in\bbR}$ with Hurst parameter $H \in (0,1)$,
as defined for instance in~\cite{Duncan2002}, can be obtained
as a limiting case of~\eqref{eq:the-SPDE-intro}
with $\gamma = H + \frac{1}{2}$ and $A = \varepsilon \id_U$ as $\varepsilon \downarrow 0$, see Proposition~\ref{prop:fQWP-limit-of-Zgamma}.
The proof is based on a Mandelbrot--Van Ness~\cite{MandelbrotVanNess1968} type
integral representation of $W^Q_H$, again using the two-sided stochastic integral from Subsection~\ref{subsec:prelims:stochastic-int-twosided-Wiener}, 
see Proposition~\ref{prop:fQWP}.
The case $H = \frac{1}{2}$ corresponds to a (non-fractional)
$Q$-Wiener process and is thus Markov. Conversely, although
the results of Theorems~\ref{thm:weakMarkov-Gaussian-necessity} and~\ref{thm:weakMarkov-Gaussian-sufficiency} do not
apply directly, the above observation
provides evidence that $W^Q_H$ does not satisfy a weak Markov property 
for $H \in (0,1) \setminus \{\frac{1}{2}\}$.

\subsection{Outline}
In Section~\ref{sec:prelims} we begin by establishing the necessary notation, see Subsection~\ref{subsec:prelims:Notation},
followed by the construction of the stochastic integral with respect to a two-sided
(cylindrical)
$Q$-Wiener process in Subsection~\ref{subsec:prelims:stochastic-int-twosided-Wiener}.
Section~\ref{sec:markov-properties} is devoted to
defining, relating and 
(for solutions to~\eqref{eq:abstract-spatiotemporal-SPDE}) 
characterizing various notions of Markov behavior for $U$-valued stochastic processes.
The goal of Section~\ref{sec:frac-stoch-ACP-R} is to define and analyze
the
mild solution to~\eqref{eq:the-SPDE-intro} on $\bbT = \R$.
To this end, we first
describe the setting and define $(\partial_t + A)^\gamma$ with $\gamma \in \R$
in Subsections~\ref{subsec:fractional-int-diff:setting} and~\ref{subsec:fractional-int-diff}, respectively,
and subsequently investigate for which values of $\gamma$ the process exhibits Markov behavior in Subsection~\ref{subsec:fractional-int-diff:Markov}.
In Section~\ref{sec:frac-QWiener} we recall
the definition from~\cite{Duncan2002} of a $Q$-fractional Wiener process
and prove a Mandelbrot--Van Ness type integral representation, which allows us to exhibit it as a limiting case of~\eqref{eq:the-SPDE-intro}.

This article is supplemented by three appendices: 
Appendix~\ref{app:mixed-norm-bochner} collects some definitions and facts about 
spaces of functions with values in a Hilbert or Banach space.
It is followed by
Appendix~\ref{app:aux}, which
contains auxiliary results
relating to specific results from the main text
whose statements and proofs were postponed for readability;
subjects include
conditional independence, filtrations indexed by $\bbR$
and the mean-square differentiability of stochastic convolutions. 
Appendix~\ref{app:frac-powers} 
is a short overview of results
regarding fractional powers of linear operators and 
their relation to the fractional parabolic operator $(\partial_t + A)^\gamma$.

\section{Preliminaries}
\label{sec:prelims}
\subsection{Notation} \label{subsec:prelims:Notation}

The sets $\N := \{1,2,3,\dots\}$ and $\N_0 := \N \cup \{0\}$ 
denote the positive and non-negative integers, 
respectively.
We write $s\wedge t$ (or $s \vee t$) for the 
minimum (or maximum)  
of two real numbers $s,t\in\R$. 
Empty sums and products are assigned the values zero and one, respectively.
Given two arbitrary sets $D$ and $E$, 
the preimage of $E_0 \subseteq E$ under the function $f \from D \to E$ is defined by
${f^{-1}[E_0]  = \{ f \in E_0 \} = \{x \in D: f(x)\in E_0\}}$.
The identity on $D$ and the indicator function of $D_0 \subseteq D$
are respectively denoted by $\id_D \from D \to D$ and
$\mathbf 1_{D_0} \from D \to \{0,1\}$.

Let
$(U, \scalar{\, \cdot \, , \, \cdot \,}{U})$ and 
$(\widetilde U, \scalar{\, \cdot \, , \, \cdot \,}{\widetilde U})$ denote Hilbert spaces
with induced norms ${\norm{\cdot}{U}}$ and $\norm{\cdot}{\widetilde U}$.
Banach spaces are usually denoted by
$(E, \norm{\cdot}{E})$ and $(F, \norm{\cdot}{F})$. These
spaces are all assumed to be real and separable.

We write ${\mathscr L(E;F)}$ 
for the Banach space of bounded linear operators equipped with the usual operator norm.
The adjoint operator of $T\in \mathscr L(E;F)$ 
is denoted by ${\dual T \in \mathscr L(\dual F; \dual E)}$, where $E^*, F^*$ denote the dual spaces; 
if $T \in \LO({U}; \widetilde U)$, then
we can identify $\dual U = U$ 
and $\dual {\widetilde{U}} = \widetilde{U}$ via the Riesz 
maps, so that  
$\dual T\in \mathscr L(\widetilde U; {U})$.
We 
abbreviate 
$\mathscr L(E) := \mathscr L(E;E)$, 
and similarly for the other spaces of linear operators. 

We write $T \in \mathscr L^+(U)$
to indicate that $T \in \LO(U)$ is self-adjoint 
positive definite, meaning that $\dual T = T$ and
there exists $\theta\in(0,\infty)$ such that 
$\scalar{Tx,x}{U} \geq \theta \|x\|_U^2$ 
for all $x\in U$.
Let
$\mathscr L_2(U; \widetilde U) \subseteq \mathscr L(U; \widetilde U)$
denote the Hilbert space of all Hilbert--Schmidt operators
equipped with
the inner product 
$\langle T,S \rangle_{\mathscr L_2(U; \widetilde U)} := \sum_{j \in \N} \langle Te_j, Se_j \rangle_{\widetilde U}$, 
where
$(e_j)_{j\in\N}$ is any orthonormal basis for~$U$.  
The trace of $T \in \LO^+(U)$ is defined by
$
\tr T := \sum_{j \in \N} \scalar{Te_j, e_j}{U}.
$
We say that $T \in \LO(U)$ 
belongs to the Banach space $\mathscr L_1(U)$ of trace class operators 
if $\norm{T}{\LO_1(U)} := \tr |T| < \infty$ for $|T| := (\dual T T)^{\frac{1}{2}}$,
using the operator square root 
(see e.g.~\cite[Proposition~8.27]{Neerven2022})  
and set $\LO_1^+(U) := \LO^+(U) \cap \LO_1(U)$.

A linear operator $A$ on $E$ with domain $\mathsf D(A)$ 
is denoted by ${A\from \mathsf D(A) \subseteq E \to E}$;
for its range we write~$\mathsf R(A)$. 
We call $A$ closed if its graph 
$\mathsf G(A):= {\{(x,Ax): x \in \mathsf D(A)\}}$ 
is closed with respect to the norm 
$\norm{(x,Ax)}{\mathsf G(A)} = \norm{x}{\dom{A}} = \norm{x}{E} + \norm{Ax}{E}$, 
and densely defined if $\mathsf D(A)$ is dense in $E$. 
If the closure $\clos{\mathsf G(A)}$ is the graph of a linear operator $\clos A$, 
then we call $\clos A$ the closure of $A$.
We write $A \subseteq \widetilde A$ if $\mathsf G(A) \subseteq \mathsf G(\widetilde A)$.

Throughout this work, we assume that a complete probability space
$(\Omega,\cF,\bbP)$ is given,
meaning that $\cF$ contains the collection $\cN_{\bbP}$ of $\bbP$-null sets.
We abbreviate the phrase ``$\bbP$-almost surely'' by ``$\bbP$-a.s.''
In what follows, we call
a function $Z\from \Omega \to E$
an $E$-valued random variable if it is strongly $\bbP$-measurable, see 
Subsection~\ref{subsec:app-func-spaces:measurable-cont} in Appendix~\ref{app:mixed-norm-bochner}.
We write $Z \sim N(m, Q)$ if $Z$ is a $U$-valued Gaussian random variable
with mean $m \in U$ and covariance operator $Q \in \mathscr L^+_1(U)$;
its existence is guaranteed by~\cite[Theorem~2.3.1]{Bogachev1998}.
Two stochastic processes $(X(t))_{t \in \bbT}$ and
$(\widetilde X(t))_{t \in \bbT}$ are said to be modifications
of each other if $\bbP(X(t) = \widetilde X(t)) = 1$ for all $t \in \bbT$,
where $\bbT = [0,\infty)$, $\bbT = [0,T]$ for some $T \in (0,\infty)$ or $\bbT = \bbR$.

Let $\cG_1,\cH,\cG_2\subseteq \cF$ be sub-$\sigma$-algebras of $\cF$. 
The join of two $\sigma$-algebras is denoted by $\cG_1 \vee \cG_2 := \sigma(\cG_1 \cup \cG_2)$.
We write $\cG_1 \ind \cG_2$ to indicate that $\cG_1$ and $\cG_2$ are independent.
The expression $\E[Z \mid \cH]$ denotes the conditional expectation of a random variable $Z$ given $\cH$,
and the conditional probability of $A \in \cF$ given $\cH$ is defined by $\bbP(A \mid \cH) := \E[1_{A} \mid \cH]$, $\bbP$-a.s.
The notation $\cG_1 \ind_{\cH} \cG_2$ indicates that
$\cG_1$ and $\cG_2$ are conditionally independent given $\cH$, i.e.,
\begin{equation}
	\forall 
	G_1 \in \cG_1, \, G_2 \in \cG_2 : \quad  
	\bbP(G_1 \cap G_2 \mid \cH) 
	=
	\bbP(G_1 \mid \cH)\bbP(G_2 \mid \cH), \quad \bbP\text{-a.s.},
\end{equation}
When conditioning on
the natural $\sigma$-algebra $\sigma(Y) = \{ Y^{-1}[B] : B \in \cB(E)\}$ generated
by an $E$-valued random variable $Y$, 
we may simply write $Y$ instead of $\sigma(Y)$; e.g., $\E[Z \mid Y]$ or $\cG_1 \ind_{Y} \cG_2$.

\subsection{Stochastic integration with respect to a two-sided Wiener process}
\label{subsec:prelims:stochastic-int-twosided-Wiener}

Let $(W_{1}^Q(t))_{t\ge0}$ and $(W_{2}^Q(t))_{t\ge0}$
be independent $U$-valued standard $Q$-Wiener processes for a given $Q \in \LO_1^+(U)$,
see for instance~\cite[Section~2.1]{LiuRockner2015}, and define
\[
W^Q(t) := \begin{cases}
	W_1^Q(t), &\quad t \in [0,\infty); \\
	W_2^Q(-t), &\quad t \in (-\infty, 0).
\end{cases}
\]
Then the \emph{two-sided $Q$-Wiener process} $W^Q := (W^Q(t))_{t\in\R}$ satisfies the following:
\begin{enumerate}[label=(WP\arabic{*})]
	\item \label{item:WP1}
	$W^Q(t)$ has mean zero and
	$W^Q(t)-W^Q(s) \sim N(0, (t-s)Q)$ for $t \ge s$;
	\item $W^Q$ has continuous sample paths;
	\item  \label{item:WP3} 
	$W^Q(t_4) - W^Q(t_3) \ind W^Q(t_2) - W^Q(t_1)$ for
	$t_1 < t_2 \le t_3 < t_4$.
\end{enumerate}
One can 
define a stochastic integral with respect to such a process
using a construction analogous to the one-sided case, as presented for instance in~\cite[Section~2.3]{LiuRockner2015}.
Restricting ourselves to deterministic integrands $\Phi \from \R \to \LO(U; \widetilde U)$, 
this procedure yields a square-integrable stochastic integral
$\int_{\bbR} \Phi(t) \rd W^Q(t) \in L^2(\Omega; \widetilde U)$
which exists if and only if
$\Phi(\,\cdot\,)Q^\frac{1}{2} \in L^2(\bbR; \mathscr L_2(U; \widetilde U))$; 
see Subsection~\ref{subsec:app-func-spaces:Bochner-Sobolev} in Appendix~\ref{app:mixed-norm-bochner} 
for the definitions of these (Bochner) spaces.
In this case, it satisfies the following It\^o isometry:
\begin{equation}\label{eq:ito-R}
	\Norm{\int_{\bbR} \Phi(t) \rd W^Q(t)}{L^2(\Omega; \widetilde U)}^2 
	=
	\int_{\bbR} \norm{\Phi(t)Q^\frac{1}{2}}{\mathscr L_2(U; \widetilde U)}^2 \rd t.
\end{equation}
As in the one-sided case, we can extend the definition of the stochastic integral
to allow for an operator $Q \in \LO^+(U) \setminus \LO_1^+(U)$ 
which is not trace-class.
To this end, one defines a \emph{cylindrical $Q$-Wiener process}
$(W^Q(t))_{t\in\bbR}$
on $U$, which
can be identified with a $\widetilde Q$-Wiener process on 
a larger Hilbert space $\mathscr U \supseteq U$ 
with Hilbert--Schmidt embedding ${\iota \in \mathscr L_2(U; \mathscr U)}$,
where $\widetilde Q = \iota Q \dual\iota \in \mathscr L_1^+(\mathscr U)$.
We can then define the stochastic integral with respect to the associated
cylindrical $Q$-Wiener process, which still satisfies~\eqref{eq:ito-R}.
Such an embedding $\iota \in \mathscr L_2(U; \mathscr U)$ 
always exists, and the resulting integral does not depend
on $\iota$, see~\cite[Section~2.5]{LiuRockner2015}.

Now we turn to the matter of $\R$-indexed filtrations 
on $(\Omega, \cF, \bbP)$ associated
to $( W^Q(t) )_{t\in\R}$.
First we recall the situation in
the one-sided case. Here, the
integral process
$\bigl( \int_0^t \Phi(r) \rd W_1^Q(r) \bigr)_{t\geq 0} 
:= 
\bigl( \int_0^\infty \mathbf 1_{[0,t]}(r) \Phi(r) \rd W_1^Q(r) \bigr)_{t\geq 0}$
is adapted to the natural filtration
\[
\cF_t^{W_1^Q} 
:= 
\sigma\bigl( W_1^Q(s) : 0 \leq s \leq t \bigr) 
\vee \sigma(\cN_{\bbP}), 
\quad t \in [0,\infty),
\]
whenever $\Phi(\,\cdot\,)Q^\frac{1}{2} \in L^2(0,t; \mathscr L_2(U; \widetilde U))$ 
for all $t \in [0,\infty)$,
which is immediate from the definition of the stochastic integral.
Since property~\ref{item:WP3} for $(W_1^Q(t))_{t\geq 0}$ 
implies that $W_1^Q(t') - W_1^Q(s') \ind \cF_s^{W_1^Q}$ 
for all $0 \leq s \leq s' < t' \le t$, 
the stochastic integral $\int_s^t \Phi(r)\rd W_1^Q(r)$ 
is also independent of $\cF_s^{W_1^Q}$\!.
The combination of the previous two facts implies that
$\bigl( \int_0^t \Phi(r)\rd W_1^Q(r) \bigr)_{t\geq 0}$ 
is a martingale with respect to $\bigl(\cF^{W_1^Q}_t\bigr)_{t\geq 0}$.

In the two-sided case, we instead work with the 
\emph{(completed) filtration $(\cF_t^{\delta W^Q})_{t\in\R}$ generated by the increments of $W^Q$},
defined by
\begin{equation}\label{eq:increment-filtration}
	\cF_t^{\delta W^Q} 
	:=
	\sigma (W^Q(u) - W^Q(s) :  s < u \le t ) \vee \sigma(\cN_{\bbP}),
	\quad
	t \in \R.
\end{equation}
Note that we have $\cF^{\delta W^Q}_t \subseteq \cF^{W^Q}_t$
for all $t \in \bbR$ and $\cF^{W^Q}_t = \cF^{\delta W^Q}_t$
for $t \in [0,\infty)$, where
$\cF^{W^Q}_t$ is generated
by $(W^Q(s))_{s\in(-\infty, t]}$ for each $t \in \bbR$.
We point out that $(\cF_t^{\delta W^Q})_{t\in\R}$ is normal:
It is complete by definition and
right-continuous by Proposition~\ref{prop:filtr-rightcont}
in Appendix~\ref{app:aux}.
By property~\ref{item:WP3},
our two-sided Wiener process $(W^Q(t))_{t\in \R}$ now satisfies
\begin{equation}\label{item:WP3prime}
	\tag{WP3$'$}
	\forall s \le s' < t : \quad  
	W^Q(t) - W^Q(s') \ind \cF_s^{\delta W^Q}\!,
\end{equation}
so that by an argument, 
analogous to the one-sided case, 
$\bigl( \int_{-\infty}^t \Phi(r) \rd W^Q(r) \bigr)_{t\in\bbR}$
is a martingale with respect to $(\cF_t^{\delta W^Q})_{t\in\R}$
for every $\Phi(\,\cdot\,)Q^\frac{1}{2} \in L^2(\bbR; \mathscr L_2(U, \widetilde U))$.
Unlike $(W_1^Q(t))_{t\geq 0}$, however, 
the process $(W^Q(t))_{t\in\R}$ itself will not be a martingale
with respect to any filtration, see  Proposition~\ref{prop:WP-R-not-martingale} in Appendix~\ref{app:aux}.
We refer the reader to~\cite{BOGP2010, BOGP2014} for more details on the subject of real-valued
martingale type processes indexed 
by $\bbR$ and stochastic integration
with respect to such processes.

We finally record a sufficient condition for 
changing the order of integration between a deterministic Lebesgue integral
over a $\sigma$-finite measure space and a stochastic integral over $\R$ with respect to $(W^Q(t))_{t\in \R}$.
See~\cite{vNV2006} for a version of this result in the one-sided case with stochastic integrands.
For deterministic integrands, the two-sided generalization admits the following short proof.

\begin{theorem}[Stochastic Fubini theorem]\label{thm:stoch-fub}
	Let $(S, \mathscr A, \mu)$ be a $\sigma$-finite measure space
	and suppose $\Psi \from S \times \R \to \mathscr L(U, \widetilde U)$ 
	satisfies $\Psi Q^{\frac{1}{2}} \in L^{1,2}(S \times \R; \LO_2(U;\widetilde{U}))$
	(see Subsection~\ref{subsec:app-func-spaces:Bochner-Sobolev} of Appendix~\ref{app:mixed-norm-bochner}).
	Then
	$\Psi Q^\frac{1}{2} \in L^{2,1}(\R \times S; \LO_2(U;\widetilde{U}))$
	and
	\begin{equation}\label{eq:interchange-det-stoch-int}
		\int_{S} \int_\R \Psi(s,r) \rd W^Q(r) \rd \mu(s)
		=
		\int_\R \int_{S} \Psi(s,r) \rd \mu(s) \rd W^Q(r),
		\quad 
		\bbP\text{-a.s.}
	\end{equation}
\end{theorem}
\begin{proof}
	Let $(\Psi_n)_{n\in \N}$ be a sequence of functions
	$\Psi_n \from S \times \R \to \mathscr L(U; \widetilde U)$ 
	which are linear combinations of 
	functions of the form 
	$(\mathbf 1_B \otimes \mathbf 1_{(a, b]}) \otimes T$
	for $a < b$, $B \in \mathscr A$ with $\mu(B) < \infty$ 
	and $T \in \LO(U; \widetilde U)$ such that
	\begin{equation}\label{eq:simple-func-conv-L1L2}
		\lim_{n \to \infty} 
		\int_{S} 
		\biggl[ \int_\R \| [\Psi(s,r) - \Psi_n(s,r)] Q^\frac{1}{2}\|^2_{\mathscr L_2(U; \widetilde U)} \rd r \biggr]^{\frac{1}{2}}  
		\rd \mu(s) 
		=
		0.
	\end{equation}
	An application of the Minkowski inequality for integrals~\cite[Section~A.1]{Stein1970}
	to the function $g_n(s, r) := \| [\Psi(s,r) - \Psi_n(s,r)] Q^\frac{1}{2}\|_{\mathscr L_2(U; \widetilde U)}$
	for each $n \in \N$
	then yields
	\[
	\biggl[  
	\int_{\bbR} 
	\biggl| 
	\int_S g_n(s, r) \rd \mu(s) 
	\biggr|^2 
	\rd r \biggr]^{\frac{1}{2}}
	\leq 
	\int_{S} 
	\biggl[ 
	\int_{\bbR} 
	|g_n(s, r)|^2 \rd r \biggr]^{\frac{1}{2}} 
	\rd \mu(s),
	\]
	which together with~\eqref{eq:simple-func-conv-L1L2} shows that 
	\begin{equation}\label{eq:simple-func-conv-L2L1}
		\lim_{n \to \infty} 
		\int_{\R} 
		\biggl| \int_S \| [\Psi(s,r) - \Psi_n(s,r)] Q^\frac{1}{2}\|_{\mathscr L_2(U; \widetilde U)} \rd \mu(s) \biggr|^2  
		\rd r
		=
		0.
	\end{equation} 
	By the respective definitions
	of the (deterministic and stochastic) integrals
	involved, 
	\eqref{eq:simple-func-conv-L1L2} 
	and \eqref{eq:simple-func-conv-L2L1} 
	imply the first and last steps of
	\begin{align*} 
		&\int_{S} \int_\R \Psi(s,r) \rd W^Q(r) \rd \mu(s)
		=
		\lim_{n\to \infty}
		\int_{S} \int_\R \Psi_n(s,r) \rd W^Q(r) \rd \mu(s)
		\\
		&\:
		=
		\lim_{n\to \infty}
		\int_\R \int_{S} \Psi_n(s,r) \rd \mu(s) \rd W^Q(r) 
		=
		\int_\R \int_{S} \Psi(s,r) \rd \mu(s) \rd W^Q(r) 
		\text{ in } L^2(\Omega; \widetilde U),  
	\end{align*}  
	the second identity can be verified by direct computation
	for simple functions.
\end{proof}

\section{Markov properties for Hilbert space valued stochastic processes}
\label{sec:markov-properties}
Let $X = (X(t))_{t\in\bbT}$ be a
$U$-valued stochastic process  
indexed by
$\bbT$, see Subsection~\ref{subsec:prelims:Notation}.
Intuitively, $X$ is said to be a Markov process if,
at any instant,
its past and future states
are independent conditional on the present.
Varying the amount of information from  the present gives rise to different
Markov properties, which we will list in decreasing order of strength.

\subsection{Simple Markov property}\label{subsec:markov:simple}
The following definition
is often just referred to as the Markov property, see also~\cite[p.~77]{DaPrato2014} 
or \cite[Equation~(6.2), p.\ 81]{Doob1954}.
We use the adjective \emph{simple} to differentiate it from
the weaker notions of Markov behavior which will be given below.
\begin{definition}\label{def:simple-markov}
	An $(\cF_t)_{t\in\bbT}$-adapted $U$-valued
	stochastic process $(X(t))_{t\in\bbT}$ is said to have
	the \emph{simple Markov property} if for all 
	$s \le t$ and
	$B \in \cB(U)$, we have
	\begin{equation}\label{eq:markov-cond-prob}
		\bbP(X(t) \in B \mid \cF_s) = \bbP(X(t) \in B \mid X(s)),
		\quad 
		\bbP\text{-a.s.}
	\end{equation}
\end{definition}

Note that~\eqref{eq:markov-cond-prob} in the above definition can equivalently be replaced by
\begin{equation}\label{eq:simple-Markov-phi}
	\E[\varphi(X(t)) \mid \cF_s] = \E[\varphi(X(t)) \mid X(s)], \quad \bbP\text{-a.s.},
\end{equation}
for all $s \le t$ and $\varphi$ belonging to the Banach space $B_b(U)$
of bounded measurable functions from 
$U$ to $\R$ equipped with the norm $\norm{\varphi}{B_b(U)} := \sup_{x \in U} |f(x)|$;
this can be
shown using approximation by simple functions.

In particular, this viewpoint suggests the following characterization of the simple Markov property
by means of \emph{transition operators}, which is sometimes taken as the definition, 
for instance in~\cite[Chapter~III, Definition~1.3]{RevuzYor1999}.

\begin{proposition}\label{def:markov-TO}
	An $(\cF_t)_{t\in\bbT}$-adapted $U$-valued
	stochastic process $(X(t))_{t\in\bbT}$ 
	is simple Markov if and only if there
	exists a family $(T_{s,t})_{s \le t \in \bbT}$ 
	of linear operators on $B_b(U)$ satisfying, 
	for all $s \le t$ in $\bbT$ and $\varphi \in B_b(U)$,
	\begin{equation}\label{eq:markov-cond-prob-TO}
		\E[\varphi(X(t)) \mid \cF_s] = T_{s,t} \varphi(X(s)),
		\quad 
		\bbP\text{-a.s.}
	\end{equation}
	In this case, the \emph{transition operators} $(T_{s,t})_{s \le t \in \bbT}$ have the following properties:
	\begin{enumerate}[label=(TO\arabic{*})]
		\item $T_{s,t}\varphi(x) \ge 0$ for all $x \in U$ if $\varphi \in B_b(U)$ is non-negative,
		\label{item:TO1}
		\item $T_{s,t}\mathbf 1_U = \mathbf 1_U$,\label{item:TO2}
		\item $T_{s,u} \varphi(X(s)) = T_{s,t} T_{t,u}\varphi(X(s))$\label{item:TO3}, $\bbP$-a.s., 
		for all $\varphi \in B_b(U)$ and $s \le t \le u$.
	\end{enumerate}
\end{proposition}
\begin{proof}
	The sufficiency part of the assertion can be verified by applying 
	$\E[\; \cdot \mid X(s)]$
	on both sides of~\eqref{eq:markov-cond-prob-TO} and using
	elementary properties of conditional expectations to conclude that
	$\E[\varphi(X(t)) \mid X(s)]=T_{s,t} \varphi(X(s))$ holds 
	$\bbP$-a.s.\ for any $\varphi \in B_b(U)$.
	This implies~\eqref{eq:simple-Markov-phi}, hence $X$ has the simple Markov property.
	
	Conversely, suppose that $X$ is simple Markov. 
	Since $U$ is separable 
	and thus a Polish space, there exists a regular conditional 
	distribution of $X(t)$ given $X(s)$
	for each $s \le t$~\cite[Theorem~8.5]{Kallenberg2002},
	i.e., a mapping
	$\pi_{s,t} \from U \times \cB(U) \to [0,1]$ such that
	for any $x \in U$ and $B \in \cB(U)$,
	$\pi_{s,t}(x, \,\cdot\,)$ is a probability measure on $(U, \cB(U))$,
	$\pi_{s,t}(\,\cdot\,, B) \from U \to [0,1]$ is Borel measurable
	and 
	\[
	\pi_{s,t}(X(s), B) 
	= 
	\bbP(X(t) \in B \mid X(s))
	=
	\bbP(X(t) \in B \mid \cF_s),
	\quad
	\bbP\text{-a.s.}
	\]
	We  use this to define the transition operator
	\[
	T_{s,t} \varphi(x) 
	:=
	\int_U \varphi(y) \, \pi_{s,t}(x, \mathrm d y),
	\quad
	\varphi \in B_b(U), \; x \in U.
	\]
	Properties~\ref{item:TO1} and~\ref{item:TO2} are consequences of the fact that
	$\pi_{s,t}(x, \,\cdot\,)$ is a probability measure on $\cB(U)$ for each $x \in U$.
	Equation~\eqref{eq:markov-cond-prob-TO} 
	can be verified using approximation by simple functions as before;
	in conjunction with the tower property 
	of conditional expectations, we find~\ref{item:TO3}.
\end{proof}

Property~\ref{item:TO3} is also known as the \emph{Chapman--Kolmogorov relation}.
If $T_{s,t}$ only depends on the difference $t-s$, then $X$ is said to be
\emph{time-homogeneous} and we write $T_{t} := T_{0,t}$, 
so that~\ref{item:TO3} reduces to the semigroup law $T_{t+s} = T_t T_s$.

Lastly, we can also characterize the simple Markov property in terms of conditional independence:
By Theorem~\ref{thm:Doob-cond-ind} in Appendix~\ref{app:aux},
the simple Markov property is equivalent to
\[
\forall s \le t : \quad  
\cF_s \ind_{X(s)} \sigma(X(t)).
\]
In fact, according to \cite[Lemma~11.1]{Kallenberg2002}, this
is in turn equivalent to
\begin{equation}
	\forall s \in \bbT : \quad  
	\cF_s \ind_{X(s)} \sigma(X(t) : t \ge s).
\end{equation}
%

\subsection{Multiple Markov property}
The following weaker notion of Markov behavior
dates back to Doob, who introduced it
in the context of stationary real-valued Gaussian processes \cite[pp.~271--272]{Doob1944}.
We generalize it to square-integrable $U$-valued processes
with some mean-square differentiability, i.e., 
$(X(t))_{t\in\bbT} \subseteq L^2(\Omega; U)$
such that $t \mapsto X(t)$ is classically differentiable
from $\bbT$ to $L^2(\Omega; U)$.
In works such as~\cite{HidaHitsuda1993, Rozanov1982},
which treat the real-valued setting,
this is called the
\emph{multiple Markov property (in the restricted sense)},
where ``restricted'' refers to the required differentiability; we will omit this descriptor.

\begin{definition}\label{def:Nple-markov-Doob}
	Let $X = (X(t))_{t\in\bbT} \subseteq L^2(\Omega; U)$
	be an $(\cF_t)_{t\in\bbT}$-adapted $U$-valued
	stochastic process and suppose that $N \in \N$. 
	Then $X$ has the
	\emph{$N$-ple Markov property} if it
	has $N-1$ mean square derivatives 
	and, for all
	$s \le t$ in $\bbT$ and
	$B \in \cB(U)$,
	\begin{equation}\label{eq:Nple-markov-cond-prob}
		\bbP(X(t) \in B \mid \cF_s) = \bbP(X(t) \in B \mid X(s), X'(s), \dots, X^{(N-1)}(s)), 
		\quad 
		\bbP\text{-a.s.}
	\end{equation}
\end{definition}
Setting $\mathbf X(t) := (X^{(k)}(t))_{k=0}^{N-1}$, one defines
a process taking values in the direct product Hilbert space $(U^N, \scalar{\, \cdot \, , \, \cdot \,}{U^N}\!)$, where
the inner product
\begin{equation}
	\scalar{\mathbf x, \mathbf y}{U^N} := \sum_{j=1}^N \scalar{x_j,y_j}{U},
	\qquad
	\mathbf x = (x_j)_{j=1}^N, \, \mathbf y = (y_j)_{j=1}^N \in U^N \!,
\end{equation}
induces the product topology on the set $U^N$\!.
In particular, the Borel $\sigma$-algebra of $U^N$
satisfies $\cB(U^N) = \otimes^N \cB(U)$ by~\cite[Lemma~1.2]{Kallenberg2002}.
Theorem~\ref{thm:Doob-cond-ind}
in Appendix~\ref{app:aux} again yields an equivalent formulation of the $N$-ple Markov property
in terms of conditional independence:
\begin{equation}\label{eq:Nple-markov-cond-ind}
	\forall s \in \bbT : \quad  
	\cF_s \ind_{\mathbf X(s)}
	\sigma
	(X(t) : t \ge s).
\end{equation}
Note that $\sigma(\mathbf X(s)) \vee \cF_s = \cF_s$ since the mean-square derivatives of $X$ can be replaced
by left derivatives, see the proof of Proposition~\ref{prop:Markov-relations} below.
Arguing as in~\cite[Lemma~11.1]{Kallenberg2002}, one can show that this
is in turn equivalent to
\begin{equation}\label{eq:Nple-markov-cond-prob-vec}
	\forall s \in \bbT : \quad  
	\cF_s \ind_{\mathbf X(s)}
	\sigma
	(\mathbf X(t) : t \ge s)
	,
\end{equation}
which is nothing more than the simple Markov property
for $\mathbf X$. Thus, 
we can apply Proposition~\ref{def:markov-TO} to derive the following characterization.

\begin{corollary}\label{def:Nple-markov-TO}
	An $(\cF_t)_{t\in\bbT}$-adapted and square-integrable 
	$U$-valued stochastic process $X = (X(t))_{t\in\bbT}$
	with $N-1$ mean-square derivatives
	is $N$-ple Markov
	if and only if there
	exists a family $(T_{s,t})_{s \le t \in \bbT}$ 
	of linear operators on $B_b(U^N)$ satisfying, 
	for all $s \le t \le u$ in $\bbT$ and $\varphi \in B_b(U^N)$,
	\begin{equation}\label{eq:Nple-markov-cond-prob-TO}
		\E[\varphi(\mathbf X(t)) \mid \cF_s] = T_{s,t} \varphi(\mathbf X(s)),
		\quad 
		\bbP\text{-a.s.}
	\end{equation}
	In this case, $(T_{s,t})_{s \le t \in \bbT}$ satisfies properties~\ref{item:TO1}--\ref{item:TO3}.
\end{corollary}
%

\subsection{Weak Markov properties; relations between concepts}
We now define two Markov properties for which
the ``present'' at time $s \in \bbT$ is represented by 
information from neighborhoords around $s$.
As we will prove in Proposition~\ref{prop:Markov-relations} below, these two notions are equivalent.
They appear in, e.g.,~\cite[p.~62]{Rozanov1982} and~\cite[Equation~(5.87), p.~115]{HidaHitsuda1993}.

\begin{definition}\label{def:weak-markov}
	An $(\cF_t)_{t\in\bbT}$-adapted $U$-valued
	stochastic process $(X(t))_{t\in\bbT}$ has
	\begin{enumerate}
		\item 
		the \emph{weak Markov property} if, for every $s \in \bbT$,
		\begin{equation}\label{eq:weak-markov-cond-ind}
			\exists \delta > 0 : \forall \varepsilon \in (0,\delta) : \quad  
			\cF_s \ind_{\mathscr A_\varepsilon(s)} \sigma(X(t) : t \ge s),
		\end{equation}
		where $\mathscr A_\varepsilon(s) := \sigma(X(u) : u \in  (s-\varepsilon, s + \varepsilon) \cap \bbT)$;
		\item the \emph{$\sigma$-Markov property} if, for every $s \in \bbT$,
		\begin{equation}\label{eq:sigma-markov-cond-ind}
			\cF_s \ind_{\partial \mathscr A(s)} \sigma(X(t) : t \ge s),
		\end{equation}
		where $\partial \mathscr A(s) := \bigcap_{\varepsilon > 0} \mathscr A_\varepsilon(s)$.
	\end{enumerate}
\end{definition}

\begin{proposition}\label{prop:Markov-relations}
	Let $X = (X(t))_{t\in\bbT}$ be an $(\cF_t)_{t\in\bbT}$-adapted $U$-valued stochastic process.
	We have the following
	relations between Markov properties:
	\[
	\text{simple Markov} \implies  \text{$\sigma$-Markov} \iff \text{weak Markov}.
	\]
	If $N, M \in \N$ are such that $N \ge M$ and
	$X$ has $N - 1$ mean-square derivatives, then we moreover have
	\[
	\text{$M$-ple Markov} \implies \text{$N$-ple Markov} \implies \text{weak Markov}.
	\]
\end{proposition}
\begin{proof}
	If $X$ has the weak Markov property, then by definition
	we have the following identity for fixed $s \in \bbT$, $B_- \in \cF_s$ and $B_+ \in \sigma(X(t) : t \ge s)$:
	\begin{equation}\label{eq:def-weakmarkov}
		\bbP(B_- \mid \mathscr A_{\nicefrac{1}{n}}(s)) \bbP(B_+ \mid \mathscr A_{\nicefrac{1}{n}}(s)) 
		=
		\bbP(B_- \cap B_+ \mid \mathscr A_{\nicefrac{1}{n}}(s)),
		\quad \bbP\text{-a.s.},
	\end{equation}
	whenever $n \in \N$ is large enough.
	Note that $(\mathscr G_n)_{n\in\N} := (\mathscr A_{\nicefrac{1}{n}}(s))_{n\in\N}$ is 
	a non-increasing sequence of
	sub-$\sigma$-algebras of $\cF$, i.e., a backward filtration on $(\Omega, \cF, \bbP)$.
	Therefore, $(\bbP(B \mid \mathscr G_n))_{n\in\N}$ is a backward martingale with respect to
	$(\mathscr G_n)_{n\in\N}$ for any $B \in \cF$.
	Combined with the fact that $\bigcap_{n \in \N} \mathscr G_n = \partial \mathscr A(s)$,
	the backward martingale convergence theorem~\cite[Section~{12.7}, Theorem~4]{GrimmettStirzaker2001} implies that we may take
	the $\bbP$-a.s.\ limit as $n \to \infty$ in~\eqref{eq:def-weakmarkov} to find that $X$ is $\sigma$-Markov.
	
	Now let $N, M \in \N$ with $N \ge M$
	be such that $X$ has the $M$-ple Markov
	property and $N - 1$ mean-square derivatives.
	When considering $\sigma(X'(s))$ at $s \in \bbT$,
	we can restrict ourselves to mean-square left derivatives, i.e.,
	we consider the sequence
	\begin{equation}\label{eq:difference-quotients}
		(\Delta_n(s))_{n\in\N} := (n[X(s) - X(s-n^{-1})])_{n\in\N}
	\end{equation}
	converging to $X'(s)$ in the $L^2(\Omega; U)$-norm as $n \to \infty$.
	Consequently, there exists a subsequence $(\Delta_{n_k}(s))_{k\in\N}$ such that
	$\Delta_{n_k}(s) \to X'(s)$, $\bbP$-a.s., as $k \to \infty$.
	Since $\Delta_{n_k}(s)$ is $\cF_s$-measurable for each $k \in \N$, we conclude that
	$X'(s)$ is $\cF_s$-measurable and thus
	$\sigma(X'(s)) \subseteq \cF_s$. By induction, this extends to
	\[
	\sigma(X(s), X'(s), \dots, X^{(M-1)}(s)) \subseteq \sigma(X(s), X'(s), \dots, X^{(N-1)}(s)) \subseteq \cF_s,
	\]
	so that Lemma~\ref{lem:cond-ind-props}\ref{lem:cond-ind-props:b} yields the $N$-ple 
	Markov property as formulated in~\eqref{eq:Nple-markov-cond-ind}.
	
	It remains
	to show that the $N$-ple Markov property for $N \in \N$ 
	and the $\sigma$-Markov property imply the weak Markovianity of $X$.
	Fixing $s \in \bbT$ and $\varepsilon > 0$, we can set
	\begin{equation}\label{eq:H1tilde-H2tilde}
		\begin{aligned}
			&\cH_1' := \sigma(X(u) : u \in (s - \varepsilon, s] \cap \bbT) \subseteq \cF_s, \\
			&\cH_2' := \sigma(X(u) : u \in [s, s+\varepsilon) \cap \bbT) \subseteq \sigma( X(t) : t \in [s,\infty) \cap \bbT ), \\
			&\cH' := \cH_1' \vee \cH_2' = \mathscr A_\varepsilon(s).
		\end{aligned}
	\end{equation}
	Since $\sigma(X(s)) \subseteq \partial \mathscr A(s) \subseteq \mathscr A_\varepsilon(s)$,
	by Lemma~\ref{lem:cond-ind-props}\ref{lem:cond-ind-props:c} the simple (i.e., 1-ple) Markov or $\sigma$-Markov property
	of $X$ would imply
	\begin{equation}\label{eq:weakMarkov-conclusion}
		\cF_s \ind_{\mathscr A_\varepsilon(s)} \sigma( X(t) : t \in [s,\infty) \cap \bbT ),
	\end{equation}
	and thus the weak Markov property since $\varepsilon > 0$ was arbitrary.
	It remains to show that~\eqref{eq:weakMarkov-conclusion} also holds if $X$ is $N$-ple Markov.
	Choosing $K \in \N$ so large that $n_k > \varepsilon^{-1}$ 
	for all $k \ge K$, 
	we find that $(\Delta_{n_k}(s))_{k \ge K}$ (see~\eqref{eq:difference-quotients})
	is a sequence of 
	$\mathscr A_\varepsilon(s)$-measurable random variables
	converging $\bbP$-a.s.\ to $X'(s)$. As before, this procedure
	can be repeated inductively to yield
	\[
	\cH := \sigma(X(s), X'(s), \dots, X^{(N-1)}(s))
	\subseteq 
	\mathscr A_\varepsilon(s).
	\]
	This justifies the use of Lemma~\ref{lem:cond-ind-props}\ref{lem:cond-ind-props:c}
	to establish~\eqref{eq:weakMarkov-conclusion}
	for the remaining case, and the desired conclusion follows.
\end{proof}

\begin{remark}\label{rem:Markov-relations-weak}
	An analog to Definition~\ref{def:weak-markov} for \emph{generalized} $U$-valued stochastic processes
	$(X(\phi))_{\phi \in C_c^\infty(\bbT)}$
	is obtained by replacing $\sigma(X(u) : u \in J)$
	with the natural $\sigma$-algebra
	$\sigma(X(\phi) : \phi \in C_c^\infty(\bbT), \, \supp \phi \subseteq J)$ 
	generated by $X$ on an open set $J \subseteq \bbT$.
	Since pointwise evaluation is not meaningful for such processes, 
	there is no analog to the simple Markov property.
	Furthermore, 
	although the proof of the implication
	\[
	\text{weak Markov} \implies \text{$\sigma$-Markov}
	\]
	carries over, its converse now fails:
	The distributional derivative of white noise is 
	a generalized process which is $\sigma$-Markov but not weak Markov, see~\cite[p.~62]{Rozanov1982}.
\end{remark}

\subsection{Characterization of weakly Markov Gaussian processes}
\label{subsec:weakMarkov-gaussian}
A $U$-valued stochastic process $X = (X(t))_{t \in \bbT}$ is said to be Gaussian if,
for any $n \in \N$ and $\{t_i\}_{i=1}^n \subseteq \bbT$,
the $U^n$-valued random variable
$(X(t_1), X(t_2), \dots, X(t_n))$ is Gaussian.
For such processes, we shall characterize the weak Markov property of Definition~\ref{def:weak-markov}
by extending the theory of Rozanov~\cite{Rozanov1982}
from real-valued to $U$-valued processes.

We consider the case of a mean-square continuous Gaussian process $X$
which
is the solution of a stochastic evolution equation of the form
$\cL X = \dot W$ for some linear operator $\cL \from \dom{\cL} \subseteq L^2(\bbT; U) \to L^2(\bbT; U)$;
here, $\dot W$ denotes spatiotemporal Gaussian white noise, 
cf.~\eqref{eq:abstract-spatial-SPDE} and~\eqref{eq:the-SPDE-intro}.
More precisely, we assume that $\cL$ has a bounded inverse $\cL^{-1}$ 
which \emph{colors} $X$, meaning
\begin{equation}\label{eq:X-colored-by-Linverse}
	\langle X, \phi \rangle_{L^2(\bbT; U)} \overset{d}{=} \mathscr W(\cL^{-*} \phi), 
	\quad 
	\forall 
	\phi \in C_c^\infty(\bbT; U),
\end{equation}
where $\cL^{-*} \coloneqq (\cL^{-1})^* = (\cL^*)^{-1}$ and $\overset{d}{=}$ indicates equality in distribution.
Here, $(\mathscr W(f))_{f \in L^2(\bbT; U)} \subseteq L^2(\Omega)$ is an $L^2(\bbT; U)$-isonormal Gaussian process,
i.e., a family of mean-zero and real-valued Gaussian random variables satisfying
\begin{equation}
	\scalar{\mathscr W(f), \mathscr W(g)}{L^2(\Omega)}
	=
	\scalar{f, g}{L^2(\bbT; U)},
	\quad 
	\forall f, g \in L^2(\bbT;U). \label{eq:isonormal}
\end{equation}
The following theorem then states that the locality of the
\emph{precision operator} $\cL^* \cL$ is necessary for
$X$ to possess the weak Markov property.
\begin{theorem}\label{thm:weakMarkov-Gaussian-necessity}
	Let $\cL \from \dom{\cL} \subseteq L^2(\bbT; U) \to L^2(\bbT; U)$ be a
	boundedly invertible linear
	operator,
	and suppose that $X = (X(t))_{t\in\bbT}$ is a mean-square continuous Gaussian $U$-valued process
	colored by $\cL^{-1}$.
	Let $F$ be a dense subset of $U$ for which
	$C_c^\infty(\bbT; F) \subseteq \dom{\cL}$.
	Furthermore, suppose that
	$C_c^\infty(\bbT; F)$ and its image under $\cL$ are dense
	subsets of
	$L^2(\bbT; U)$.
	
	If $X$ has the weak Markov property from Definition~\ref{def:weak-markov} with respect
	to its natural filtration $(\cF_t^X)_{t\in\bbT}$, then
	\begin{equation}\label{eq:precision-operator-local}
		\forall J \in \mathscr I : \quad 
		\langle \cL \phi, \cL \psi \rangle_{L^2(\bbT; U)} = 0,
		\quad 
		\forall
		\phi \in C_c^\infty(J; F), \psi \in C_c^\infty(\bbT \setminus \clos J; F),
	\end{equation}
	where $\mathscr I$ denotes the set of 
	all open intervals $J \subseteq \bbT$.
\end{theorem}
\begin{proof}
	For all $J \in \mathscr I$ we 
	define a closed subspace $\mathfrak H(J)$ of $L^2(\Omega)$ by
	\begin{equation}\label{eq:def-HJ}
		\mathfrak H(J) 
		:=
		\clos{\{\langle X, \phi \rangle_{L^2(\bbT; U)} : \phi \in C_c^\infty(J; F) \}}^{L^2(\Omega)}.
	\end{equation}
	Then the family $(\mathfrak H(J))_{J \in \mathscr I}$ is 
	a \emph{Gaussian random field} in the sense of~\cite[Chapter~2, Section~3.1]{Rozanov1982},
	and we can connect it to the present setting by showing that
	\begin{equation}
		\sigma(X(t) : t \in J)
		=
		\sigma(\mathfrak H(J)).
	\end{equation}
	Indeed, we have $\sigma(\mathfrak H(J)) \subseteq \sigma(X(t) : t \in J)$
	since $\langle X, \phi \rangle_{L^2(\bbT; U)}$ is 
	measurable with respect to the latter $\sigma$-algebra for all $\phi \in C_c^\infty(J; F)$ with $\supp \phi \subseteq J$.
	In order to establish the converse inclusion, it suffices to verify the claim that $X(t)$ is $\sigma(\mathfrak H(J))$-measurable 
	for each $t \in J$. 
	Let $(e_j)_{j\in\N}$ be an orthonormal basis of $U$ 
	and write
	\begin{equation}\label{eq:series-expansion}
		X(t) = \sum_{j = 1}^\infty \scalar{X(t), e_j}{U} \, e_j
		\quad 
		\text{in } L^2(\Omega; U).
	\end{equation}
	Now we will show that $\scalar{X(t), e_j}{U}$ is $\sigma(\mathfrak H(J))$-measurable for every $j \in \N$.
	In fact, by the density of $F \subseteq U$ it suffices to consider $\langle X(t), x \rangle_U$ for $x \in F$.
	Let $(\phi_n)_{n\in \N} \subseteq C_c^\infty(J)$ be a sequence of 
	bump functions concentrating around $t$,
	i.e., we have $\lim_{n\to\infty} \int_{\bbT} f(s) \phi_n(s) \rd s = f(t)$ in $E$
	for any $f \in C(\bbT; E)$, where $E$ is an arbitrary Banach space.
	It follows from the mean-square continuity of $X$ that
	$f := \scalar{X(\,\cdot\,), x}{U} \in C(\bbT; L^2(\Omega))$, thus with $E := L^2(\Omega)$ we obtain
	\[
	\scalar{X(t), x}{U} 
	=
	\lim_{n \to \infty} \int_\bbT \scalar{X(s), x}{U} \phi_n(s) \rd s
	=
	\lim_{n \to \infty} \int_\bbT \scalar{X(s), \phi_n(s) x}{U} \rd s
	\quad 
	\text{in } L^2(\Omega).
	\]
	Passing to a $\bbP$-a.s.-convergent subsequence in the rightmost expression,
	we find that $\scalar{X(t), x}{U}$ is a limit of $\sigma(\mathfrak H(J))$-measurable 
	random variables. Thus, each summand in~\eqref{eq:series-expansion} is 
	$\sigma(\mathfrak H(J))$-measurable, and passing to a $\bbP$-a.s.-convergent subsequence of
	$(\sum_{j = 1}^N \scalar{X(t), e_j}{U} \, e_j)_{N \in \N}$ proves the claim.
	
	The theory of~\cite[Chapter~2, Section~3.1]{Rozanov1982}, 
	now implies that $X$ has the weak Markov property from Definition~\ref{def:weak-markov} if and only
	if $(\mathfrak H(J))_{J \in \mathscr I}$ is Markov in the sense of~\cite[p.~97]{Rozanov1982}.
	For a general $B \subseteq \bbT$, we define
	\begin{equation}\label{eq:mathfrakHplus}
		\mathfrak H_+(B) := \bigcap\nolimits_{\varepsilon>0} \mathfrak H(B^\varepsilon),
	\end{equation}
	where 
	$B^\varepsilon := \{t \in \bbT : \operatorname{dist}(t, B) < \varepsilon \}$ 
	denotes an open $\varepsilon$-neighborhood of $B$.
	Using the definition~\eqref{eq:mathfrakHplus} for $B \in \{\partial J, J, \bbT\setminus J\}$,
	the Markov property for $(\mathfrak H(J))_{J \in \mathscr I}$ 
	implies that~\cite[Equations~(3.14), p.~97]{Rozanov1982}
	are satisfied for every $J \in \mathscr I$:
	\begin{equation}
		\mathfrak H_+(\partial J) = \mathfrak H_+(J) \cap \mathfrak H_+(\bbT \setminus J)
		\quad \text{and} \quad 
		\mathfrak H_+(J)^\perp \perp \mathfrak H_+(\bbT \setminus J)^\perp, \label{eq:condition-2-weak-Markov-GRF}
	\end{equation}
	where we take $L^2(\Omega)$-orthogonal complements in $\mathfrak H(\bbT)$.
	
	Next we define
	$\langle X^*, \,\cdot\, \rangle \from C_c^\infty(\bbT; F) \to L^2(\Omega)$
	by 
	\begin{equation}\label{eq:dual-field}
		\langle X^*, \phi \rangle := \mathscr W(\cL \phi), 
		\quad \bbP\text{-a.s.},
		\quad \phi \in C_c^\infty(\bbT; F),
	\end{equation}
	to which we associate the spaces
	\begin{equation}\label{eq:def-HstarJ}
		\mathfrak H^*(J)
		:= 
		\clos{\{ \langle X^*, \phi \rangle : \phi \in C_c^\infty(J; F) \}}^{L^2(\Omega)}, 
		\quad 
		J \in \mathscr I.
	\end{equation}
	Then $(\langle X^*, \phi \rangle)_{\phi \in C_c^\infty(\bbT; F)}$ is dual to 
	$(\langle X, \psi \rangle_{L^2(\bbT;U)})_{\psi \in C_c^\infty(\bbT; F)}$
	in the sense that
	\begin{equation}\label{eq:dual-process}
		\begin{aligned}
			\E[\langle X, \phi \rangle_{L^2(\bbT; U)} \langle X^*, \psi \rangle]
			=
			\E[\mathscr W(\cL^{-*} \phi) \mathscr W(\cL \psi)]
			=
			\langle \phi , \psi \rangle_{L^2(\bbT; U)}
		\end{aligned}
	\end{equation}
	for $\phi, \psi \in C_c^\infty(\bbT; F)$.
	Next we will prove
	\begin{equation}
		\mathfrak H(\bbT) = \mathfrak H^*(\bbT). \label{eq:duality-condition-1}
	\end{equation}
	by showing that both of these sets are equal to
	\[
	\mathscr Z := \clos{\{\mathscr W(f) : f \in L^2(\bbT; U)\}}^{L^2(\Omega)}.
	\]
	First, we note that $\mathfrak H(\bbT)$ and $\mathfrak H^*(\bbT)$
	are clearly contained in $\mathscr Z$.
	Now let $Z \in \mathscr Z$ and $\varepsilon > 0$ be arbitrary, and choose
	$f \in L^2(\bbT; U)$ such that $\norm{\mathscr W(f) - Z}{L^2(\Omega)} < \frac{1}{3}\varepsilon$.
	Since the image of $C_c^\infty(\bbT; F)$ under $\cL$ is assumed to be dense in $L^2(\bbT; U)$, 
	we may furthermore choose $\phi \in C_c^\infty(\bbT; F)$ such that 
	$\norm{\cL \phi - f}{L^2(\Omega; U)} < \frac{2}{3}\varepsilon$.
	It follows that
	\[
	\begin{aligned}
		\norm{Z - \langle X^*, \phi\rangle }{L^2(\Omega)}
		&\le 
		\norm{Z - \mathscr W(f)}{L^2(\Omega)}
		+
		\norm{\mathscr W(\cL\phi - f)}{L^2(\Omega)}
		\\&=
		\norm{Z - \mathscr W(f)}{L^2(\Omega)}
		+
		\norm{\cL\phi - f}{L^2(\bbT; U)} < \varepsilon,
	\end{aligned}
	\]
	which shows $Z \in \mathfrak H^*(\bbT)$ since $\varepsilon > 0$ was arbitrary.
	On the other hand, 
	since $\cL$ is densely defined and has a bounded inverse, it is in particular closed,
	hence $\cL^*$ exists and is also densely defined by~\cite[Proposition~10.22]{Neerven2022}. 
	It thus follows that
	$\range{\cL^{-*}} = \range{(\cL^*)^{-1}} = \dom{\cL^*}$ is dense in $L^2(\bbT; U)$,
	so that there exists $g \in L^2(\bbT; U)$ satisfying
	$\norm{f - \cL^{-*}g}{L^2(\bbT; U)} < \frac{1}{3}\varepsilon$. Finally,
	we choose $\psi \in C_c^\infty(\bbT; F)$ such that
	$\norm{\psi - g}{L^2(\bbT; U)} < \norm{\cL^{-*}}{\LO(L^2(\bbT;U))}^{-1} \frac{1}{3} \varepsilon$ so that
	\[
	\begin{aligned}
		\norm{Z - \scalar{X, \psi}{L^2(\bbT;U)}}{L^2(\Omega)}
		&<
		\tfrac{1}{3}\varepsilon
		+
		\norm{\mathscr W(f - \cL^{-*}g)}{L^2(\Omega)}
		+
		\norm{\mathscr W(\cL^{-*}g - \cL^{-*} \psi)}{L^2(\Omega)}
		\\&=
		\tfrac{1}{3}\varepsilon
		+
		\norm{f - \cL^{-*}g}{L^2(\bbT;U)}
		+
		\norm{\cL^{-*}(g - \psi)}{L^2(\bbT;U)} < \varepsilon,
	\end{aligned}
	\]
	hence also $Z \in \mathfrak H(\bbT)$. 
	We conclude that~\eqref{eq:duality-condition-1} holds.
	
	The necessity of~\eqref{eq:precision-operator-local} for the weak Markov property
	of $X$ will follow from
	\begin{equation}\label{eq:easier-inclusion-duality}
		\mathfrak H^*(J) \subseteq \mathfrak H_+(\bbT \setminus J)^\perp, \quad \forall J \in \mathscr I.
	\end{equation}
	Indeed, if $X$ is weakly Markov, then~\eqref{eq:easier-inclusion-duality}
	would imply that the random variables
	\[
	\begin{aligned}
		\xi := \langle X^*, \phi \rangle &\in \mathfrak H^*(J) \subseteq \mathfrak H_+(\bbT\setminus J)^\perp,
		\\
		\eta := \langle X^*, \psi \rangle &\in \mathfrak H^*(\bbT \setminus J) \subseteq \mathfrak H_+(J)^\perp,
	\end{aligned}
	\]
	where $\phi \in C_c^\infty(J; F)$ and $\psi \in C_c^\infty(\bbT\setminus J; F)$,
	are orthogonal by~\eqref{eq:condition-2-weak-Markov-GRF}.
	Therefore,
	\[
	0 = \scalar{\xi, \eta}{L^2(\Omega)}
	=
	\E[\langle X^*, \phi \rangle \langle X^*, \psi \rangle]
	=
	\scalar{\cL \phi, \cL \psi }{L^2(\bbT; U)},
	\]
	which shows~\eqref{eq:precision-operator-local}.
	Note that by definition~\eqref{eq:def-HstarJ} and density, the orthogonality extends to all $\xi \in \mathfrak H^*(J)$
	and $\eta \in \mathfrak H^*(\bbT\setminus J)$.
	
	In order to verify~\eqref{eq:easier-inclusion-duality}, we again take $\xi = \langle X^*, \phi \rangle$
	with $\phi$ as above.
	By the compact support of the latter, there exists $\varepsilon > 0$ such that
	$\scalar{\phi, \psi}{L^2(\bbT; U)} = 0$ for all $\psi \in C_c^\infty((\bbT \setminus J)^\varepsilon; F)$.
	Hence, for $\eta = \scalar{X, \psi}{L^2(\bbT; U)} \in \mathfrak H((\bbT\setminus J)^\varepsilon)$ we have
	$\xi \perp \eta$
	by~\eqref{eq:dual-process}, from which we can deduce $\xi \perp \mathfrak H_+(\bbT \setminus J)$,
	and thus~\eqref{eq:easier-inclusion-duality}.
\end{proof}

In order to state and prove sufficient conditions for weak Markovianity of $X$
in terms of the locality of its precision operator, 
we first need to collect some definitions which are based on objects encountered
in the proof of Theorem~\ref{thm:weakMarkov-Gaussian-necessity}.
Namely, we will define spaces $(\cH(J))_{J \in \mathscr I}$ 
such that $\cH(\bbT)$ is
unitarily isomorphic to $\mathfrak H(\bbT)$
and there exists a dense injection $\iota \from C_c^\infty(\bbT; F) \to \cH(\bbT)$.

Associating,
to each $\eta \in \mathfrak H(\bbT)$,
a mapping $I^{-1} \eta \from C_c^\infty(\bbT; F) \to \bbR$ given by
\begin{equation}\label{eq:I-inverse}
	I^{-1} \eta (\phi) := \E[\scalar{X, \phi}{L^2(\bbT; U)} \eta],
	\quad
	\phi \in C_c^\infty(\bbT; F),
\end{equation}
sets up a linear map $I^{-1} \from \mathfrak H(\bbT) \to \cH(\bbT)$, 
where $\cH(\bbT)$ is defined as the range of $I^{-1}$.
It is also
injective since $I^{-1} \eta(\phi) = 0$ for all $\phi \in C_c^\infty(\bbT; F)$
means $\eta \perp C_c^\infty(\bbT; F)$ in $\mathfrak H(\bbT)$,
and thus $\eta = 0$ by~\eqref{eq:def-HJ}.
Equipping $\cH(\bbT)$ with the inner product 
\[
\scalar{v_1, v_2}{\cH(\bbT)}
:=
\scalar{I v_1, I v_2}{L^2(\Omega)},
\quad 
v_1, v_2 \in \mathfrak H(\bbT),
\]
renders
$I \from \cH(\bbT) \to \mathfrak H(\bbT)$
a unitary isomorphism.
For $J \in \mathscr I$ we can then define
\begin{equation}\label{eq:def-cHJ}
	\cH(J)
	:=
	\bigvee\nolimits_{\! \! \varepsilon>0}\,
	\{
	v \in \cH(\bbT) : v(\phi) = 0 \text{ for all } \phi \in C_c^\infty((\bbT\setminus J)^{\varepsilon}; F)
	\},
\end{equation}
where $\bigvee$ denotes the closed linear span.

A dense injection $\iota \from C_c^\infty(\bbT; F) \to \cH(\bbT)$  is obtained
by defining ${\iota v \from C_c^\infty(\bbT; F) \to \R}$ in the following way,
for any $v \in C_c^\infty(\bbT; F)$:
\[
\iota v(\phi)
:=
\scalar{v, \phi}{L^2(\bbT; U)},
\quad 
\phi \in C_c^\infty(\bbT; F).
\]
Indeed, we find $\iota v \in \cH(\bbT)$ 
since the duality relations~\eqref{eq:dual-process} and~\eqref{eq:duality-condition-1}
between $X$ and $X^*$
imply that
$\langle X^*, v \rangle \in \mathfrak H^*(\bbT) = \mathfrak H(\bbT)$
satisfies
\[
\iota v(\phi)
=
\E[\scalar{X, \phi}{L^2(\bbT; U)} \langle X^*, v \rangle]
=
I^{-1} \langle X^*, v \rangle(\phi),
\quad 
\forall \phi \in C_c^\infty(\bbT; F).
\]
Moreover, the injectivity follows in the same way as for $I^{-1}$.
To show density of the range, fix an arbitrary $v \in \cH(\bbT)$.
Then ${I v \in \mathfrak H(\bbT) = \mathfrak H^*(\bbT)}$ and thus
there exists a sequence
$(\psi_n)_{n\in\N} \subseteq C_c^\infty(\bbT; F)$
such that $\langle X^*, \psi_n \rangle \to Iv$ in $L^2(\Omega)$.
Consequently,
we have $\iota \psi_n = I^{-1} \langle X^*, \psi_n \rangle \to v$ in $\cH(\bbT)$.

\begin{remark}
	For centered, real-valued Gaussian random fields $(Z(x))_{x \in \cX}$
	which are indexed by a compact metric space $(\cX, d_{\cX})$ and moreover mean-square continuous,
	a unitary isomorphism can be established between the $L^2(\Omega;\bbR)$-closure
	of all linear combinations of point evaluations
	and the dual of the Cameron--Martin space for its associated Gaussian measure on 
	the space $L^2(\cX; \bbR)$, see~\cite[Lemma~4.1(iii)]{KB2022}.
	We point out its analogy to
	the unitary isomorphism
	$I \from \cH(\bbT) \to \mathfrak H(\bbT)$ defined above.
\end{remark}

\begin{theorem}\label{thm:weakMarkov-Gaussian-sufficiency}
	Let the linear operator $\cL \from \dom{\cL} \subseteq L^2(\bbT; U) \to L^2(\bbT; U)$,
	the $U$-valued process $X = (X(t))_{t\in\bbT}$ and the subset $F \subseteq U$
	be as in Theorem~\ref{thm:weakMarkov-Gaussian-necessity}.
	Recall that $\cH(\bbT)$ is the range of the 
	linear mapping $I^{-1}$ defined by~\eqref{eq:I-inverse}
	and $\cH(J)$ is given by~\eqref{eq:def-cHJ} for all $J \in \mathscr I$.
	If 
	\begin{equation}\label{eq:additional-condition}
		\cH(J) = \clos{\iota C_c^\infty(J; F)}^{\cH(\bbT)}, \quad \forall J \in \mathscr I,
	\end{equation}
	then~\eqref{eq:precision-operator-local} implies that $X$ has the weak Markov property
	from Definition~\ref{def:weak-markov}.
\end{theorem}
\begin{proof}
	We will show that, under the additional assumption~\eqref{eq:additional-condition},
	\begin{equation}\label{eq:duality-condition-2}
		\mathfrak H^*(J) = \mathfrak H_+(\bbT \setminus J)^\perp, 
		\quad 
		\forall J \in \mathscr I.
	\end{equation}
	Note that inclusion~\eqref{eq:easier-inclusion-duality} also holds without this assumption,
	see the proof of Theorem~\ref{thm:weakMarkov-Gaussian-necessity}.
	Identities~\eqref{eq:duality-condition-1} and~\eqref{eq:duality-condition-2} express that
	$(\mathfrak H(J))_{J \in \mathscr I}$ and $(\mathfrak H^*(J))_{J \in \mathscr I}$ are \emph{dual}
	in the sense of~\cite[Chapter~2, Section~3.5]{Rozanov1982}.
	In this situation, the theorem on~\cite[p.~100]{Rozanov1982} yields that
	$X$ is weakly Markov if and only if $(\mathfrak H^*(J))_{J \in \mathscr I}$
	is \emph{orthogonal}, meaning that
	\begin{equation}\label{eq:dual-field-orthogonal}
		\mathfrak H^*(J) \perp \mathfrak H^*(\bbT \setminus \clos J), \quad \forall J \in \mathscr I,
	\end{equation}
	which is equivalent to~\eqref{eq:precision-operator-local}
	by the respective definitions of $X^*$ and $(\mathfrak H^*(J))_{J \in \mathscr I}$.
	
	To verify~\eqref{eq:duality-condition-2}, note that for each $\varepsilon > 0$,
	\[
	\begin{aligned}
		\mathfrak H((\bbT\setminus J)^\varepsilon)^\perp 
		&=
		\{
		\eta \in \mathfrak H(\bbT) : \scalar{\eta, \xi}{L^2(\Omega)} = 0 \text{ for all } \xi \in \mathfrak H((\bbT\setminus J)^\varepsilon)
		\}
		\\
		&= 
		\{ \eta \in \mathfrak H(\bbT) : U^{-1}\eta(\phi) = 0 \text{ for all } \phi \in C_c^\infty((\bbT\setminus J)^\varepsilon; F) \}
		\\
		&\cong
		\{ v \in \cH(\bbT) : v(\phi) = 0 \text{ for all } \phi \in C_c^\infty((\bbT\setminus J)^\varepsilon; F) \},
	\end{aligned}
	\]
	and it follows that
	\begin{equation}\label{eq:RHKS-isomorphic-to-Hplusthing}
		\cH(J) \cong \mathfrak H_+(\bbT \setminus J)^\perp.
	\end{equation}
	On the other hand, definition~\eqref{eq:def-HstarJ} implies
	\[
	\cH(J) \cong \mathfrak H^*(J),
	\]
	so together we indeed find~\eqref{eq:duality-condition-2}.
\end{proof}
\begin{remark}
	In order for
	locality of the precision operator $\cL^*\cL$ to imply
	the weak Markovianity of $X$,
	one needs to verify
	the additional condition~\eqref{eq:additional-condition}.
	In the real-valued case, 
	two examples of sufficient conditions on $\cL$ for~\eqref{eq:additional-condition}
	to hold are \cite[Lemmas~1 and~2, pp.~108--111]{Rozanov1982},
	which are expressed in terms of boundedness of multipication and translation operators,
	respectively, with respect to the norms $\norm{\cL^{-*}\,\cdot\,}{L^2(\bbT)}$
	and/or $\norm{\cL\,\cdot\,}{L^2(\bbT)}$.
	In~\cite[Chapter~3, Section~3.2]{Rozanov1982}, these results are applied
	to a class of differential operators with sufficiently regular coefficients.
	
	Although it is expected that analogous results can be derived in the $U$-valued setting,
	this subject is out of scope for the processes considered in the remainder of this article,
	since we establish Markovianity using direct methods instead of Theorem~\ref{thm:weakMarkov-Gaussian-sufficiency}, 
	see Subsection~\ref{subsec:fractional-int-diff:Markov}.
	However, we do use Theorem~\ref{thm:weakMarkov-Gaussian-necessity} 
	to show when the process lacks Markov behavior in Subsection~\ref{subsec:not-Markov-fractional}.
\end{remark}

\section{Fractional stochastic abstract Cauchy problem on $\R$}
\label{sec:frac-stoch-ACP-R}
The aim of this section is to define a Hilbert space valued stochastic process $(Z_\gamma(t))_{t\in \R}$
which can be interpreted as a \emph{mild solution} to an equation of the form
\begin{equation}\label{eq:the-SPDE}
	(\partial_t + A)^\gamma X(t) = \dot W^Q(t), \quad t \in \R.
\end{equation}
In Section~\ref{subsec:fractional-int-diff:setting} we specify the setting in which equation~\eqref{eq:the-SPDE} will be considered.
The fractional parabolic differential operator 
$(\partial_t + A)^\gamma$ is defined in Section~\ref{subsec:fractional-int-diff},
whereas
the noise term $\dot W^Q$ in~\eqref{eq:the-SPDE} is the formal
time derivative of the two-sided $Q$-Wiener process defined in Section~\ref{subsec:prelims:stochastic-int-twosided-Wiener}.
In Section~\ref{subsec:mildsol} we combine these two notions to 
give a rigorous definition of the process $(Z_\gamma(t))_{t\in \R}$,
and we indicate its relation to the fractional $Q$-Wiener process defined in Section~\ref{sec:frac-QWiener}.
Lastly, in Section~\ref{subsec:fractional-int-diff:Markov}
we prove that $(Z_\gamma(t))_{t\in \R}$ is $N$-ple Markov if $\gamma = N \in \N$,
but does in general not satisfy the weak Markov property when $\gamma \not\in \N$.

\subsection{Setting}
\label{subsec:fractional-int-diff:setting}
The standing assumption throughout this section on the Hilbert space $U$
and the linear operator $A$ is as follows.

\begin{assumption}\label{ass:standing}
	The linear operator
	$-A \from \mathsf D(A) \subseteq U \to U$
	on the separable real Hilbert space $U$ generates an
	\emph{exponentially stable} $C_0$-semigroup $(S(t))_{t\ge0}$ of bounded
	linear operators on $U$, i.e.,
	\begin{equation}\label{eq:semigroup-est}
		\exists M_0 \in [1,\infty), w \in (0, \infty) : \forall t \in [0,\infty) : \quad
		\norm{S(t)}{\mathscr L(U)} \le M_0 e^{-wt}.
	\end{equation}
\end{assumption}
In addition, we may assume one or both of the following conditions
on the fractional power $\gamma$ and the linear operator $Q$:
\begin{assumption}\label{ass:HAQ}
	$\,$
	\begin{enumerate}[leftmargin=1cm]
		\item \label{ass:HAQ:1}
		There exist $\gamma_0 \in (\nicefrac{1}{2}, \infty)$ and $Q \in \mathscr L^+(U)$ such that
		\[
		\int_0^\infty \| t^{\gamma_0-1} S(t) Q^\frac{1}{2} \|_{\mathscr L_2(U)}^2 \rd t < \infty.
		\]
		\item \label{ass:HAQ:2}
		The $C_0$-semigroup
		$(S(t))_{t\ge0}$ is analytic.
	\end{enumerate}
\end{assumption}

For a general account of the theory of $C_0$-semigroups, the reader
is referred to~\cite{Engel1999, Pazy1983}. 
Note that the results in these works, while stated for
complex Hilbert spaces, can be applied to the real setting by employing \emph{complexifications}
of (linear operators on) $U$, see e.g.\ Subsection~B.2.1 of~\cite[Appendix~B]{KW2022}.
In particular, by Assumption~\ref{ass:HAQ}\ref{ass:HAQ:2} we mean that the complexification of $(S(t))_{t\ge0}$
is analytic in the sense of~\cite[Section~2.5]{Pazy1983}.
We will not emphasize this matter in what follows.

Note that
$(\nicefrac{1}{2}, \infty)$ is the maximal range on which Assumption~\ref{ass:HAQ}\ref{ass:HAQ:1} can hold.
Moreover, if Assumption~\ref{ass:HAQ}\ref{ass:HAQ:1} holds for some $\gamma_0 \in (\nicefrac{1}{2}, \infty)$ then
the same is true for all $\gamma' \in [\gamma_0, \infty)$.
These facts are proven in Subsection~\ref{app:aux:facts-about-assumption} of Appendix~\ref{app:aux}.

Under Assumption~\ref{ass:HAQ}\ref{ass:HAQ:2},
we have the classical derivative 
$\frac{\mathrm d}{\mathrm d t}A^j S(t) = -A^{j+1}S(t)$
from $(0,\infty)$ to $\mathscr L(U)$ for all $j \in \N$;
moreover, by~\cite[Chapter~2, Theorem~6.13(c)]{Pazy1983},
\begin{equation}\label{eq:analytic-semigroup-est}
	\exists M_j \in [1,\infty) \colon \: \forall t \in (0,\infty) : \quad
	\norm{A^j S(t)}{\mathscr L(U)} \le M_j t^{-j} e^{-wt}.
\end{equation}

In what follows, we will simply write $(\cF_t)_{t \in \bbR} := (\cF^{\delta W^Q}_{t})_{t\in\bbR}$
for the increment filtration of $W^Q$ defined by~\eqref{eq:increment-filtration}.

\subsection{Fractional parabolic calculus and the deterministic problem}
\label{subsec:fractional-int-diff}
In this section we first consider the following deterministic counterpart to equation~\eqref{eq:the-SPDE}:
\begin{equation}\label{eq:the-PDE}
	(\partial_t + A)^\gamma u(t) = f(t), \quad t \in \R,
\end{equation}
where $f \in L^2(\R; U)$.
In order to define its mild solution,
we introduce the operations of fractional parabolic integration and differentiation,
generalizing the scalar-valued setting with $A = 0$ which is treated in~\cite[Chapter~2]{KST2006}.

\begin{definition}\label{def:Weyl-fractional-parabolic-integral}
	Let Assumption~\ref{ass:standing} hold.
	Given a function $f \from \R \to U$, we define its \emph{Weyl type fractional parabolic integral} 
	$\mathfrak I^\gamma f \from \R \to U$ of order $\gamma \in (0,\infty)$
	by
	\begin{equation}\label{eq:def-Igamma}
		\mathfrak I^\gamma f(t) 
		:= \frac{1}{\Gamma(\gamma)} \int_{-\infty}^t (t-s)^{\gamma-1} S(t-s) f(s) \rd s
		=
		\frac{1}{\Gamma(\gamma)} \int_{0}^\infty r^{\gamma-1} S(r) f(t-r) \rd r
	\end{equation}
	if these Bochner integrals,
	which are equivalent by the change of variables ${r := t - s}$,
	exist for almost all $t \in \R$.
\end{definition}

Viewing $\mathfrak I^\gamma$ as a linear operator, it turns out that
\[
\forall p \in (1,\infty) : \quad 
\mathfrak I^\gamma \in \LO(L^p(\R;U)) 
\quad \text{with} \quad
\norm{\mathfrak I^\gamma}{\mathscr L(L^p(\R;U))}
\le
\frac{M_0}{w^{\gamma }},
\]
which follows by combining estimate~\eqref{eq:semigroup-est} with Minkowski's integral inequality~\cite[Section~{A.1}]{Stein1970}. 
Setting $\mathfrak I^0 := \id_U$,
the family $(\mathfrak I^\gamma)_{\gamma\ge0}$ is a semigroup of bounded operators
on $L^2(\R; U)$. Indeed,
the semigroup law
\begin{equation}\label{eq:semigroup-law-I}
	\forall \gamma_1, \gamma_2 \in [0,\infty)\colon  \quad 
	\mathfrak I^{\gamma_1 + \gamma_2} = \mathfrak I^{\gamma_1} \mathfrak I^{\gamma_2},
\end{equation}
follows from an argument involving Fubini's theorem and \cite[Equation~(5.12.1)]{Olver2010}.
The adjoint operator $\mathfrak I^{\gamma*}$ of $\mathfrak I^{\gamma}$ satisfies
the following formula for all $t \in \bbR$:
\begin{equation}\label{eq:Igamma-adjoint}
	\mathfrak I^{\gamma*} f(t) 
	= \frac{1}{\Gamma(\gamma)} \int_t^\infty (s-t)^{\gamma-1} [S(s-t)]^* f(s) \rd s.
\end{equation}

Given a linear operator $T \from \mathsf D(T) \subseteq U \to U$
and a measure space $(S, \mathscr A, \mu)$, let us define
$\cT_S \from \mathsf D(\cT_S) \subseteq L^2(S; U) \to L^2(S; U)$ by
\begin{equation}\label{eq:Bochner-counterpart}
	\mathsf D(\cT_S) = L^2(S; \mathsf D(T)) \quad \text{and} \quad
	[\cT_S f](s) := Tf(s), \; \text{a.a. } s\in S, \;
	f \in \mathsf D(\cT_S).
\end{equation}
Using this notion, we have
\begin{equation}
	(\partial_t + \cA_\bbR)f = \partial_t f + \cA_\bbR f, 
	\quad f \in \dom{\partial_t + \cA_\bbR} = H^1(\bbR; U) \cap L^2(\R; \dom{A}),
\end{equation}
where
$\partial_t$ denotes the Bochner--Sobolev weak derivative 
from Subsection~\ref{subsec:app-func-spaces:Bochner-Sobolev} in Appendix~\ref{app:mixed-norm-bochner} with 
$\mathsf D(\partial_t) = H^1(\R; U) \subset L^2(\R; U)$.

Now define the
\emph{Weyl type fractional parabolic derivative} of order $\gamma \in [0,\infty)$ 
as the linear operator
$
\mathfrak D^\gamma := (\partial_t + \cA_\R)^{\lceil \gamma \rceil} \mathfrak I^{\lceil \gamma \rceil - \gamma}
$
on $L^2(\R; U)$ on its natural domain
\[
\dom{\mathfrak D^\gamma}
=
\Bigl\{
f \in L^2(\R; U)
:
\mathfrak I^{\lceil \gamma \rceil - \gamma} f
\in 
\mathsf D \bigl({(\partial_t + \cA_\bbR)^{\lceil \gamma \rceil}}\bigr)
\Bigr\}.
\]
The operator $\mathfrak D^\gamma$
is to be interpreted as $(\partial_t + A)^\gamma$ in~\eqref{eq:the-PDE}.
This raises the question of how this definition relates
to the concept of fractional powers, 
such as the one defined for $C_0$-semigroup generators in Appendix~\ref{app:frac-powers}.
In general, we cannot define fractional powers
of the operator $\partial_t + \cA_\R$ itself since it may fail to be closed,
which would exclude the applicability of both Definition~\ref{def:frac-power-Hille-Phillips}
(since generators of $C_0$-semigroups are closed~\cite[Chapter~1, Corollary~2.5]{Pazy1983}) 
and more general definitions
(see the introduction of~\cite{MartinezCarracedo2001}).
In fact, by the operator sum approach to maximal $L^p$-regularity, see~\cite[Discussion~1.18]{KunstmannWeis2004},
$\partial_t + \cA_\bbR$ being closed is equivalent $A$ having maximal $L^p$-regularity.
In the Hilbert space setting, this is equivalent to $-A$ generating a bounded analytic semigroup
(i.e., Assumptions~\ref{ass:standing} and~\ref{ass:HAQ}\ref{ass:HAQ:2}), see~\cite[Corollary~1.7]{KunstmannWeis2004}.

However, Proposition~\ref{prop:parabolic-operator-semigroup} shows that
the closure $\cB := \clos{\partial_t + \cA_\R}$ exists and does admit
fractional powers $\cB^\gamma$ in the sense of Definition~\ref{def:frac-power-Hille-Phillips}
for all $\gamma \in \R$. In particular, $\cB$ is boundedly invertible and in fact
we have $\cB^{-\gamma} = \mathfrak I^\gamma$ for $\gamma \in [0,\infty)$.
By~\cite[Proposition~3.2.1(b)]{Haase2006}, 
it holds that $\cB^{\gamma_1} \cB^{\gamma_2} \subseteq \cB^{\gamma_1 + \gamma_2}$
for all $\gamma_1, \gamma_2 \in \R$ and $\dom{\cB^{\gamma_1} \cB^{\gamma_2}} = \dom{\cB^{\gamma_2}} \cap \dom{\cB^{\gamma_1 + \gamma_2}}$.
We thus recover~\eqref{eq:semigroup-law-I} and find
\[
\forall \gamma \in [0,\infty) : \quad 
\mathfrak D^\gamma 
=
(\partial_t + \cA_\R)^{\lceil \gamma \rceil} \mathfrak I^{\lceil \gamma \rceil - \gamma}
\subseteq 
\cB^{\lceil \gamma \rceil} \mathfrak I^{\lceil \gamma \rceil - \gamma}
=
\cB^{\lceil \gamma \rceil} \cB^{\gamma - \lceil \gamma \rceil}
=
\cB^\gamma.
\]
It follows that $\mathfrak D^\gamma$ and
$\mathfrak I^\gamma$ are inverses of each other 
in the sense that $\mathfrak D^\gamma \mathfrak I^\gamma f = f$ 
and $\mathfrak I^\gamma \mathfrak D^\gamma f = f$ 
for all $f \in L^2(\R; U)$ such that the respective left-hand sides exist.
Therefore, $u := \mathfrak I^\gamma f$
is the unique solution to the equation $\mathfrak D^\gamma u = f$
whenever it belongs to $\dom{\mathfrak D^\gamma}$ (which may fail if $\partial_t + \cA_\bbR$ is not closed).
Relaxing this assumption, we shall call $u \in C(\R; U)$
a \emph{mild solution} to~\eqref{eq:the-PDE} if it is given by the formula
\begin{equation}
	\forall t \in \R : \quad 
	u(t) = \mathfrak I^\gamma f(t) = \frac{1}{\Gamma(\gamma)} \int_{-\infty}^t (t-s)^{\gamma-1} S(t-s) f(s) \rd s.
\end{equation}
%

\subsection{Mild solution process}
\label{subsec:mildsol}
Combining the spatiotemporal fractional integration theory from Subsection~\ref{subsec:fractional-int-diff}
with the stochastic integral defined in Subsection~\ref{subsec:prelims:stochastic-int-twosided-Wiener},
we can give a rigorous definition for the mild solution to~\eqref{eq:the-SPDE}.
We first introduce the notion of predictability for a stochastic process indexed by $\bbR$.

\begin{definition}\label{def:predictable}
	An $(\cF_t)_{t\in\bbR}$-adapted $U$-valued process $(X(t))_{t\in\bbR}$ is \emph{predictable}
	with respect to $(\cF_t)_{t\in\bbR}$
	if the mapping $(t, \omega) \mapsto X(t,\omega)$
	is strongly measurable with respect to the \emph{predictable $\sigma$-algebra}
	$
	\cP_{\bbR \times \Omega} := \sigma((s, t] \times F_s : s < t, \, F_s \in \cF_s) \vee \sigma(\cN_\bbP).
	$
\end{definition}

\begin{definition}\label{def:mildsol}
	Let Assumption~\ref{ass:standing} be satisfied and
	let $\gamma \in (\nicefrac{1}{2}, \infty)$ be such that 
	Assumption~\ref{ass:HAQ}\ref{ass:HAQ:1} holds with $\gamma_0 = \gamma$.
	Define the process
	$Z_\gamma = (Z_\gamma(t))_{t\in\bbR}$ by
	\begin{equation}\label{eq:def-minusinfty-Zgamma}
		Z_\gamma(t) := \frac{1}{\Gamma(\gamma)}\int_{-\infty}^t (t-s)^{\gamma-1} S(t-s) \rd W^Q(s), \quad t \in \bbR.
	\end{equation}
	A $(\cF_t^{\delta W^Q})_{t\in \R}$-predictable modification of $Z_\gamma$ is said to be a \emph{mild solution}
	to~\eqref{eq:the-SPDE}.
\end{definition}
The first part of the following result 
asserts that $Z_\gamma$ is mean-square continuous.
Since $Z_\gamma$ is adapted to $\bigl(\cF_t^{\delta W^Q}\bigr)_{t\in\bbR}$,
there does indeed exist a modification of $Z_\gamma$ which is predictable with respect
to this filtration, cf.\ \cite[Proposition~3.7(ii)]{DaPrato2014}, the proof of which
can be generalized to unbounded index sets.

\begin{proposition}\label{prop:Zgamma-diffbty}
	Let Assumption~\textup{\ref{ass:standing}} be satisfied.
	Suppose that $t_0 \in [-\infty, \infty)$, $\gamma \in (\nicefrac{1}{2}, \infty)$ 
	are given.
	Define $\bbT := [t_0, \infty)$ if $t_0 \in \R$ or $\bbT := \R$ if $t_0 = -\infty$
	and let
	Assumption~\textup{\ref{ass:HAQ}\ref{ass:HAQ:1}} hold
	for $\gamma_0 = \gamma$. Then the process $(Z_{\gamma}(t \mid t_0))_{t \in \bbT}$ defined by
	\begin{equation}\label{eq:Zgamma-t-t0}
		Z_{\gamma}(t \mid t_0) := \frac{1}{\Gamma(\gamma)}\int_{t_0}^t (t-s)^{\gamma-1} S(t-s) \rd W^Q(s), 
		\quad 
		t \in \bbT,
	\end{equation}
	where $Z_\gamma(\,\cdot\, \mid -\infty) := Z_\gamma$, is mean-square continuous on $\bbT$.
	
	If in addition Assumption~\textup{\ref{ass:HAQ}\ref{ass:HAQ:2}} is satisfied and there exists $N \in \N$ 
	such that Assumptions~\textup{\ref{ass:HAQ}\ref{ass:HAQ:1}} holds
	for $\gamma_0 = \gamma - N$, 
	then $(Z_{\gamma}(t \mid t_0))_{t\in \bbT}$ has $N$ mean square derivatives and
	\begin{equation}\label{eq:recurrence-Zint}
		\forall t \in [t_0, \infty), \, n \in \{0,\dots,N-1\} \colon \: 
		\left(\tfrac{\mathrm d}{\mathrm dt} + A\right)
		\tfrac{\mathrm d^n}{\mathrm dt^n}
		Z_{\gamma}(t \mid t_0) 
		= 
		\tfrac{\mathrm d^n}{\mathrm dt^n}
		Z_{\gamma-1}(t \mid t_0).
	\end{equation}
\end{proposition}

\begin{proof}
	The mean-square continuity follows directly
	by Lemma~\ref{lem:stoch-conv-Psi-diffb-under-int} in Appendix~\ref{app:aux},
	so we turn to the mean-square differentiability.
	We introduce the
	auxiliary processes
	\[
	Z_{\beta,j}(t) 
	:= 
	\frac{1}{\Gamma(\beta)}\int_{t_0}^t (t-s)^{\beta-1} A^j S(t-s)\rd W^Q(s),
	\quad
	t \in [t_0,\infty),
	\]
	for $j \in \N_0$ and $\beta \in (\nicefrac{1}{2}, \infty)$
	such the right-hand side exists.
	
	We claim
	that, under Assumptions~\textup{\ref{ass:HAQ}\ref{ass:HAQ:1}--\ref{ass:HAQ:2}}
	with
	$\gamma_0 = \gamma - N$,
	the function
	${t \mapsto t^{\beta - 1} A^j S(t)Q^\frac{1}{2}}$
	belongs to 
	$H^1_{0,\{0\}}(0, \infty; \mathscr L_2(U))$
	if
	$\beta - j - \gamma + N \in [1,\infty)$.
	To this end, we first note that the product rule for the (classical) derivative yields
	\[
	\tfrac{\mathrm d}{\mathrm dt}t^{\beta - 1} A^j S(t)Q^\frac{1}{2}
	=
	(\beta - 1) t^{\beta-2} A^j S(t)Q^\frac{1}{2}
	-
	t^{\beta - 1} A^{j+1} S(t)Q^\frac{1}{2}
	\]
	with values in $\LO(U)$ for all $t \in (0,\infty)$.
	Combining~\eqref{eq:analytic-semigroup-est} with an argument involving a change of variables and the semigroup property
	(cf.~the proof of Lemma~\ref{lem:intb-assumption-nestedness} in Appendix~\ref{app:aux}),
	one can show that the $L^2(0,\infty; \mathscr L_2(U))$-norms of
	the two functions on the right-hand side
	can be estimated by that of $t \mapsto t^{\beta-j-2}S(t)Q^\frac{1}{2}$,
	which is finite since $\beta - j - 1 \ge \gamma_0$.
	Again by~\eqref{eq:analytic-semigroup-est}, we have
	\[
	\norm{t^{\beta - 1} A^j S(t)Q^\frac{1}{2}}{\LO(U)} \le M_j t^{\beta - j - 1} \norm{Q^\frac{1}{2}}{\LO(U)} \to 0
	\quad 
	\text{as } t \downarrow 0
	\]
	since $\beta - j - 1 \ge \gamma_0>0$. Noting that $\LO_2(U) \hookrightarrow \LO(U)$ 
	and using~\cite[Lemma~A.9]{KW2022} then proves the claim.
	
	Thus, we may apply Lemma~\ref{lem:stoch-conv-Psi-diffb-under-int} from Appendix~\ref{app:aux},
	write the result as two separate integrals,
	and pull the
	closed operator $A$ out of the stochastic integral defining
	$Z_{\beta, j+1}$ (cf.~\cite[Proposition~4.30]{DaPrato2014}) to find
	\begin{align}
		Z_{\beta,j}'(t)
		&= 
		Z_{\beta - 1, j}(t) - 
		Z_{\beta, j+1}(t). \label{eq:recurrence-Zint-N1j1} 
		\\
		&= 
		Z_{\beta-1, j}(t)
		-
		A Z_{\beta, j}(t). \label{eq:AZint}
	\end{align}
	Rearranging equation~\eqref{eq:AZint} for $\beta = \gamma$ and $j = 0$
	implies~\eqref{eq:recurrence-Zint} in the case $n = 0$.
	Applying~\eqref{eq:recurrence-Zint-N1j1} iteratively, we find
	that $Z_{\beta,j}$ has the $n$th mean-square derivative
	\begin{equation}\label{eq:mean-square-deriv-Zbetaj}
		Z_{\beta,j}^{(n)}(t) 
		= 
		\sum_{m=0}^n (-1)^m {n \choose m} Z_{\beta-n+m, j + m}(t),
	\end{equation}
	provided that $\beta - j - \gamma + N \in [n,\infty)$. 
	Now we again let $\beta = \gamma$ and $j = 0$
	and apply \eqref{eq:AZint}
	with
	$\beta' = \gamma - n + m$ and
	$j' = m$ to each term on the right-hand side to derive \eqref{eq:recurrence-Zint} for the remaining values of $n$.
\end{proof}

The statistical relevance of the process $Z_\gamma$ is motivated by the following fact,
which demonstrates a situation in which it has a marginal
temporal covariance
structure of Mat\'ern type. Let $(Q_{Z_\gamma}(s,t))_{s,t\in \R} \subseteq \LO_1^+(U)$ denote
the covariance operators of $Z_\gamma$,
defined in general via the relation
\begin{equation}
	\label{eq:cov-ops}
	\scalar{ Q_{Z_\gamma}(s,t) x, y }{U}  
	= 
	\E\left[ \Scalar{Z_\gamma(s)-\E \bigl[Z_\gamma(s)\bigr], x}{U} \Scalar{Z_\gamma(t)-\E[Z_\gamma(t)], y}{U} \right]
\end{equation}
for all $s,t \in \R$ and $x,y\in U$.
Note that $\E \bigl[Z_\gamma(\,\cdot\,)\bigr] \equiv 0$ in this case.

\begin{proposition}\label{prop:separable-cov-matern}
	Let $\gamma \in (\nicefrac{1}{2}, \infty)$, $A := \kappa \id_{U}$ with ${\kappa \in (0,\infty)}$ and
	suppose that
	Assumption~\textup{\ref{ass:HAQ}\ref{ass:HAQ:1}} is satisfied for $\gamma_0 = \gamma$.
	Then the covariance of $Z_\gamma$ is 
	separable and its temporal part 
	is of Matérn type, 
	i.e.,
	\begin{equation}
		\forall s, t \in \bbR, \, s \neq t : \quad
		Q_{Z_\gamma}(s,t) 
		=
		\frac{2^{\frac{1}{2}-\gamma}\kappa^{1-2\gamma}}
		{\sqrt{\pi}\,\Gamma(\gamma)}
		(\kappa |t-s|)^{\gamma-\frac{1}{2}} 
		K_{\gamma-\frac{1}{2}}(\kappa |t-s|)
		\, Q.
		\label{eq:temporal-asymp-matern-Q}  
	\end{equation}
\end{proposition}
\begin{proof}
	For $A = \kappa \id_{U}$, Assumption~\ref{ass:standing} is trivially satisfied and
	the definition of $Z_\gamma$ takes on the following form for all $t \in \R$:
	\begin{equation}
		Z_\gamma(t) 
		= 
		\frac{1}{\Gamma(\gamma)} \int_{-\infty}^t (t-r)^{\gamma-1} e^{-\kappa(t-r)} \rd W^Q(r)
		=
		\int_\R k_{\gamma, \kappa}(t-r) \rd W^Q(r),
	\end{equation}
	with real-valued convolution kernel 
	$k_{\gamma, \kappa}(t) := \frac{1}{\Gamma(\gamma)}t^{\gamma-1}_+ e^{-\kappa t}$,
	where $t_+^{\gamma-1} := t^{\gamma - 1}$ if $t \in [0,\infty)$ and $t_+^{\gamma-1} := 0$ otherwise.
	Define the linear operator $\widetilde k(s, r; x) \in \LO(U; \bbR)$
	for $s, r \in \R$ and $x \in U$
	by 
	\[
	\widetilde k(s, r; x) h := k_{\gamma,\kappa}(s - r) \scalar{x, h}{U}, \quad h \in U.
	\]
	Then combining the It\^o isometry~\eqref{eq:ito-R} and the polarization identity yields
	\[ 
	\begin{aligned}
		&\E[ \scalar{Z_\gamma(s), x}{U} \, \scalar{Z_\gamma(t), y}{U} ]
		=
		\E\biggl[\int_\bbR \widetilde k(s,r; x) \rd W^Q(r) \int_\bbR \widetilde k(t,r; y) \rd W^Q(r) \biggr]
		\\
		&=
		\int_\bbR \scalar{\widetilde k(s,r; x) Q^\frac{1}{2},  \widetilde k(t,r;y)Q^\frac{1}{2}}{\LO_2(U; \R)} \rd r
		=
		\biggl\langle 
		\int_\bbR k_{\gamma,\kappa}(s-r) k_{\gamma,\kappa}(t-r) \rd r \, Qx, y\biggr\rangle_U. 
	\end{aligned}
	\] 
	Since $x, y \in U$ were arbitrary, 
	we find
	for all $h \in \R \setminus \{0\}$ the covariance operators
	\begin{equation}
		Q_{Z_\gamma}(t + h, t)
		=
		Q_{Z_\gamma}(t, t + h) = \int_\R k_{\gamma, \kappa}(t-r) k_{\gamma,\kappa}(t + h-r) \rd r \, Q.
	\end{equation}
	Using the change of variables $u(r) := h + 2(t-r)$ in the integral, we obtain
	\begin{equation}
		\begin{aligned}
			\int_\R &k_{\gamma, \kappa}(t-r) k_{\gamma,\kappa}(t + h-r) \rd r
			=
			\frac{1}{2}
			\int_\R k_{\gamma, \kappa}(\tfrac{u-h}{2}) k_{\gamma,\kappa}(\tfrac{u+h}{2}) \rd u
			\\
			&=
			\frac{2^{1-2\gamma}}{[\Gamma(\gamma)]^2}
			\int_{|h|}^\infty (u^2-h^2)^{\gamma-1} e^{-\kappa u} \rd u
			=
			\frac{2^{\frac{1}{2}-\gamma}\kappa^{1-2\gamma}}
			{\sqrt{\pi}\,\Gamma(\gamma)}
			(\kappa |h|)^{\gamma-\frac{1}{2}} 
			K_{\gamma-\frac{1}{2}}(\kappa |h|),
		\end{aligned}
	\end{equation}
	where the last identity follows by~\cite[Part~I, Equation~(3.13)]{OberhettingerBadii1973}.
\end{proof}

\subsection{Markov behavior}
\label{subsec:fractional-int-diff:Markov}
In this section we consider the Markov behavior of the process $Z_\gamma$ defined in Section~\ref{subsec:mildsol}.
Namely, we will show that $Z_N$ is $N$-ple Markov for $N \in \N$, 
whereas $Z_\gamma$ is in general not weak Markov if $\gamma\not\in \N$, see Theorem~\ref{thm:Nple-Markov-ZN} and Example~\ref{ex:non-Markovianity}, respectively.

\subsubsection{Integer case; main results}
\label{subsubsec:Markov-results}
We first introduce the necessary notation 
and intermediate results leading up
to the main theorem asserting the $N$-ple Markov property of $Z_N$.
The proofs are postponed to Subsection~\ref{subsubsec:Markov-proofs}.

If $\gamma \in (\nicefrac{1}{2}, \infty)$ is such that 
Assumptions~\ref{ass:standing} and~\ref{ass:HAQ}\ref{ass:HAQ:1} hold with $\gamma_0 = \gamma$,
then we define for $t_0 \in \R$ the 
truncated integral process $(\widetilde Z_\gamma(t \mid t_0))_{t\in\R}$ by 
\begin{equation}\label{eq:def-truncated-process}
	\widetilde Z_\gamma(t \mid t_0) 
	:= 
	\frac{1}{\Gamma(\gamma)}
	\int_{-\infty}^{t \wedge t_0} (t-s)^{\gamma-1} S(t-s) \rd W^Q(s),
\end{equation}
so that $\widetilde Z_\gamma(\,\cdot \mid t_0) = Z_\gamma$ on $(-\infty, t_0]$
and
$Z_\gamma = \widetilde Z_\gamma(\,\cdot \mid t_0)
+
Z_\gamma(\,\cdot\, \mid t_0)$ on $(t_0, \infty)$,
where we recall the process $(Z_{\gamma}(t \mid t_0))_{t\in[t_0,\infty)}$ from~\eqref{eq:Zgamma-t-t0}.
It immediately follows that $\widetilde Z_\gamma(t \mid t_0)$ has the same temporal regularity
at $t \in \bbR\setminus\{t_0\}$
as $Z_\gamma(t)$ (and ${Z_\gamma(t \mid t_0)}$ if $t \in (t_0, \infty)$).
In the case $\gamma = N \in \N$,
both have $N - 1$ mean-square derivatives 
by Proposition~\ref{prop:Zgamma-diffbty} if Assumption~\ref{ass:HAQ}\ref{ass:HAQ:1}
holds for $\gamma_0 = \gamma - (N - 1) = 1$.
The same holds at the critical point $t = t_0$ since
the first $N - 1$ mean-square (right) derivatives of $Z_\gamma(\,\cdot \mid t_0)$ vanish there, 
see~\eqref{eq:mean-square-deriv-Zbetaj} in the proof of Proposition~\ref{prop:Zgamma-diffbty}.

Under Assumption~\ref{ass:HAQ}\ref{ass:HAQ:2}, we have
$A^j S(t) \in \mathscr L(U)$ 
for all $j \in \N_0$ and $t \in (0,\infty)$, see~\eqref{eq:analytic-semigroup-est}.
Therefore, we can define the function
$\overline\Gamma(n, (\,\cdot\,) A) \from [0,\infty) \to \mathscr L(U)$ by
\begin{equation}\label{eq:def-incgamma}
	\begin{aligned}
		\overline\Gamma(n, tA) 
		:= 
		\begin{cases}
			\sum_{j=0}^{n-1} \frac{t^j}{j!} A^j S(t), \quad &t \in (0,\infty); \\
			\id_U, \quad &t = 0,
		\end{cases}
	\end{aligned}
\end{equation}
for $n \in \N$.
We point out the analogy with integer-order scalar-valued normalized
upper incomplete gamma functions \cite[Equations~(8.4.10) and~(8.4.11)]{Olver2010}.
We will use these functions to derive an expression
for $(\widetilde Z_N(t \mid t_0))_{t\in[t_0, \infty)}$
in terms of $Z_N$ and its mean-square derivatives at $t_0$.
Recall that $\mathbf Z_N = (Z_N^{(n)})_{n=0}^{N-1}$ indicates
the $U^N$-valued process consisting of $Z_N$ and its first $N-1$ mean-square derivatives.
\begin{proposition}\label{prop:relation-ZN-initialvalue}
	Let Assumptions~\textup{\ref{ass:standing} and~\ref{ass:HAQ}\ref{ass:HAQ:2}} be satisfied
	and suppose that Assumption~\textup{\ref{ass:HAQ}\ref{ass:HAQ:1}} holds for $\gamma_0 = 1$.
	Then for all $N \in \N$, $t_0 \in \bbR$ and $t \in [t_0, \infty)$,
	\begin{equation}\label{eq:relation-ZN-minusinfty-initialval-2}
		\widetilde Z_N(t \mid t_0) 
		= 
		\zeta_N(t \mid t_0) \mathbf Z_N (t_0), 
		\quad
		\bbP\text{-a.s.},
	\end{equation}
	where we define,
	for any $\bm\xi \in L^2(\Omega, \cF_{t_0}, \bbP; U^N)$,
	\begin{equation}\label{eq:def-zetaN}
		\zeta_{N}(t \mid t_0)\bm\xi 
		:= 
		\sum_{k=0}^{N-1} \frac{(t-t_0)^k}{k!}
		\overline\Gamma(N-k, (t-t_0)A)\, \xi_k,
	\end{equation}
	using the incomplete gamma functions defined in~\eqref{eq:def-incgamma}.
	
	In particular, adding $Z_N(t \mid t_0)$ on both sides of equation~\eqref{eq:relation-ZN-minusinfty-initialval-2} yields
	\begin{equation}\label{eq:infinite-time-process-as-finite-time-process}
		\forall t \in [t_0, \infty) : \quad 
		Z_N(t) = Z_N(t \mid t_0, \mathbf Z_N(t_0)), \quad \bbP\text{-a.s.},
	\end{equation}
	where the process $(Z_N(t \mid t_0, \bm\xi))_{t \in [t_0, \infty)}$ is defined by
	\begin{equation}
		\label{eq:def-Zt0xi}
		Z_N(t \mid t_0, \bm{\xi}) 
		:= 
		\zeta_{N}(t \mid t_0)\bm\xi
		+
		Z_N(t \mid t_0), 
		\quad 
		t \in [t_0, \infty).
	\end{equation}
\end{proposition}
By~\eqref{eq:infinite-time-process-as-finite-time-process}, it suffices to 
show that $(Z_N(t \mid t_0, \bm\xi))_{t \in [t_0, \infty)}$
has the $N$-ple Markov property in the sense of Definition~\ref{def:Nple-markov-Doob}
for arbitrary $t_0 \in \R$ and $\bm\xi \in L^2(\Omega, \cF_{t_0}, \bbP; U^N)$.
In fact, we will show that it is $N$-ple Markov using
Corollary~\ref{def:Nple-markov-TO}; this is the 
subject of the following result, which is the main theorem of this section.
\begin{theorem}\label{thm:Nple-Markov-ZN}
	Let $N \in \N$, $t_0 \in \bbR$ and 
	$\bm\xi = (\xi_k)_{k=0}^{N-1} \in L^2(\Omega, \cF_{t_0},\bbP; U^{N+1})$
	be given.
	Let Assumptions~\textup{\ref{ass:standing} and~\ref{ass:HAQ}\ref{ass:HAQ:2}} hold
	and suppose that Assumption~\textup{\ref{ass:HAQ}\ref{ass:HAQ:1}} is satisfied for $\gamma_0 = 1$.
	Then the process $(Z_N(t \mid t_0, \bm{\xi}))_{t\in[t_0,\infty)}$ from~\eqref{eq:def-Zt0xi}
	has the $N$-ple Markov property in the sense of Definition~\textup{\ref{def:Nple-markov-Doob}}
	with respect to the transition operators
	$(T_{s,t})_{t_0 \le s \le t}$ on $B_b(U^N)$ defined by
	\begin{equation}\label{eq:transition-Z}
		T_{s, t}\varphi(\mathbf x) := \E[\varphi(\mathbf Z_N(t \mid s, \mathbf x))],
		\quad
		\varphi \in B_b(U^N), 
		\;
		\mathbf x \in U^N,
	\end{equation}
	and the increment filtration $(\cF_t)_{t\in[t_0, \infty)} := (\cF_t^{\delta W^Q})_{t\in[t_0, \infty)}$
	from~\eqref{eq:increment-filtration}.
	
	The process $(Z_N(t))_{t\in \bbR}$ from~\eqref{eq:def-minusinfty-Zgamma}
	has $N-1$ mean-square derivatives and
	is $N$-ple Markov in the sense of Definition~\ref{def:Nple-markov-Doob}
	with respect to $(\cF_t)_{t\in\bbR}$.
\end{theorem}
The statements and proofs of Proposition~\ref{prop:relation-ZN-initialvalue} and Theorem~\ref{thm:Nple-Markov-ZN} 
use the following result regarding the mean-square differentiability of 
$(\zeta_{N}(t \mid t_0)\bm\xi)_{t \in [t_0,\infty)}$, which is similar to~Proposition~\ref{prop:Zgamma-diffbty}.
\begin{proposition}\label{prop:recurrence-zeta}
	Let $N \in \{2,3,\dots\}$, $t_0 \in \bbR$ and $\bm\xi \in L^2(\Omega, \cF_{t_0},\bbP; U^{N})$
	be given,
	where $\xi_k \in L^2(\Omega; \mathsf D(A))$ for $k \in \{0, \dots,  N-2\}$.
	Suppose that Assumptions~\textup{\ref{ass:standing} and~\ref{ass:HAQ}\ref{ass:HAQ:2}} hold.
	Then the process $(\zeta_{N}(t \mid t_0)\bm{\xi})_{t \in [t_0,\infty)}$ from~\eqref{eq:def-zetaN}
	is infinitely mean-square differentiable at any $t \in (t_0, \infty)$
	and, for $n \in \{0,\dots,N-2\}$,
	\begin{equation}\label{eq:recurrence-zeta}
		\left(\tfrac{\mathrm d}{\mathrm dt} + A\right)\tfrac{\mathrm d^n}{\mathrm dt^n}
		\zeta_{N}(t \mid t_0) \bm \xi
		= 
		\tfrac{\mathrm d^n}{\mathrm dt^n} \zeta_{N-1}(t \bigm\vert t_0)(\xi_{k+1} + A\xi_{k})_{k=0}^{N-2},
		\quad
		\bbP\text{-a.s.}
	\end{equation}
	Moreover, we have $\left(\tfrac{\mathrm d}{\mathrm dt} + A\right) \tfrac{\mathrm d^n}{\mathrm dt^n} \zeta_{1}(t \mid t_0)\xi 
	= 
	0$, $\bbP$-a.s.,
	for $\xi \in L^2(\Omega, \cF_{t_0}, \bbP; U)$.
\end{proposition}
Combining Propositions~\ref{prop:Zgamma-diffbty}, \ref{prop:relation-ZN-initialvalue} and~\ref{prop:recurrence-zeta} yields:
\begin{corollary}\label{cor:recurrence-}
	Let $N \in \{2,3,\dots\}$, $t_0 \in \bbR$ and $\bm\xi \in L^2(\Omega, \cF_{t_0},\bbP; U^{N})$
	be given,
	where $\xi_k \in L^2(\Omega; \mathsf D(A))$ for $k \in \{0, \dots,  N-2\}$.
	Let Assumptions~\textup{\ref{ass:standing} and~\ref{ass:HAQ}\ref{ass:HAQ:2}} hold
	and suppose that Assumption~\textup{\ref{ass:HAQ}\ref{ass:HAQ:1}} is satisfied for $\gamma_0 = 1$.
	Then the process $(Z_{N}(t \mid t_0, \bm{\xi}))_{t \in [t_0,\infty)}$ from~\eqref{eq:def-Zt0xi}
	is $N-1$ times mean-square differentiable at any $t \in (t_0, \infty)$
	and satisfies, for $n \in \{0, \dots, N-2\}$,
	\begin{equation}\label{eq:recurrence-Z0-SP}
		\left(\tfrac{\mathrm d}{\mathrm dt} + A\right) \tfrac{\mathrm d^n}{\mathrm dt^n} Z_{N}(t \mid t_0, \bm \xi) 
		= 
		\tfrac{\mathrm d^n}{\mathrm dt^n} Z_{N-1}\bigl(t \bigm\vert t_0, (\xi_{k+1} + A\xi_{k})_{k=0}^{N-2}\,\bigr),
		\quad
		\bbP\text{-a.s.}
	\end{equation}
	In particular, it holds for all $t \in \bbR$ and $n \in \{0, \dots, N-2\}$ that
	\begin{equation}
		\left(\tfrac{\mathrm d}{\mathrm dt} + A\right)\tfrac{\mathrm d^n}{\mathrm dt^n} Z_{N}(t) 
		= 
		\tfrac{\mathrm d^n}{\mathrm dt^n} Z_{N-1}(t),
		\quad
		\bbP\text{-a.s.}
	\end{equation}
\end{corollary}
\begin{remark}\label{rem:relation-ZN-IVP}
	Corollary~\ref{cor:recurrence-} can be interpreted as saying
	that $(Z_{N}(t \mid t_0, \bm\xi))_{t \in [t_0, \infty)}$ for $N \in \{2,3,\dots\}$ solves 
	the $L^2(\Omega, \cF, \bbP; U)$-valued 
	initial value problem
	\begin{equation}
		\left\lbrace
		\begin{aligned}
			\left(\tfrac{\mathrm d}{\mathrm dt} + A\right) X(t) 
			&= Z_{N-1}(t \mid t_0, (\xi_{k+1} + A\xi_{k})_{k=0}^{N-2}), 
			\quad \forall t \in (t_0, \infty), \\
			X(t_0) &= \xi_0,
		\end{aligned}
		\right.
	\end{equation}
	whenever $\xi_k \in L^2(\Omega, \cF_{t_0}, \bbP; \mathsf D(A))$ for $k \in \{0,\dots,N-2\}$. 
	This observation 
	is the key to the proofs of Propositions~\ref{prop:relation-ZN-initialvalue} and~\ref{prop:restart} below.
	It is also of interest for computational methods, as it
	implies that the computation of $Z_{N}(t \mid t_0, \bm\xi)$ amounts to solving 
	a first-order problem $N-1$ times. In fact, inductively applying this result 
	and the fact that 
	$(\tfrac{\mathrm d}{\mathrm dt} + A) \zeta_{1}(t \mid t_0) \eta 
	= 
	0$
	for any $\eta \in L^2(\Omega, \cF_{t_0},\bbP; U)$,
	we see that
	for $N \in \N$ 
	we may interpret $(Z_{N}(t \mid t_0, \bm\xi))_{t \in [t_0, \infty)}$ as the mild solution
	to the $N$th order initial value problem
	\begin{equation}
		\left\lbrace
		\begin{aligned}
			\left(\tfrac{\mathrm d}{\mathrm dt} + A\right)^{N} X(t) &= \dot W^Q(t), &&\forall t \in (t_0, \infty), \\
			\tfrac{\mathrm d^k}{\mathrm dt^k} X(t_0) &= \xi_k, &&\forall k \in \{0, \dots, N-1\}.
		\end{aligned}
		\right.
	\end{equation}
\end{remark}
Another key step in the proof of in Theorem~\ref{thm:Nple-Markov-ZN} 
is given by the following result, which essentially amounts to
the fact that $(T_{s,t})_{s \le t}$ satisfies~\ref{item:TO3}.
\begin{proposition}\label{prop:restart}
	Let $N \in \N$, $t_0 \in \bbR$ and $\bm\xi \in L^2(\Omega, \cF_{t_0},\bbP; U^{N})$
	be given.
	Let Assumptions~\textup{\ref{ass:standing} and~\ref{ass:HAQ}\ref{ass:HAQ:2}} hold
	and suppose that Assumption~\textup{\ref{ass:HAQ}\ref{ass:HAQ:1}} is satisfied for $\gamma_0 = 1$.
	The stochastic process $(Z_N(t \mid t_0, \bm{\xi}))_{t\in[t_0,\infty)}$ from~\eqref{eq:def-Zt0xi}
	has
	the \emph{$N$-ple Chapman--Kolmogorov property}, i.e.,
	for all $t_0 \le s \le t$ we have
	\begin{equation}\label{eq:Nple-restarting}
		\mathbf Z_N(t \mid t_0, \bm{\xi}) 
		= 
		\mathbf Z_N(t \mid s, \mathbf Z_N(s \mid t_0, \bm\xi) ),
		\quad 
		\bbP\text{-a.s.}
	\end{equation}
\end{proposition}
%

\subsubsection{Integer case; proofs}
\label{subsubsec:Markov-proofs}
As indicated in Subsection~\ref{subsubsec:Markov-results} above, 
the statements and proofs of Proposition~\ref{prop:relation-ZN-initialvalue} and Theorem~\ref{thm:Nple-Markov-ZN} 
rely on Proposition~\ref{prop:recurrence-zeta}, which we prove first.

\begin{proof}[Proof of Proposition~\ref{prop:recurrence-zeta}]
	We first make some general remarks regarding the operators $\overline\Gamma(n, tA)$
	from~\eqref{eq:def-incgamma}.
	Under Assumption~\ref{ass:HAQ}\ref{ass:HAQ:2},
	estimate~\eqref{eq:analytic-semigroup-est} implies that
	the set $\{t^j A^jS(t) : t \in (0,\infty)\} \subseteq \mathscr{L}(U)$ is uniformly bounded.
	It follows that
	$t \mapsto \overline\Gamma(n, tA)$ 
	is a strongly continuous
	function from $[0,\infty)$ to $\mathscr{L}(U)$ for any $n \in \N$,
	which at $t \in (0,\infty)$ admits a classical derivative
	satisfying the recurrence relation
	\begin{equation}\label{eq:recurrence-incompletegamma}
		\left(\tfrac{\mathrm d}{\mathrm d t} + A\right)\overline{\Gamma}(n, tA) 
		= 
		\begin{cases}
			0, &n = 1; \\
			A\overline{\Gamma}(n-1, tA), &n \in \{2,3,\dots\}.
		\end{cases}
	\end{equation} 
	To prove the proposition, we may assume $t_0 = 0$, so fix $t \in (0,\infty)$.
	For arbitrary $M \in \N$, $j \in \N_0$ and $\bm\eta \in L^2(\Omega; U^{M})$, we define
	$
	\zeta_{M,j}(t) \bm\eta
	:= 
	A^j \zeta_{M}(t \mid 0) \bm\eta.
	$
	Combining the product rule in the form 
	$(\frac{\mathrm d}{\mathrm dt} + A)(uv) = u'v + u[(\frac{\mathrm d}{\mathrm dt} + A)v]$
	with the above recurrence relation yields for $M \in \N$
	\begin{align}
		(\tfrac{\mathrm d}{\mathrm dt} &+ A)\zeta_{M,j}(t) \bm \eta
		\\&=
		\sum_{k=1}^{M-1} \frac{t^{k-1}}{(k-1)!} A^j \overline\Gamma(M-k, tA) \eta_k
		+
		\sum_{k=0}^{M-2} \frac{t^k}{k!} A^{j+1} \overline\Gamma(M-1-k, tA) \eta_k
		\notag
		\\
		&=
		\zeta_{(M-1),j}(t) (\eta_{k+1})_{k=0}^{M-1} + \zeta_{(M-1),(j+1)}(t) (\eta_{k})_{k=0}^{M-1}
		\label{eq:recurrence-Z0-SP-k0j}.
	\end{align}
	This shows in particular that~\eqref{eq:recurrence-zeta}
	holds for integers $N \ge 2$ and $n = 0$, 
	by applying~\eqref{eq:recurrence-Z0-SP-k0j} with $M = N$, $\bm\eta = \bm\xi$ and $j = 0$.
	Moreover, note that
	\[
	\left(\tfrac{\mathrm d}{\mathrm dt} + A\right) \zeta_{1,j}(t) \, \eta 
	=
	\left(\tfrac{\mathrm d}{\mathrm dt} + A\right) A^j S(t) \, \eta
	=
	0.
	\]
	Iteratively applying~\eqref{eq:recurrence-Z0-SP-k0j} and the latter identity then
	yields that $\zeta_{M,j}\bm\eta$ is $M-1$ times (mean-square) differentiable
	with an $n$th derivative of the form
	\begin{equation}\label{eq:deriv-initial-val-term}
		\zeta^{(n)}_{M,j}(t) \bm \eta 
		=
		\sum_{\ell=0}^n \sum_{m=0}^\ell C_{\ell,m} \zeta_{(M-\ell),(j+n-m)}(t) B_{\ell, m} \bm \eta,
	\end{equation}
	where $C_{\ell,m} \in \R$, $B_{\ell, m} \in \mathscr L(U^{M}; U^{M-\ell})$ and
	$\zeta_{(M-\ell),(j+n-m)} := 0$ if $M-\ell < 1$.
	In particular, $\zeta_{N}(\,\cdot \mid t_0)\bm{\xi}$ is $N-1$ times (mean-square) differentiable as claimed.
	In order to deduce that~\eqref{eq:recurrence-zeta} also holds for $n \in \{1, \dots, N-2\}$, we need
	to justify taking the $n$th derivative on both sides and commuting
	it with $A$. Since $A$ is closed,
	it suffices to verify that $\zeta_{N}'(\,\cdot \mid t_0)\bm \xi$,
	$A^j \zeta_{N}(\,\cdot \mid t_0)\bm \xi$,
	$\zeta_{N-1}(\,\cdot \mid t_0)  (\xi_{j+1})_{j=0}^{N-1}$ and $A^j \zeta_{N-1}(\,\cdot \mid t_0)  (\xi_{j})_{j=0}^{N-1}$
	admit $n$th derivatives for $j \in \{0,1\}$. Indeed,
	these assertions follow from~\eqref{eq:deriv-initial-val-term}.
\end{proof}
We can now prove Propositions~\ref{prop:relation-ZN-initialvalue} and~\ref{prop:restart}.
Note that for the proof of the latter we may use Corollary~\ref{cor:recurrence-}, since it combines
Propositions~\ref{prop:Zgamma-diffbty}, \ref{prop:relation-ZN-initialvalue} and~\ref{prop:recurrence-zeta}.
\begin{proof}[Proof of Proposition~\ref{prop:relation-ZN-initialvalue}]
	We use induction on $N \in \N$.
	For $N=1$ and $t \in (t_0, \infty)$,
	\begin{equation}
		\begin{aligned}
			\widetilde Z_1(t \mid t_0) 
			= 
			S(t-t_0) \int_{-\infty}^{t_0} S(t_0-s)\rd W^Q(s)
			=
			\zeta_1(t \mid t_0) Z_1(t_0), \quad \bbP\text{-a.s.}
		\end{aligned}
	\end{equation}
	Now suppose that the statement is true for a given $N \in \N$.
	By Proposition~\ref{prop:Zgamma-diffbty} and
	the discussion below~\eqref{eq:def-truncated-process}, 
	$Z_{N+1}$ and $\widetilde Z_{N+1}(\,\cdot \mid t_0)$ have $N$ mean-square derivatives
	which satisfy
	\begin{equation}
		\left(\tfrac{\mathrm d}{\mathrm dt} + A \right)Z_{N+1}^{(k)}(t) = Z_{N}^{(k)}(t)
		\quad 
		\text{and}
		\quad
		\left(\tfrac{\mathrm d}{\mathrm dt} + A \right) \widetilde Z_{N+1}(t \mid t_0)
		=
		\widetilde Z_{N}(t \mid t_0), \quad \bbP\text{-a.s.},
	\end{equation}
	for all 
	$k \in \{0, \dots, N-1\}$ and $t \in (t_0, \infty)$.
	Combined with Proposition~\ref{prop:recurrence-zeta}
	and the induction hypothesis, we find
	\begin{align}
		\bigl(\tfrac{\mathrm d}{\mathrm dt} &+ A \bigr) \zeta_{N+1}(t\mid t_0) \mathbf Z_{N+1}(t_0) 
		= 
		\zeta_N(t \mid t_0) \bigl[\bigl(\tfrac{\mathrm d}{\mathrm dt} + A\bigr)Z_{N+1}^{(k)}(t_0)\bigr]_{k=0}^{N-1}
		\\
		&=
		\zeta_N(t \mid t_0)\mathbf Z_{N}(t_0)
		=
		\widetilde Z_N(t \mid t_0)
		=
		\left(\tfrac{\mathrm d}{\mathrm dt} + A \right) \widetilde Z_{N+1}(t \mid t_0).
	\end{align}
	Since $\widetilde Z_{N+1}(t_0 \mid t_0) = Z_{N+1}(t_0) = \zeta_{N+1}(t_0 \mid t_0) \mathbf Z_{N+1}(t_0)$,
	we conclude that~\eqref{eq:relation-ZN-minusinfty-initialval-2}
	with $N+1$ holds on $[t_0, \infty)$ 
	by the uniqueness of solutions to $L^2(\Omega, \cF, \bbP; U)$-valued abstract 
	Cauchy problems, see~\cite[Chapter~4, Theorem~1.3]{Pazy1983}.
\end{proof}
\begin{proof}[Proof of Proposition~\ref{prop:restart}]
	Let $t_0 \le s \le t$.
	We use induction on $N\in \N$. For the base case $N=1$ we have
	\begin{equation*}\begin{aligned}
			&Z_1(t \mid s, Z_1(s \mid t_0, \xi) ) 
			= S(t-s)Z_1(s \mid t_0, \xi) + \int_{s}^t S(t-r)\rd W^Q(r) 
			\\
			&= 
			S(t-s)S(s-t_0)\xi + S(t-s)\int_{t_0}^{s} S(s-r)\rd W^Q(r) + \int_{s}^t S(t-r)\rd W^Q(r)
			\\
			&= 
			S(t-t_0)\xi + \int_{t_0}^{s} S(t-r)\rd W^Q(r) + \int_{s}^t S(t-r)\rd W^Q(r)
			=
			Z_1(t \mid t_0, \xi),
		\end{aligned}
	\end{equation*}
	$\bbP$-a.s., for $\xi \in L^2(\Omega, \cF_{t_0}, \bbP; \mathsf D(A))$.
	Now suppose that the result holds for $N \in \N$ and
	let $\bm\xi \in L^2(\Omega, \cF_{t_0}, \bbP; \mathsf D(A)^N)$.
	Then for any $t \in (s, \infty)$ we have
	\begin{align*}
		\bigl(\tfrac{\mathrm d}{\mathrm dt} &+ A\bigr) Z_{N+1} (t \mid s, \mathbf Z_{N+1}(s \mid t_0, \bm\xi))
		=
		Z_{N}\left(t \, \middle\vert \, s, \bigl[\bigl(\tfrac{\mathrm d}{\mathrm dt} + A\bigr)Z_{N+1}^{(k)}(s \mid t_0, \bm\xi)\bigr]_{k=0}^{N-1}\right)
		\\
		&=
		Z_{N}\bigl(t \bigm\vert s, {\mathbf Z}_{N}(s \mid t_0, [\xi_{k+1} + A\xi_k]_{k=0}^{N-1})\bigr)
		=
		Z_{N}(t \mid t_0, [\xi_{k+1} + A\xi_k]_{k=0}^{N-1})
		\\&=
		\bigl(\tfrac{\mathrm d}{\mathrm dt} + A\bigr) Z_{N+1} (t \mid t_0, \bm\xi), \quad \text{$\bbP$-a.s.,}
	\end{align*}
	where we applied Corollary~\ref{cor:recurrence-}
	in every identity except the third, which uses the induction hypothesis.
	Moreover, $Z_{N+1} (s \mid s, \mathbf Z_{N+1}(s \mid t_0, \bm\xi))
	=
	Z_{N+1}(s \mid t_0, \bm\xi)$ is evident from the definitions.
	Together, these facts imply that the difference process
	$Y := Z_{N+1} (\,\cdot \mid s, \mathbf Z_{N+1}(s \mid t_0, \bm\xi)) - Z_{N+1} (\,\cdot \mid t_0, \bm\xi)$
	solves the abstract Cauchy problem
	\begin{equation}
		\left\lbrace\begin{aligned}
			\left(\tfrac{\mathrm d}{\mathrm d t} + \cA_{\Omega}\right)Y(t) &= 0, \quad \forall t \in (s, \infty); \\
			Y(s) &= 0,
		\end{aligned}\right.
	\end{equation}
	where $\cA_\Omega \from L^2(\Omega; \mathsf D(A)) \subseteq L^2(\Omega; U) \to L^2(\Omega; U)$ is as in~\eqref{eq:Bochner-counterpart}.
	Since $-\cA_\Omega$ is the generator of a $C_0$-semigroup on
	$L^2(\Omega; U)$, see Lemma~\ref{lem:Bochner-space-operator}\ref{lem:Bochner-space-operator-c} in Appendix~\ref{app:frac-powers},
	the uniqueness result \cite[Chapter~4, Theorem~1.3]{Pazy1983}
	shows that $Y \equiv 0$ on $[s, \infty)$, 
	meaning that for all $t \in [s, \infty)$ we have
	\begin{equation}
		Z_{N+1}(t \mid t_0, \bm{\xi}) 
		= 
		Z_{N+1}(t \mid s, \mathbf Z_{N+1}(s \mid t_0, \bm\xi) ),
		\quad 
		\bbP\text{-a.s.}
	\end{equation}
	Taking $n$th mean-square derivatives on both sides for $n \in \{0, \dots, N\}$,
	which is justified by Corollary~\ref{cor:recurrence-}, 
	we find~\eqref{eq:Nple-restarting}. In order to establish this identity
	for general $\bm\xi \in L^2(\Omega, \cF_{t_0},\bbP; U^N)$, we employ the
	density of $\mathsf D(A)$ in $U$~\cite[Chapter~1, Corollary~2.5]{Pazy1983},
	which implies the density of $L^2(\Omega; \mathsf D(A)^N)$ in $L^2(\Omega; U^N)$,
	so that it suffices to argue that the mapping 
	$\bm \eta \mapsto \mathbf Z_{N+1}(t \mid t_0, \bm\eta)$ from $L^2(\Omega; U^N)$ to itself is continuous
	for any fixed $t \in [t_0, \infty)$.
	The continuity of $\bm\eta \mapsto \zeta_{N+1}(t \mid t_0)\bm\eta$ 
	follows from the fact that $\eta \mapsto \overline\Gamma(N+1-k, (t-t_0)) \eta$
	is bounded on $L^2(\Omega; U)$ for any $k \in \{0, \dots, N\}$.
	The same holds for the derivatives of $\zeta_{N+1}(\,\cdot \mid t_0)\bm\eta$ 
	since they are
	of the same form by Proposition~\ref{prop:recurrence-zeta}.
	Together, the conclusion follows.
\end{proof}

With these intermediate results in place, we are ready to prove the main 
theorem asserting the $N$-ple Markovianity 
of $Z_N$.
Its proof is a generalization of~\cite[Theorem~9.14]{DaPrato2014} and \cite[Theorem~9.30]{PeszatZabczyk2007},
which concern simple Markovianity for $N = 1$.
We divide it in two parts: In the first part, we verify that
$(T_{s,t})_{t_0 \le s \le t}$ is a well-defined family of transition operators on $B_b(U^N)$.
In the second part, we show that $Z_N(\,\cdot\mid t_0, \bm\xi)$ is $N$-ple Markov with respect to
$(T_{s,t})_{t_0 \le s \le t}$, see Corollary~\ref{def:Nple-markov-TO}
\begin{proof}[Proof of Theorem~\ref{thm:Nple-Markov-ZN}]
	\textbf{Step 1: Well-definedness of $(T_{s,t})_{t_0 \le s \le t}$}.
	We have to
	show that $T_{s, t}\varphi = \E[\varphi(\mathbf Z_N(t \, | \, s, \,\cdot\,))]$ is measurable 
	for $\varphi \in B_b(U^N)$.
	For a monotone class argument,
	we introduce the linear space $\mathscr H$ of bounded functions
	$\varphi \from U^N \to \bbR$ such that $\mathbf x \mapsto \E[\varphi(\mathbf Z_N(t\mid s, \mathbf x))]$ is measurable.
	Arguing similarly to the end of the proof of Proposition~\ref{prop:restart}, we find that 
	${\mathbf x \mapsto [\mathbf Z_N(t\mid s, \mathbf x)](\omega)}$
	is continuous on $U^N$ for $\bbP$-a.e.\ $\omega \in \Omega$.
	Thus, for $\varphi \in \mathscr C := C_b(U^N)$, the dominated convergence theorem implies that
	$\mathbf x \mapsto \E[\varphi(\mathbf Z_N(t\mid s, \mathbf x))]$
	is also continuous, hence (Borel) measurable.
	It follows that $\mathscr C \subseteq \mathscr H$.
	Moreover, $\mathscr H$ contains all constant functions, and given a
	sequence
	$(\varphi_n)_{n\in\N} \subseteq \mathscr H$ such that ${0 \le \varphi_n \uparrow \varphi}$ pointwise 
	for some bounded limit function
	$\varphi$, we find $\varphi \in \mathscr H$ by monotone convergence.
	Since $\mathscr C$ is closed under pointwise multiplication,
	we conclude that 
	$B_b(U^N\!, \sigma(\mathscr C)) = B_b(U^N) \subseteq \mathscr H$ by the monotone class
	theorem~\cite[Chapter~0, Theorem~2.2]{RevuzYor1999}.
	
	\textbf{Step 2: $N$-ple Markovianity}.
	For $t_0 \le s \le t$ and $\varphi \in B_b\left(U^N\right)$, 
	we show
	\begin{equation}
		\E[\varphi(\mathbf Z_N(t\mid t_0, \bm\xi)) \mid \cF_{s}] 
		= 
		T_{s, t} \varphi(\mathbf Z_N(s \mid t_0, \bm\xi)),
		\quad
		\bbP\text{-a.s.},
	\end{equation}
	for all $\bm \xi \in L^2(\Omega, \cF_{s}, \bbP; U^N)$.
	By Proposition~\ref{prop:restart}, it suffices to 
	verify that 
	\begin{equation}
		\E[\varphi(\mathbf Z_N(t\mid s, \bm \xi)) \mid \cF_{s}] 
		= 
		T_{s, t} \varphi(\bm\xi),
		\quad
		\bbP\text{-a.s.}
	\end{equation}
	By a monotone
	class argument similar to that of Step 1, 
	it suffices to consider $\varphi \in C_b(U^N)$.
	If $\bm\xi = \sum_{j=1}^n \mathbf x_j \mathbf 1_{A_j}$ with $n \in \N$,
	$\{\mathbf x_1, \dots, \mathbf x_n\} \subseteq U^N$ and
	disjoint events
	$\{A_1, \dots, A_n\} \subseteq \cF_{s}$ covering $\Omega$,
	then
	\[
	\mathbf Z_N(t\mid s, \bm \xi)
	=
	\sum\nolimits_{j=1}^n \mathbf Z_N(t \mid s, \mathbf x_j) \mathbf 1_{A_j}, \quad \bbP\text{-a.s.}
	\]
	For every $j \in \{1,\dots,n\}$, $\zeta_N(t\,|\,s)\mathbf x_j$
	is deterministic, 
	whereas $Z_N(t \, | \, s, \bm 0)$ is independent of
	$\cF_{s}$ by~\eqref{item:WP3prime} and the definition of the stochastic integral,
	thus $Z_N(t \mid s, \mathbf x_k) \ind \cF_{s}$.
	Since the mean-square derivatives of ${Z_N(\,\cdot\mid s, \mathbf x_k)}$ have the same form
	(see Proposition~\ref{prop:recurrence-zeta}), 
	we deduce
	$\mathbf Z_N(t \mid s, \mathbf x_k) \ind \cF_{s}$, so that
	\begin{equation}
		\begin{aligned}
			\E&[\varphi(\mathbf Z_N(t\mid s, \bm \xi)) \mid \cF_{s}]
			=
			\sum\nolimits_{k=1}^n 
			\E[\varphi(\mathbf Z_N(t\mid s, \mathbf x_k)) \mathbf 1_{A_k} \mid \cF_{s}] 
			\\
			&=
			\sum\nolimits_{k=1}^n 
			\E[\varphi(\mathbf Z_N(t\mid s, \mathbf x_k))] \mathbf 1_{A_k}
			=
			\E[\varphi(\mathbf Z_N(t\mid s, \bm\xi))]
			=
			[T_{s, t}\varphi](\bm\xi),
			\quad
			\bbP\text{-a.s.}
		\end{aligned}
	\end{equation}
	This shows the desired property for simple $\bm\xi$.
	For a general $\bm \xi \in L^2(\Omega, \cF_{s}, \bbP; U^N)$, 
	we can find a sequence $(\bm \xi_n)_{n\in\N}$ of 
	$\cF_{s}$-measurable, $U^N$-valued simple 
	random variables such that $\bm\xi_n \to \bm\xi$ in $L^2(\Omega; U^N)$.
	By the continuity of the mapping
	$\bm\xi \mapsto \mathbf Z_N(t\mid s, \bm\xi)$ on $L^2(\Omega; U^N)$
	and of $\varphi$,
	we have ${\varphi(\mathbf Z_N(t\mid s, \bm\xi_n)) \to \varphi(\mathbf Z_N(t\mid s, \bm\xi))}$
	in $L^2(\Omega; \R)$. From this, one can derive that
	\[
	\E[\varphi(\mathbf Z_N(t\mid s, \bm \xi_n)) \mid \cF_{s}](\omega) \to \E[\varphi(\mathbf Z_N(t\mid s, \bm \xi)) \mid \cF_{s}](\omega),
	\]
	$\bbP$-a.s., passing to a subsequence if needed.
	On the other hand, we can pass to a further subsequence
	in order to assume that $\bm \xi_n(\omega) \to \bm \xi(\omega)$ in $U^N$, $\bbP$-a.s., 
	and deduce that $[T_{s, t}\varphi](\bm\xi_n(\omega)) \to [T_{s, t}\varphi](\bm\xi(\omega))$ in $\bbR$,
	finishing the density argument.
	
	Finally, the statement regarding $(Z_N(t))_{t\in\bbR}$ follows from Proposition~\ref{prop:relation-ZN-initialvalue}.
\end{proof}

\subsubsection{Non-Markovianity in the fractional case}
\label{subsec:not-Markov-fractional}
We conclude this section by showing how 
Theorem~\ref{thm:weakMarkov-Gaussian-necessity} can be used
to deduce that $Z_\gamma$ is not weakly Markov (see Definition~\ref{def:weak-markov})
if $\gamma \not\in\N$. To this end, we determine
the coloring operator of $Z_\gamma$.

\begin{proposition}
	Let
	Assumption~\textup{\ref{ass:standing}} be
	satisfied and suppose that $\gamma \in (\nicefrac{1}{2},\infty)$
	is such that Assumption~\textup{\ref{ass:HAQ}\ref{ass:HAQ:1}} 
	holds for $\gamma_0 = \gamma$.
	Then
	the coloring operator of
	$Z_\gamma$, see~\eqref{eq:X-colored-by-Linverse}, is given by 
	$\cL^{-1}_\gamma = \mathfrak I^{\gamma} \cQ_{\bbR}^\frac{1}{2}  \in \LO(L^2(\bbR; U))$,
	where $\mathfrak I^\gamma$ is as in~\eqref{eq:def-Igamma}.
\end{proposition}
\begin{proof}
	Note that for each $f \in C_c^\infty(\bbR; U)$, we have
	\begin{align}
		\scalar{Z_\gamma, f}{L^2(\bbR; U)}
		&=
		\frac{1}{\Gamma(\gamma)} \int_\bbR \int_{-\infty}^t \scalar{(t-s)^{\gamma-1}S(t-s) \rd W^Q(s), f(t)}{U} \rd t, 
		\\
		&= 
		\int_\bbR \int_\bbR \widetilde \Psi_f(t,s) \rd W^Q(s) \rd t,
		\quad \bbP\text{-a.s.},
	\end{align}
	where $\widetilde \Psi_f \from \bbR^2 \to \LO(U; \bbR)$ is given by 
	\begin{equation}
		\widetilde \Psi_f(t,s)u
		:= 
		\frac{1}{\Gamma(\gamma)}\scalar{(t-s)^{\gamma-1} S(t-s) u, f(t)}{U} \, \mathbf 1_{(-\infty, t)}(s),
		\quad 
		t, s \in \bbR,\,  u \in U.
	\end{equation}
	Indeed, the latter integral is well-defined since
	\begin{align*}
		\int_\bbR &\biggl[\int_\bbR \Norm{\widetilde \Psi_f(t,s) Q^{\frac{1}{2}}}{\LO_2(U; \bbR)}^2 \rd s\biggr]^{\frac{1}{2}} \rd t
		\\
		&=
		\frac{1}{\Gamma(\gamma)}
		\int_\bbR \biggl[\int_{-\infty}^t \norm{(t-s)^{\gamma-1} Q^\frac{1}{2}[S(t-s)]^* f(t)}{U}^2 \rd s\biggr]^{\frac{1}{2}} \rd t 
		\\
		&\le 
		\frac{1}{\Gamma(\gamma)}
		\biggl[\int_0^\infty \norm{s^{\gamma-1} Q^\frac{1}{2}[S(s)]^*}{\LO_2(U)}^2 \rd s\biggr]^{\frac{1}{2}}
		\biggl[\int_\bbR \norm{f(t)}{U}  \rd t\biggr]  < \infty;
	\end{align*}
	note that the first integral on the last line is finite by Assumption~\ref{ass:HAQ}\ref{ass:HAQ:1} 
	and the second by the compact support of $f$.
	This also shows that Theorem~\ref{thm:stoch-fub} is applicable, hence
	\begin{align}
		\scalar{Z_\gamma, f}{L^2(\bbR; U)}
		&=
		\int_\bbR \int_\bbR \widetilde \Psi_f(t,s) \rd t \rd W^Q(s)
		=
		\int_\bbR \widetilde \Phi_f(s) \rd W^Q(s), \quad \bbP\text{-a.s.},
	\end{align}
	where $\widetilde \Phi_f \from \bbR \to \LO(U; \bbR)$ is given by 
	\begin{align}
		\widetilde \Phi_f(s)u 
		:= 
		\int_\bbR \widetilde \Psi_f(t,s) u \rd t
		&=
		\frac{1}{\Gamma(\gamma)}
		\int_{s}^\infty
		\scalar{(t-s)^{\gamma-1} S(t-s) u, f(t)}{U}\rd t
		\\
		&=
		\scalar{u, \mathfrak I^{\gamma*} f(s)}{U},
		\quad 
		\forall
		s \in \bbR,\,  u \in U,
	\end{align}
	where we recall equation~\eqref{eq:Igamma-adjoint} for the adjoint of $\mathfrak I^\gamma$.
	Consequently, the It\^o isometry~\eqref{eq:ito-R} yields
	\begin{equation}
		\norm{\scalar{Z_\gamma, f}{L^2(\bbR; U)}}{L^2(\Omega)}^2
		=
		\int_\bbR \norm{\widetilde \Phi_f(t) Q^\frac{1}{2}}{\LO_2(U; \bbR)}^2 \rd t
		=
		\norm{\cQ_\bbR^\frac{1}{2} \mathfrak I^{\gamma*} f}{L^2(\bbR; U)}^2.
	\end{equation}
	An application of the polarization identity then shows
	the coloring property~\eqref{eq:X-colored-by-Linverse}
	with $\cL^{-*}_\gamma = \cQ_\bbR^\frac{1}{2} \mathfrak I^{\gamma*}$.
\end{proof}

\begin{example}\label{ex:non-Markovianity}
	Let
	Assumptions~\textup{\ref{ass:standing} and~\ref{ass:HAQ}\ref{ass:HAQ:2}} be
	satisfied and suppose that $\gamma \in (\nicefrac{1}{2},\infty)$
	is such that Assumption~\textup{\ref{ass:HAQ}\ref{ass:HAQ:1}} 
	holds for $\gamma_0 = \gamma$.
	The latter implies that
	$\partial_t + \cA_\bbR = \cB$, and we always have
	$\cB^{-\gamma} = \mathfrak I^\gamma$, see Subsection~\ref{subsec:fractional-int-diff}. Thus,
	\[
	\cL_\gamma^* \cL_\gamma 
	=
	\bigl(\cQ_{\bbR}^{-\frac{1}{2}} \cB^{\gamma} \bigr)^* \cQ_{\bbR}^{-\frac{1}{2}} \cB^{\gamma} 
	=
	\cB^{\gamma*} \cQ_{\bbR}^{-1} \cB^{\gamma} 
	=
	(\partial_t + \cA_\bbR)^{\gamma*} \cQ_{\bbR}^{-1} (\partial_t + \cA_\bbR)^{\gamma}.
	\]
	Moreover, this assumption implies that $\dom{A^n}$ is dense in $U$ for all
	$n \in \N$ by \cite[Chapter~2, Theorem~6.8(c)]{Pazy1983}; choosing
	$n$ large enough, we also find that $C_c^\infty(\bbT; \dom{A^n})$ is dense in $\dom{\cB^\gamma}$,
	so we can take $F = \dom{A^n}$ in Theorem~\ref{thm:weakMarkov-Gaussian-necessity}.
	
	Although $Q^{-1}$ 
	may be a nonlocal spatial operator, $\cQ_\bbR^{-1}$ is always local in time.
	Thus for $\gamma \in \N$, the precision operator
	is local as a composition of three local operators, which is in accordance with
	the Markovianity shown in Subsection~\ref{subsec:fractional-int-diff:Markov}.
	
	For $\gamma \not \in \N$, we will show that the precision operator is not local in general.
	Suppose that $A$ has an eigenvector $v \in U$ with corresponding eigenvalue $\lambda \in \R$.
	Such eigenpairs exists for example if $A = (\kappa^2 - \Delta)^\beta$ 
	with $\kappa, \beta \in (0,\infty)$ and $\Delta$ the Dirichlet Laplacian on a bounded Euclidean
	domain $\cD \subsetneq \bbR^d$.
	If we moreover assume that $v \in \dom{Q^{-\frac{1}{2}}}$, then we find
	\begin{equation}
		\forall \phi \in C_c^\infty(\R) : \quad 
		\cQ_{\bbR}^{-\frac{1}{2}} \cB^\gamma (\phi \otimes v) = [(\partial_t + \lambda)^\gamma \phi] \otimes Q^{-\frac{1}{2}}v
	\end{equation}
	since the spectral mapping theorem implies $S(t)v = e^{-\lambda t}v$ for all $t \in [0,\infty)$.
	It thus suffices to consider the case $A = \lambda \in \bbR$, i.e., 
	we wish to find disjointly supported $\phi, \psi \in C_c^\infty(\bbR)$ such that
	\begin{equation}\label{eq:Fgamma}
		F_\gamma(\phi, \psi) :=
		|\scalar{(\partial_t + \lambda)^\gamma \phi, (\partial_t + \lambda)^\gamma \psi}{L^2(\R)}|
		\neq 0.
	\end{equation}
	We will discuss this by means of a numerical experiment for the case $\lambda = 1$,
	using the following smooth function $\phi \in C_c^\infty(\bbR)$ supported on $[-1,1]$:
	\begin{equation}\label{eq:bump}
		\phi(t) := \begin{cases}
			\exp\left(-\frac{1}{1-x^2}\right)\!, &x \in (-1, 1), \\
			0, &x \in \R\setminus (-1,1),
		\end{cases}
	\end{equation}
	and taking $\psi := \phi( \,\cdot\, - 2 - \delta)$ for some $\delta \in (0,\infty)$.
	In Figure~\ref{fig:frac-parab-derivs-plots}, 
	we see that the parabolic derivatives of $\phi$ consists of
	(positive or negative) peaks.
	For $\gamma \not\in \N$, the support of the last of these peaks appears to
	include the whole of $[1,\infty)$,
	with its absolute value
	taking rapidly decaying yet nonzero values there.
	Therefore, the idea is to take $\delta$ small enough,
	making the right-hand side tail of $\phi$ overlap with
	the first peak of $\psi$ to obtain a non-zero $L^2(\bbR)$-inner product.
	Table~\ref{tab:inner-prods} shows the approximate outcomes 
	of this process for various values of $\gamma$ and $\delta$, using symbolic differentiation and 
	numerical integration.
	
	\begin{table}
		\small
		\centering
		\begin{tabular}{r|cccccccccccc}
			\toprule
			\diagbox{$\delta$}{$\gamma$} & 0.25 & 0.50 & 0.75 & 1 & 1.25 & 1.50 & 1.75 & 2 & 2.25 & 2.50 & 2.75 & 3 \\
			\midrule
			$10^{-1}$ & 0.004 & 0.007 & 0.007 & 0 & 0.016 & 0.042 & 0.059 & 0 & 0.298 & 1.078 & 2.111 & 0 \\
			$10^{-2}$ & 0.005 & 0.009 & 0.009 & 0 & 0.024 & 0.065 & 0.098 & 0 & 0.622 & 2.568 & 5.829 & 0 \\
			$10^{-3}$ & 0.005 & 0.009 & 0.009 & 0 & 0.025 & 0.068 & 0.104 & 0 & 0.678 & 2.850 & 6.601 & 0 \\
			\bottomrule
		\end{tabular}%
		\caption{\label{tab:inner-prods} Numerically approximated values of 
			$F_\gamma(\phi, \psi)$, see~\eqref{eq:Fgamma},
			with $\phi$ 
			from~\eqref{eq:bump} and $\psi := \phi(\,\cdot\, - 2 - \delta)$
			for certain values of $\gamma$ and $\delta$.}
	\end{table}%
	
	\begin{figure}
		\centering
		\includegraphics[width=.8\textwidth]{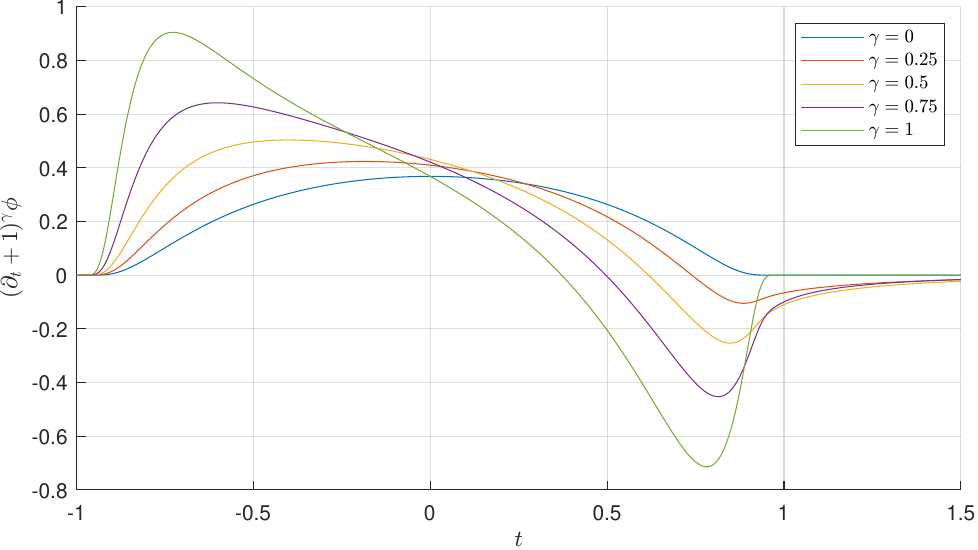}
		\includegraphics[width=.8\textwidth]{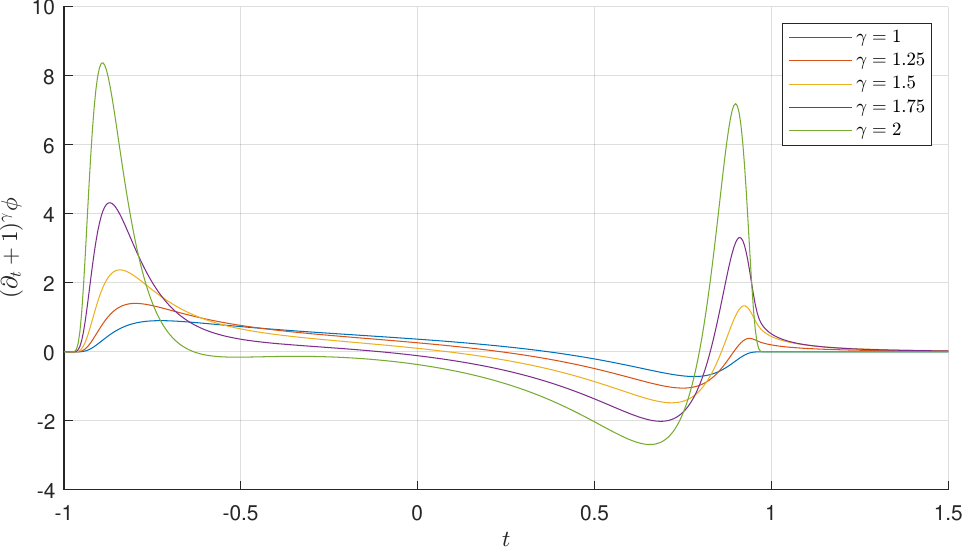}
		\caption{Graphs of the fractional parabolic derivative $(\partial_t + 1)^\gamma \phi$ with $\phi$ 
			from~\eqref{eq:bump} for certain values of $\gamma \in [0,2]$.
			Note the different scales on the $y$-axes.}
		\label{fig:frac-parab-derivs-plots}
	\end{figure}
	
	Note the contrast with the merely spatial Mat\'ern case, where the self-adjointness of
	the shifted Laplacian $\kappa^2 - \Delta$
	causes $L_\beta^* L_\beta = \tau^2 (\kappa^2 - \Delta)^{2\beta}$, thus 
	we find
	a weak Markov property also for half-integer values of $\beta \in (0,\infty)$.
\end{example}

\section{Fractional $Q$-Wiener process}
\label{sec:frac-QWiener}
In this section we consider the \emph{fractional $Q$-Wiener process}, 
which is a $U$-valued analog to 
fractional Brownian motion; 
we define it as in~\cite[Definition~2.1]{Duncan2002}.
\begin{definition}\label{def:fQWP}
	Let $Q \in \LO_1^+(U)$.
	A $U$-valued Gaussian process $(W^Q_H(t))_{t\in\R}$ is called a \emph{fractional $Q$-Wiener process
		with Hurst parameter $H \in (0, 1)$} if
	\begin{enumerate}[label=(f-WP\arabic*), leftmargin=1.6cm]
		\item $\E[W^Q_H(t)] = 0$ for all $t \in \bbR$; \label{item:fWP1}
		\item $Q_H(s,t) = \frac{1}{2}(|t|^{2H} + |s|^{2H} - |t-s|^{2H}) Q$ for all $s, t \in \R$;  \label{item:fWP2}
		\item $W^Q_H$ has continuous sample paths. \label{item:fWP3}
	\end{enumerate}
	Here, $(Q_{H}(s,t))_{s,t\in \R} \subseteq \mathscr L_1^+(U)$
	are the covariance operators of $W^Q_H$,
	cf.~\eqref{eq:cov-ops}.
\end{definition}
Note that for $H = \frac{1}{2}$, the above definition reduces to a characterization of a standard 
(non-fractional) $Q$-Wiener process when restricted to $[0,\infty)$.

In the real-valued setting, the fractional Brownian motion 
can also be represented as a stochastic integral over $\bbR$~\cite[Definition~2.1]{MandelbrotVanNess1968}.
Proposition~\ref{prop:fQWP} below
shows that such a representation is also available in the $U$-valued setting
using the stochastic integral over $\bbR$ from Section~\ref{subsec:prelims:stochastic-int-twosided-Wiener}.
This representation can be leveraged to prove that a fractional $Q$-Wiener process
can be viewed as a limiting case of the class of processes discussed in Section~\ref{sec:frac-stoch-ACP-R},
see Proposition~\ref{prop:fQWP-limit-of-Zgamma}.
We conclude this section by making some remarks about the Markov behavior of $W^Q_H$, see
Subsection~\ref{subsec:fQWP:Markov}.

\subsection{Integral representation and relation to $Z_\gamma$}

Let $Q \in \LO_1^+(U)$ and let $(W^Q(t))_{t\in\bbR}$ be a two-sided $Q$-Wiener process, see 
Subsection~\ref{subsec:prelims:stochastic-int-twosided-Wiener}.
We define the process $(\widehat W^Q_H(t))_{t\in\bbR}$ via
the following Mandelbrot--Van Ness~\cite{MandelbrotVanNess1968} type integral:
\begin{equation}\label{eq:WQH}
	\widehat W_H^Q(t) 
	:=
	\int_\bbR K_H(t,r) \rd W^Q(r), \quad t \in \R,
\end{equation}
where the
integration kernel $K_H \from \R^2 \to \R$ given by
\begin{equation}\label{eq:Mandelbrot-vanNess-int-kern}
	K_H(t,r) := \frac{(t-r)^{H-\frac{1}{2}}_+ - (-r)_+^{H-\frac{1}{2}}}{C_H}, 
	\quad 
	(t,r) \in \R^2.
\end{equation}
The normalizing constant $C_H$ is chosen as
\[
C_H := 
\int_\R \Bigl\vert(1-r)_+^{H-\frac{1}{2}} - (-r)^{H-\frac{1}{2}}_+ \Bigr\vert^2 \rd r 
=
\frac{3 - 2H}{4H} \mathrm B(2 - 2H, H + \tfrac{1}{2}),
\]
where $\mathrm B$ denotes the beta function~\cite[Theorem~B.1]{Picard2011},
so that $\widehat Q_H(1,1) = Q$, where $(\widehat Q_H(t,s))_{t,s\in\bbR}$
denotes the family of covariance operators associated to $\widehat W^Q_H$.
From this definition it follows that
\begin{equation}
	\widehat W^Q_H(t) 
	=
	\frac{1}{C_H}\int_{-\infty}^0 [(t-r)^{H - \frac{1}{2}} - (-r)^{H - \frac{1}{2}}] \rd W^Q(r)
	+
	\frac{1}{C_H}\int_0^t (t - r)^{H - \frac{1}{2}} \rd W^Q(r)
\end{equation}
for $t \in [0,\infty)$, whereas for $t \in (-\infty, 0)$ we have
\begin{equation}
	\widehat W^Q_H(t) 
	=
	\frac{1}{C_H}\int_{-\infty}^t [(t-r)^{H - \frac{1}{2}} - (-r)^{H - \frac{1}{2}}] \rd W^Q(r)
	+
	\frac{1}{C_H}\int_t^0 ( - r)^{H - \frac{1}{2}} \rd W^Q(r).
\end{equation}
In particular, for $H = \frac{1}{2}$ and $t \in [0,\infty)$ 
we find $\widehat W^Q_{\nicefrac{1}{2}}(t) = W^Q(t)$, $\bbP$-a.s.

We then have the following relation to the fractional $Q$-Wiener process:
\begin{proposition}\label{prop:fQWP}
	Formula~\eqref{eq:WQH} yields a well-defined square-integrable
	random variable
	$\widehat W^Q_H(t) \in L^2(\Omega, \cF_t^{\delta W^Q}, \bbP; U)$ for each $t \in \bbR$, and there exists a modification of 
	the process $(\widehat W^Q_H(t))_{t\in\bbR}$
	which is a fractional $Q$-Wiener process in the sense of Definition~\ref{def:fQWP}.
\end{proposition}
\begin{proof}
	By the It\^o isometry~\eqref{eq:ito-R} it suffices to note that
	\[
	\E\Bigl[\bigl\|\widehat W^Q_H(t)\bigr\|_{U}^2\Bigr] = \int_\bbR \norm{K_H(t,r)Q^\frac{1}{2}}{\LO_2(U)}^2 \rd r
	=
	\tr Q \int_\bbR |K_H(t,r)|^2 \rd r 
	< \infty
	\]
	to see that $\widehat W^Q_H(t) \in L^2(\Omega; U)$ is well-defined.
	Since $K_H$ is a deterministic kernel integrated with respect to a mean-zero Gaussian
	process $W^Q$, it readily follows that $\widehat W^Q_H$ is also mean-zero Gaussian.
	For the covariance operators $(\widehat Q_H(t,s))_{t,s\in\bbR}$ 
	of $\widehat W^Q_H$, we can argue as in the proof of Proposition~\ref{prop:separable-cov-matern} to find
	\begin{equation}
		\forall t,s \in \bbR : \quad 
		\widehat Q_H(t,s) = \int_\bbR K_H(s,r) K_H(t,r) \rd r \, Q
		=
		\E[B_H(t) B_H(s)] \, Q,
	\end{equation}
	where $B_H = (B_H(t))_{t\in\bbR}$ denotes (real-valued) fractional Brownian motion.
	Hence \ref{item:fWP2} holds by the properties of $B_H$.
	
	Lastly, we shall establish the existence of a continuous modification of~\eqref{eq:WQH}. 
	To this end, we first remark that
	$\widehat W^Q_H$ is self-similar with exponent $H$ and that its increments are stationary;
	this means for any $t \in \R$ we have, respectively,
	\begin{align}
		\label{eq:self-similar}
		\forall \alpha \in (0,\infty) &: \quad 
		\widehat W^Q_H(\alpha t) \overset{d}{=} \alpha^H \widehat W^Q_H(t), \\
		\forall h \in \R &: \quad 
		\widehat W^Q_H(t + h) - \widehat W^Q_H(t) \overset{d}{=} \widehat W^Q_H(h).
		\label{eq:stat-incr}
	\end{align}
	The proofs of these statements are analogous to the real-valued case;
	they involve changing variables in the integral~\eqref{eq:WQH} and using the facts that
	$\alpha^{-\frac{1}{2}} W^Q(\alpha\,\cdot\,)$ and
	$W^Q(\,\cdot\, + h)$, respectively, are also two-sided $Q$-Wiener processes.
	Thus,
	\begin{equation}
		\begin{aligned}
			\bigl\| \widehat W^Q_H(t + h) - \widehat W^Q_H(t) \bigr\|_{L^p(\Omega; U)}
			=
			\bigl\| \widehat W^Q_H(|h|) \bigr\|_{L^p(\Omega; U)}
			=
			|h|^H \bigl\| \widehat W^Q_H(1) \bigr\|_{L^p(\Omega; U)}
		\end{aligned}
	\end{equation}
	for any $p \in (1,\infty)$.
	By the Kahane--Khintchine inequalities (see, e.g.,~\cite[Theorem~6.2.6]{HvNVWVolumeII}), 
	there exists a constant $C_p \in (0,\infty)$ such that
	\[
	\bigl\| \widehat W^Q_H(1) \bigr\|_{L^p(\Omega; U)} 
	\le 
	C_p \bigl\| \widehat W^Q_H(1) \bigr\|_{L^2(\Omega; U)} = C_p \tr Q.
	\]
	We have thus shown that $t \mapsto \widehat W^Q_H(t)$
	is $H$-H\"older continuous from $\bbR$ to $L^p(\Omega; U)$.
	Since $p \in (1,\infty)$ was arbitrary, we can apply 
	the Kolmogorov--Chentsov theorem~\cite[Corollary~3.10]{Cox2021local}
	with $p \in (\nicefrac{1}{H}, \infty)$ to conclude
	that $\widehat W^Q_H$ has a modification with continuous sample paths.
\end{proof}

Now we consider the relation between fractional $Q$-Wiener process
and the process $Z_\gamma$ considered in Section~\ref{sec:frac-stoch-ACP-R}.
For $\varepsilon \in (0,\infty)$, let $Z_\gamma^\varepsilon$ denote the mild solution
to~\eqref{eq:the-SPDE} with $A = \varepsilon \id_U$
and define the process 
$\bigl(\overline Z_\gamma^\varepsilon(t)\bigr)_{t \in \R}$ by 
\begin{equation}\label{eq:Zgammaepsilon-bar}
	\overline Z_\gamma^\varepsilon(t) := \frac{\Gamma(\gamma)}{C_{\gamma-\frac{1}{2}}} (Z_\gamma^\varepsilon(t) - Z_\gamma^\varepsilon(0)),
	\quad
	t \in \R.
\end{equation}
Note that
$\widehat W^Q_H(t)$ can formally be written as a
``convergent difference of divergent integrals,'' cf.~\cite[Footnote~3]{MandelbrotVanNess1968}:
\[
\frac{1}{C_H}
\biggl[
\int_{-\infty}^{t}(t-s)^{H-\frac{1}{2}} \rd W^Q(s)
-
\int_{-\infty}^{0}(-s)^{H-\frac{1}{2}} \rd W^Q(s)
\biggr].
\]
The above expression would correspond to taking $\varepsilon = 0$ in~\eqref{eq:Zgammaepsilon-bar},
which is ill-defined
since Assumption~\ref{ass:HAQ}\ref{ass:HAQ:1} cannot be satisfied.
However, the next result shows that a fractional $Q$-Wiener process
can be viewed as a limiting case of $\overline Z_\gamma^\varepsilon$ as $\varepsilon \downarrow 0$.

\begin{proposition}\label{prop:fQWP-limit-of-Zgamma}
	Let $Q \in \LO_1^+(U)$ and $\gamma \in (\nicefrac{1}{2},\nicefrac{3}{2})$.
	The family of stochastic processes $(\overline Z_\gamma^\varepsilon)_{\varepsilon \in (0,\infty)}$
	defined by~\eqref{eq:Zgammaepsilon-bar} converges uniformly
	on compact subsets of $\bbR$ in mean-square sense to the fractional
	$Q$-Wiener process $\widehat W^Q_H$ in~\eqref{eq:WQH}
	with Hurst parameter 
	$H = \gamma - \frac{1}{2}$ as $\varepsilon \downarrow 0$:
	\begin{equation}
		\forall \, T \in (0,\infty) : \quad 
		\lim_{\varepsilon \downarrow 0} \sup_{t \in [-T, T]}
		\Norm{\widehat W^Q_{\gamma - \frac{1}{2}}(t) - \overline Z_\gamma^\varepsilon(t)}{L^2(\Omega; U)} = 0.
	\end{equation}
\end{proposition}
\begin{proof}
	For $t \in [0,\infty)$, we can write
	\begin{align*}
		\widehat W^Q_{\gamma - \frac{1}{2}}(t) &- \overline Z_\gamma^\varepsilon(t)
		=
		\frac{1}{C_{\gamma - \frac{1}{2}}}
		\int_0^t (t-s)^{\gamma-1} \bigl(1 - e^{-\varepsilon(t-s)}\bigr) \rd W^Q(s)
		\\
		&+ 
		\frac{1}{C_{\gamma - \frac{1}{2}}}
		\int_{-\infty}^0 \left[(t-s)^{\gamma-1}\bigl(1 - e^{-\varepsilon(t-s)}\bigr)
		-
		(-s)^{\gamma-1}\bigl(1 - e^{\varepsilon s}\bigr)\right] \rd W^Q(s).
	\end{align*}
	After applying the It\^o isometry to each of these integrals
	and using the respective changes of variables $s' := - s$ and $s' := t -s$, we obtain
	\begin{align}\label{eq:limiting-case-splitting}
		\biggl\|\widehat W^Q_{\gamma - \frac{1}{2}}(t) &- \overline Z_\gamma^\varepsilon(t)\biggr\|_{L^2(\Omega; U)}^2
		=
		\frac{I_1(t) + I_2(t)}{C_{\gamma - \frac{1}{2}}}(\tr Q)^2, \\
		I_1(t) 
		&:= 
		\int_0^{|t|} s^{2\gamma-2} \bigl(1 - e^{-\varepsilon s}\bigr)^2 \rd s, 
		\\
		I_2(t) 
		&:= 
		\int_{0}^\infty \left\vert(|t|+s)^{\gamma-1}\bigl(1 - e^{-\varepsilon(|t|+s)}\bigr)
		-
		s^{\gamma-1}\bigl(1 - e^{-\varepsilon s}\bigr)\right\vert^2 \rd s.
	\end{align}
	For $t \in (-\infty, 0)$ we find~\eqref{eq:limiting-case-splitting}
	by instead splitting into integrals over
	$(-\infty, t)$ and $(t, 0)$ and changing variables
	$s' := t - s$ and $s' := -s$, respectively.
	
	For $t \in [-T, T]$, the estimate $1 - e^{-x} \le x$ for $x \in [0,\infty)$ yields for the first term:
	\begin{align*}
		I_1(t)
		\le \varepsilon^2 \int_0^{\lvert t \rvert} s^{2\gamma} \rd s
		=
		\frac{\varepsilon^2 \lvert t\rvert^{2\gamma + 1}}{2\gamma + 1}
		\le 
		\frac{\varepsilon^2 T^{2\gamma + 1}}{2\gamma + 1}.
	\end{align*}
	For $I_2$, we apply the fundamental theorem of calculus to the function
	\[
	u \mapsto (u + s)^{\gamma-1} (1-e^{-\varepsilon(u+s)}),
	\]
	followed by
	Minkowski's integral inequality~\cite[Section~A.1]{Stein1970} to find
	\begin{align*}
		I_2&(t)
		=
		\int_{0}^\infty 
		\biggl\vert 
		\int_0^{\lvert t \rvert} 
		[(\gamma-1)(u+s)^{\gamma-2} (1-e^{-\varepsilon (u+s)}) + \varepsilon (u+s)^{\gamma-1} e^{-\varepsilon (u+s)}] 
		\rd u 
		\biggr\vert^2 
		\rd s
		\\
		&\le 
		\Biggl(
		\int_0^{T} 
		\biggl[
		\int_{0}^\infty \Bigl|(\gamma-1)(u+s)^{\gamma-2} (1-e^{-\varepsilon (u+s)}) + \varepsilon (u+s)^{\gamma-1} e^{-\varepsilon (u+s)}\Bigr|^2 \rd s 
		\biggr]^{\frac{1}{2}}
		\rd u
		\Biggr)^2
		\\
		&=
		\varepsilon^{3-2\gamma}
		\Biggl(
		\int_0^{T} 
		\biggl[
		\int_{\varepsilon u}^\infty  \Bigl|(\gamma-1)v^{\gamma-2} (1-e^{-v}) + v^{\gamma-1} e^{-v}\Bigr|^2 \rd v
		\biggr]^\frac{1}{2}
		\rd u
		\Biggr)^2
		\\
		&\le 
		\varepsilon^{3-2\gamma} \, T^2 \int_{0}^\infty  \Bigl|(\gamma-1)v^{\gamma-2} (1-e^{-v}) + v^{\gamma-1} e^{-v}\Bigr|^2 \rd v,
	\end{align*}
	where we performed the change of variables $v(s) := \varepsilon(u + s)$ on the third line.
	
	The improper integral on the last line converges:
	As $v \downarrow 0$, the squares of both terms are of order $\cO(v^{2\gamma - 2})$, where 
	we again use $1 - e^{-v} \le v$ for the first term,
	and we have $2\gamma - 2 \in (-1, 1)$;
	the square of the first term is of order $\cO(v^{2\gamma - 4})$ as $v \to \infty$,
	with $2\gamma - 4 \in (-3, -1)$, whereas the second term decays exponentially.
	
	The convergence thus follows by letting
	$\varepsilon \downarrow 0$ in the previous two displays.
\end{proof}

\subsection{Remarks on Markov behavior}
\label{subsec:fQWP:Markov}
Now we consider the Markov behavior of fractional $Q$-Wiener processes with Hurst parameter $H \in (0,1)$.
Since the case $H = \frac{1}{2}$ corresponds to a standard $Q$-Wiener process, we find that $W^Q_{\nicefrac{1}{2}}$
is simple Markov, whereas we can expect that $W^Q_H$ is not weakly Markov
for $H \neq \frac{1}{2}$.

In the real-valued case, the first published proof of non-Markovianity 
appears to be~\cite{Huy2003}, which shows that $B_H$ is not simple Markov for $H \neq \frac{1}{2}$
using a characterization 
in terms of its covariance function which is valid for Gaussian processes.
This result can be improved by applying the theory of~\cite[Chapter~V]{HidaHitsuda1993} for Gaussian $N$-ple Markov processes
to the Mandelbrot--Van Ness representation of $B_H$. Namely,
according to~\cite[Theorem~5.1]{HidaHitsuda1993},
any real-valued Gaussian process of the form 
\[
\biggl(\int_{-\infty}^t K(t,s) \rd B(s)\biggr)_{t\in\bbR}
\]
can only be $N$-ple Markov for $N \in \bbN$ if its integration kernel $F$ can
be written as
\[
K(s,t) = \sum_{j=1}^N f_j(s) g_j(t)
\]
for some functions $(f_j)_{j=1}^N, (g_j)_{j=1}^N$. This condition is not
satisfied for $K_H$ from~\eqref{eq:Mandelbrot-vanNess-int-kern}.

In order to establish that $W^Q_H$ (or $B_H$)
does not have the weak Markov property for $H \neq \frac{1}{2}$, 
one could attempt to associate a
nonlocal precision operator to the process and apply the necessary condition
from Theorem~\ref{thm:weakMarkov-Gaussian-necessity}.
Formally, the coloring operator $\cL^{-1}_H$ 
of $W^Q_H$ acts on functions $f \from \bbR \to U$ as
\[
\cL^{-1}_H f(t) 
=
\frac{1}{C_H}
\int_\bbR K_H(t,s) f(s) \rd s,
\quad 
\forall t \in \bbR.
\]
For certain ranges of $H$, see for instance~\cite[Equation~(31)]{Picard2011}, 
an explicit formula of
its inverse $\cL_H$ can also be determined.
The operator
$\cL^{-1}_H$ is bounded on some
weighted H\"older space by~\cite[Theorem~6]{Picard2011},
but there is no reason to expect that it is bounded on
a Hilbert space such as $L^2(\bbR; U)$. Therefore, 
Theorem~\ref{thm:weakMarkov-Gaussian-necessity} is not directly applicable, as
it
would need to be extended to the Banach space setting, 
which is beyond the scope of this work.

\section*{Acknowledgments}
The authors acknowledge
fruitful discussions with Richard Kraaij which led to
Proposition~\ref{prop:WP-R-not-martingale}
and the choice of the two-sided $Q$-Wiener process 
to define the stochastic integral in Subsection~\ref{subsec:prelims:stochastic-int-twosided-Wiener}.
Moreover, the authors thank Jan van Neerven for
carefully reading the manuscript and 
providing valuable comments. 

\section*{Funding}

K.K.\ acknowledges support 
of the research project 
\emph{Efficient spatiotemporal statistical modelling 
	with stochastic PDEs} 
(with project number  
VI.Veni.212.021) 
by the talent programme \emph{Veni}
which is financed by 
the Dutch Research Council (NWO).

\appendix

\section{Spaces of vector-valued functions}\label{app:mixed-norm-bochner}
\subsection{Spaces of measurable, continuous and differentiable functions}
\label{subsec:app-func-spaces:measurable-cont}
Let $(S, \mathscr A, \mu)$ be a measure space.
We abbreviate the phrases 
``almost everywhere'' and 
``almost all'' 
by 
``a.e.''\ and ``a.a.''\!,  
respectively. 

Let $\cB(E)$ denote the Borel $\sigma$-algebra of $E$. 
Given $x \in E$ and $f \from S \to \R$, 
we define $f \otimes x \from S \to E$ by $[f\otimes x](\,\cdot\,) := f(\,\cdot\,)x$.
A function $f \from (S, \mathscr A, \mu) \to (E, \cB(E))$
is said to be $\mu$-strongly measurable if it is the $\mu$-a.e.\ limit 
of linear combinations of $\mathbf 1_{A} \otimes x$, where 
$x \in E$ and $A \in \mathscr A$ with $\mu(A) < \infty$.
We denote by $B_b(S, \mathscr A)$ (resp.\ $C_b(S)$) the Banach space of
$\mathscr A$-measurable (resp.\ continuous) and bounded functions $f \from S \to \R$ endowed with the 
norm $\norm{f}{B_b(S, \mathscr A)} = \norm{f}{C_b(S)} = \sup_{s \in S} \vert f(s) \vert$.
If the $\sigma$-algebra is clear from context, we may simply write $B_b(S)$.
The set $C_c^\infty(J; E)$ consists of infinitely differentiable functions 
from $J$ to $E$ with compact support.

\subsection{Bochner and Sobolev spaces}
\label{subsec:app-func-spaces:Bochner-Sobolev}
For $p \in [1,\infty)$, let $L^p(S, \mathscr A\!, \mu;E)$ denote the Bochner space of 
(equivalence classes of)
strongly measurable, 
$p$-integrable functions
equipped with the norm 
$
\norm{f}{L^p(S, \mathscr A \!, \mu;E)} 
:= 
\left( \int_S \norm{f(s)}{E}^p \rd \mu(s) \right)^{\!\nicefrac{1}{p}}\!\!.
$
If the $\sigma$-algebra and measure are clear from context,
we may suppress them from our notation and simply write $L^p(S; E)$.
The space $L^2(S, \mathscr A \!, \mu; U)$ is a Hilbert space
when equipped with the inner product 
$\scalar{f,g}{L^2(S, \mathscr A \!, \mu; U)} := \int_S \scalar{f(s),g(s)}{U} \rd \mu(s)$.

If $(S, \mathscr A\!, \mu)$ and $(T, \mathscr B, \nu)$ are $\sigma$-finite measure spaces,
then the measurable space $(S \times T, \mathscr A \otimes \mathscr B)$,
equipped with the product $\sigma$-algebra
\[
\mathscr A \otimes \mathscr B := \sigma(A \times B : A \in \mathscr A, B \in \mathscr B),
\]
admits a unique $\sigma$-finite \emph{product measure $\mu\times \nu$}
satisfying 
\[
(\mu \times \nu)(A \times B) = \mu(A) \nu(B) \quad \forall A \in \mathscr A, \, B \in \mathscr B.
\]
For $p, q \in [1,\infty)$, we define the \emph{mixed-exponent Bochner space}
$L^{p,q}(S \times T; E)$ as the Banach space of (equivalence classes of) 
strongly $(\mu \times \nu)$-measurable functions $f \from S \times T \to E$
such that
\begin{equation}
	\| f \|_{L^{p,q}(S \times T; E)}
	:=
	\biggl[\int_S \biggl[ \int_T \| f(s, t) \|_E^q \rd \nu(t)  \biggr]^{\frac{p}{q}} \rd \mu(s) \biggr]^{\frac{1}{p}}
	<
	\infty.
\end{equation}
Intervals $J \subseteq \R$ will be equipped with the Lebesgue measure $\lambda$
and Lebesgue $\sigma$-algebra $\cB(J) \vee \sigma(\cN_\lambda)$.
A function $f \in L^p(J; E)$ is said to be weakly differentiable,
with weak derivative $\partial_t f$, 
if
\[ 
\int_J \phi'(t) f(t) \rd t = -\int_J \phi(t) [\partial_t f](t) \rd t \quad \forall \phi \in C^\infty_c(J; \R).
\]
More generally, we denote its $n$th weak derivative by $\partial_t^n f$.
If $\partial_t^{\, j} f \in L^p(J; E)$
for all $j \in \{0, \dots, n\}$, then $f$ belongs to
the Bochner--Sobolev space $W^{n,p}(J; E)$, which is 
a Banach space using the norm
$\| f \|_{W^{n,p}(J; E)} := (\sum_{j=0}^{n} \| \partial_t^j f \|_{L^p(J; E)}^p)^{\nicefrac{1}{p}}$.
If $p = 2$ and $E = U$, then $H^n(J; U) := W^{n,2}(J; U)$ is a Hilbert space
whose norm is induced by the
inner product 
$\langle f,g \rangle_{H^n(J; U)} := \sum_{j=0}^n \langle \partial_t^{\, j} f, \partial_t^{\, j} g\rangle_{L^2(J;U)}$.

\section{Auxiliary results}\label{app:aux}
This appendix collects some auxiliary results which
are needed in the main text but have been postponed for
the sake of readability.

\subsection{Facts regarding conditional independence}

Let $\cG_1,\cH, \cG_2\subseteq \cF$ be $\sigma$-algebras on $(\Omega, \cF, \bbP)$.
We recall a characterization of conditional independence, see~\cite[Theorem~8.9]{Kallenberg2002}, 
from which we derive a lemma which is useful for establishing relations between 
various (equivalent formulations of) Markov properties defined in Section~\ref{sec:markov-properties}.

\begin{theorem}[Doob's conditional independence property]\label{thm:Doob-cond-ind}
	We have $\cG_1 \ind_{\cH} \cG_2$ if and only if
	\begin{equation}
		\bbP(G_2 \mid \cG_1 \vee \cH) = \bbP(G_2 \mid \cH),
		\quad \bbP\text{-a.s.}, 
		\quad 
		\forall G_2 \in \cG_2.
	\end{equation}
\end{theorem}

\begin{lemma}\label{lem:cond-ind-props}
	If $\cG_1 \ind_{\cH} \cG_2$, then
	\begin{enumerate}[(a)]
		\item $\cG_1 \vee \cH \ind_{\cH} \cG_2$; \label{lem:cond-ind-props:a}
		\item $\cG_1 \ind_{\cH'} \cG_2$ for any $\sigma$-algebra $\cH' \supseteq \cH$ such that $\cH' \subseteq \cG_1$; \label{lem:cond-ind-props:b}
		\item $\cG_1 \ind_{\cH'} \cG_2$ for any $\sigma$-algebra $\cH' \supseteq \cH$ of the form 
		$
		\cH' = \cH'_1 \vee \cH_2',
		$
		where $\cH'_1,\cH'_2$ are $\sigma$-algebras satisfying 
		$\cH_1' \subseteq \cG_1 \vee \cH$ and $\cH_2' \subseteq \cG_2 \vee \cH$. \label{lem:cond-ind-props:c}
	\end{enumerate}
\end{lemma}
\begin{proof}
	Part~\ref{lem:cond-ind-props:a} is~\cite[Corollary~8.11(i)]{Kallenberg2002};
	combining it with
	$\cG_1 = \cG_1 \vee \cH'$
	and ${\cH' = \cH \vee \cH'}$ yields~\ref{lem:cond-ind-props:b}.
	To prove~\ref{lem:Bochner-space-operator-c}, first note that
	$\cG_1 \vee \cH \ind_{\cH \vee \cH_1'} \cG_2 \vee \cH$
	by parts~\ref{lem:cond-ind-props:a} and~\ref{lem:cond-ind-props:b}.
	Applying~\ref{lem:cond-ind-props:a} again, we find
	\[
	\cG_1 \vee \cH \vee \cH_1' \ind_{\cH \vee \cH_1'} \cG_2 \vee \cH \vee \cH_1'.
	\]
	Since $\cH \vee \cH_1' \subseteq \cH \vee \cH' = \cH'
	= \cH_2' \vee \cH_1' \subseteq \cG_2 \vee \cH \vee \cH_1'$, part~\ref{lem:cond-ind-props:b}
	yields
	\[
	\cG_1 \vee \cH \vee \cH_1' \ind_{\cH'} \cG_2 \vee \cH \vee \cH_1',
	\]
	which proves~\ref{lem:cond-ind-props:c} since $\cG_1 \subseteq \cG_1 \vee \cH \vee \cH_1'$ 
	and $\cG_2 \subseteq \cG_2 \vee \cH \vee \cH_1'$.
\end{proof}

\subsection{Facts regarding Assumption \ref{ass:HAQ}\ref{ass:HAQ:1}}
\label{app:aux:facts-about-assumption}
\begin{lemma}\label{lem:intb-assumption-gammaleqhalf}
	Let Assumption~\textup{\ref{ass:standing}} be satisfied, i.e., 
	suppose the linear operator
	$-A \from \mathsf D(A) \subseteq U \to U$
	on the separable real Hilbert space $U$ generates an
	exponentially stable $C_0$-semigroup $(S(t))_{t\ge0}$ of bounded
	linear operators on $U$.
	If $Q \in \LO^+(U)$ and $\gamma_0 \in (-\infty, \nicefrac{1}{2}]$, then
	\[
	\int_0^\infty \| t^{\gamma_0-1} S(t) Q^\frac{1}{2} \|_{\mathscr L_2(U)}^2 \rd t = \infty,
	\]
	that is, Assumption~\textup{\ref{ass:HAQ}}\ref{ass:HAQ:1} cannot hold for $\gamma_0 \in (-\infty, \nicefrac{1}{2}]$.
\end{lemma}
\begin{proof}
	Fix some $x \in U$ with $\| x \|_U = 1$. Then
	$\| S(t) Q^\frac{1}{2} \|_{\mathscr L_2(U)} \ge \| S(t) Q^\frac{1}{2}x \|_{U}$ for all $t \in [0,\infty)$.
	Since $t \mapsto  S(t) Q^\frac{1}{2}x$ is continuous at zero
	and $S(0) Q^\frac{1}{2}x = Q^\frac{1}{2}x$, we can choose $\delta \in (0,\infty)$
	so small that $\| S(t) Q^\frac{1}{2}x \|_U \ge \frac{1}{2} \| Q^\frac{1}{2}x \|_U$ for all $t \in [0,\delta]$.
	If $\gamma_0 \in (-\infty, \nicefrac{1}{2}]$, we then obtain
	\[
	\int_0^\infty \| t^{\gamma_0-1} S(t) Q^\frac{1}{2} \|_{\mathscr L_2(U)}^2 \rd t 
	\ge 
	\frac{1}{2} \| Q^\frac{1}{2}x \|_U \int_0^\delta t^{2(\gamma_0-1)}\rd t
	=
	\infty.
	\qedhere
	\]
\end{proof}

\begin{lemma}\label{lem:intb-assumption-nestedness}
	Let Assumption~\textup{\ref{ass:standing}} be satisfied.
	If Assumptions~\textup{\ref{ass:HAQ}\ref{ass:HAQ:1}}
	holds for some $\gamma_0 \in (\nicefrac{1}{2}, \infty)$, then 
	it also holds for all $\gamma' \in [\gamma_0, \infty)$.
\end{lemma}
\begin{proof}
	The change of variables $\tau := t/2$, the semigroup property and~\eqref{eq:semigroup-est} yield
	\begin{align*}
		\int_0^\infty \| t^{\gamma' - 1} S(t)Q^\frac{1}{2} \|_{\mathscr L_2(U)}^2 \rd t
		&\le 
		2^{2\gamma' - 1} M_0^2 \int_0^\infty e^{-2w\tau} \| \tau^{\gamma' - 1} S(\tau)Q^\frac{1}{2} \|_{\mathscr L_2(U)}^2 \rd \tau.
	\end{align*}
	For the latter integral, we split up the domain of integration and estimate
	each of the resulting integrands to find
	\begin{align*}
		\int_0^\infty &e^{-2w\tau} \| \tau^{\gamma' - 1} S(\tau)Q^\frac{1}{2} \|_{\mathscr L_2(U)}^2 \rd \tau
		=
		\sum_{k=1}^{\infty} \int_{k-1}^k e^{-2w\tau} \| \tau^{\gamma' - 1} S(\tau)Q^\frac{1}{2} \|_{\mathscr L_2(U)}^2 \rd \tau
		\\
		&\le 
		\sum_{k=1}^{\infty} e^{-2w(k-1)} k^{2(\gamma' - \gamma_0)} \int_0^\infty  \| \tau^{\gamma_0-1} S(\tau)Q^\frac{1}{2} \|_{\mathscr L_2(U)}^2 \rd \tau 
		< \infty,
	\end{align*}
	where the series converges since $|e^{-2wk} k^{2(\gamma' - \gamma_0)}|^{\frac{1}{k}} \to e^{-2w} < 1$
	as $k \to \infty$.
\end{proof}

\subsection{Filtrations indexed by the real line}

The proof of Proposition~\ref{prop:filtr-rightcont} below elaborates on~\cite[Example~3.6]{BOGP2010}.
First we state and prove a lemma which shows that collections of random variables
are generated by a $\pi$-system consisting of finite products of bounded, real-valued random variables.
\begin{lemma}\label{lem:sigma-algebra-collection-RVs}
	If $(X_i)_{i\in\cI}$ is a collection of $U$-valued random variables, then
	\begin{equation}\label{eq:join-sigma-RVs-generated-by-products}
		\sigma(X_i : i \in \cI)
		=
		\sigma\left( \prod\nolimits_{i \in \cJ} Z_i : \cJ \subseteq \cI \textup{ finite}, \, Z_i \in B_b(\Omega, \sigma(X_i)),
		\forall i \in \cJ\!\right)\!.
	\end{equation}
\end{lemma}
\begin{proof}
	The ``$\supseteq$'' inclusion holds because any $Z = \prod_{i\in \cJ} Z_i$ is
	measurable with respect to $\bigvee_{i\in \cJ}  \sigma(X_i) \subseteq \bigvee_{i\in\cI} \sigma(X_i)$
	for any $\cJ \subseteq \cI$.
	For the converse inclusion, note that the join of a family of $\sigma$-algebras $(\cA_i)_{i\in\cI}$
	has the following general property:
	\[
	\bigvee_{i \in \cI} \cA_i 
	=
	\sigma\left( \bigcap\nolimits_{i\in \cJ} A_i : \cJ \subseteq \cI \text{ finite}, \, A_i \in \cA_i \,\,
	\forall i \in \cJ\right)\!.
	\]
	Thus, by definition of $\sigma(X_i)$, we obtain (see also~\cite[Chapter~0, Theorem~2.3]{RevuzYor1999})
	\begin{equation}\label{eq:join-sigma-RVs}
		\sigma(X_i : i \in \cI)
		=
		\sigma\left(\bigcap\nolimits_{i \in \cJ} X_i^{-1}[B_i] : \cJ \subseteq \cI \text{ finite}, \, B_i \in \cB(U) \,\,
		\forall i \in \cJ\right)\!.
	\end{equation}
	Since $\bigcap_{i\in \cJ} X_i^{-1}[B_i] = (\prod_{i\in \cJ} \mathbf 1_{B_i}(X_i))^{-1}[\{1\}]$
	for any finite $\{B_i : i \in \cJ\} \subseteq \cB(U)$,
	we find that the right-hand side of~\eqref{eq:join-sigma-RVs} is contained
	in that of~\eqref{eq:join-sigma-RVs-generated-by-products}.
\end{proof}

\begin{proposition}\label{prop:filtr-rightcont}
	Let $(W^Q(t))_{t\in\bbR}$ be a $U$-valued two-sided $Q$-Wiener process with covariance
	operator $Q \in \LO_1^+(U)$,
	see Subsection~\ref{subsec:prelims:stochastic-int-twosided-Wiener}.
	The filtration 
	$(\cF^{\delta W^Q}_t)_{t\in\bbR}$ defined in~\eqref{eq:increment-filtration}
	is right-continuous, i.e., we have $\cF_{t+}^{\delta W^Q} := \bigcap_{\varepsilon>0} \cF_{t+\varepsilon}^{\delta W^Q} = \cF_t^{\delta W^Q}$
	for all $t \in \bbR$.
\end{proposition}
\begin{proof}
	Fix $t \in \bbR$. 
	It suffices to prove that
	\begin{equation}\label{eq:right-continuity-filtration-Z}
		\E[Z \mid \cF^{\delta W^Q}_{t+}]
		=
		\E[Z \mid \cF^{\delta W^Q}_{t}],
		\quad
		\bbP\text{-a.s.},
	\end{equation}
	for all $Z \in B_b(\Omega, \cF_\infty^{\delta W^Q})$, where 
	$\cF_\infty^{\delta W^Q} := \sigma(W_{s,u} : s \le u)$,
	using the shorthand notation $W_{s,u} = W^Q(u) - W^Q(s)$.
	Indeed,
	given $A \in \cF_{t+}^{\delta W^Q} \subseteq \cF_\infty^{\delta W^Q}$,
	equation~\eqref{eq:right-continuity-filtration-Z} would imply that 
	$\mathbf 1_A$ is $\bbP$-a.s.\ equal to the $\cF_{t}^{\delta W^Q}$-measurable random variable
	$\E[\mathbf 1_A \mid \cF^{\delta W^Q}_{t}]$,
	and thus $A \in \cF_{t}^{\delta W^Q}$ by the completeness of $(\cF_t^{\delta W^Q})_{t\in\R}$.
	This yields $\cF_{t+}^{\delta W^Q} \subseteq \cF_{t}^{\delta W^Q}$\!.
	
	Let $\mathscr H$ denote the linear space of bounded random variables
	$Z \from \Omega \to \bbR$ for which~\eqref{eq:right-continuity-filtration-Z}
	holds; we want to show that $B_b(\Omega, \cF_{\infty}^{\delta W^Q}) \subseteq \mathscr H$
	using the monotone class theorem~\cite[Chapter~0, Theorem~2.2]{RevuzYor1999}.
	Let $\mathscr C$ denote
	the set of random variables
	$Z \from \Omega\to\bbR$
	which can be written as a product of
	two bounded random variables $Z_1, Z_2 \from \Omega\to\bbR$, 
	where $Z_1$ is $\cF_t^{\delta W^Q}$-measurable
	and
	$Z_2$ is measurable with respect to 
	$\sigma(W_{s,u} : t + \varepsilon \le s \le u) \vee \sigma(\cN_{\bbP})$
	for some $\varepsilon>0$.
	We first observe that $\mathscr C \subseteq \mathscr H$. Indeed,
	for all $Z \in \cK$ with $Z = Z_1 Z_2$ as above,
	\[
	\E[Z \mid \cF_{t+}^{\delta W^Q}]
	=
	\E[Z_1 Z_2 \mid \cF_{t+}^{\delta W^Q}]
	=
	Z_1 \E[Z_2]
	=
	\E[Z_1 Z_2 \mid \cF_{t}^{\delta W^Q}]
	=
	\E[Z \mid \cF_{t}^{\delta W^Q}],
	\]
	$\bbP$-a.s.,
	since $Z_1$ is measurable with respect to 
	$\cF_{t}^{\delta W^Q} \subseteq \cF_{t+}^{\delta W^Q}$,
	while $Z_2$ is independent of 
	$\cF_{t+}^{\delta W^Q} \supseteq \cF_{t}^{\delta W^Q}$ by~\eqref{item:WP3prime}. 
	We note moreover
	that $\mathscr C$ is closed under pointwise multiplication and that
	$\mathscr H$ 
	contains all
	constant functions and satisfies $Z \in \mathscr H$ whenever
	$(Z_n)_{n\in\N} \subseteq \mathscr H$ is such that $0 \le Z_n \uparrow Z$ by conditional
	monotone convergence.
	Thus the conditions of the monotone class theorem are indeed satisfied, hence
	$B_b(\Omega, \sigma(\mathscr C)) \subseteq \mathscr H$
	and it remains to show $\sigma(\mathscr C) = \cF^{\delta W^Q}_{\infty}$\!.
	
	The inclusion $\sigma(\mathscr C) \subseteq \cF^{\delta W^Q}_{\infty}$
	follows from the fact that any random variables $Z_1$ and $Z_2$ as above
	are both $\cF^{\delta W^Q}_{\infty}$-measurable,
	implying the same for $Z$.
	For the converse inclusion, we first split $\cF_\infty^{\delta W^Q}$ at time $t$
	as follows:
	\begin{equation}\label{eq:Finfty-join-at-t}
		\cF^{\delta W^Q}_{\infty} 
		= 
		\sigma(W_{s,u} : s \le u)
		=
		\cF_t^{\delta W^Q} \vee \sigma(W_{s,u} : t \le s \le u).
	\end{equation}
	Since $W^Q$ is (right-)continuous, $\bbP$-a.s., we have in fact
	\[
	\cF^{\delta W^Q}_{\infty} 
	=
	\cF_t^{\delta W^Q} \vee \sigma(W_{s,u} : t < s \le u),
	\]
	which can also be written as
	\[
	\cF_\infty^{\delta W^Q} = \sigma(W_{s,u} : (s, u) \in \cI) \vee \sigma(\cN_{\bbP}),
	\]
	where 
	$\cI := \cI_1 \cup \cI_2$
	with
	$\cI_1 := \{(s,u) : s \le u \le t \}$ and
	$\cI_2 := \{(s,u) :  t < s \le u\}$.
	An application of Lemma~\ref{lem:sigma-algebra-collection-RVs}
	to the above equation shows that $\cF_{\infty}^{\delta W^Q}$ is equal to
	\[
	\sigma\biggl( 
	\prod\nolimits_{(s,u) \in \cJ} Y_{(s,u)} 
	: 
	\cJ \subseteq \cI \textup{ finite}, \,
	Y_{(s,u)} \in B_b(\Omega, \sigma(W_{s,u})) \,\,
	\forall (s,u) \in \cJ\!
	\biggr)\! 
	\vee 
	\sigma(\cN_{\bbP}).
	\]
	Consider
	$Z := \prod_{(s,u) \in \cJ} Y_{(s,u)}$
	for some finite $\cJ \subseteq \cI$.
	Define
	$Z_1 :=  \prod_{(s,u) \in \cJ \cap \cI_1} Y_{(s,u)}$
	and $Z_2 := \prod_{(s,u) \in \cJ \cap \cI_2} Y_{(s,u)}$.
	Then we have $Z = Z_1 Z_2$, with $Z_1$ and $Z_2$
	as in the definition of $\mathscr C$, finishing the proof of the identity
	$\cF_\infty^{\delta W^Q} = \sigma(\mathscr C)$.
\end{proof}

\begin{proposition}\label{prop:WP-R-not-martingale}
	A process
	$(W^Q(t))_{t\in \bbR}$ satisfying 
	\textup{\ref{item:WP1}} cannot be a martingale with respect to any filtration $(\cF_t)_{t\in\bbR}$.
\end{proposition}
\begin{proof}
	Suppose that $(W^Q(t))_{t\in \bbR}$ is a martingale with respect to some filtration $(\cF)_{t\in\bbR}$.
	Then the same holds for the real-valued process $W_h^Q(t) := \scalar{W^Q(t), h}{U}$,
	where we choose $h \in U$ such that $\scalar{Qh,h}{U}^2 > 0$
	to ensure that $(W_h^Q(t))_{t\in\bbR}$ has nontrivial increments.
	In particular, $(W^Q_h(-n))_{n \in \N}$ is a backward martingale
	with respect to $(\cF_{-n})_{n \in \N}$, implying that
	it
	converges $\bbP$-a.s.\ and in $L^1(\Omega)$ as $n \to \infty$ 
	by the backward martingale convergence theorem \cite[Section~{12.7}, Theorem~4]{GrimmettStirzaker2001}. 
	But this contradicts~\ref{item:WP1}, since $(W^Q_h(-n))_{n \in \N}$
	cannot be a Cauchy sequence in $L^1(\Omega)$ as the process
	has (non-trivial) stationary increments.
\end{proof}

\subsection{Mean-square differentiability of stochastic convolutions}

The following lemma regarding 
mean-square continuity and differentiation under
the integral sign generalizes~\cite[Propositions~3.18 and~3.21]{KW2022} to stochastic convolutions
with respect to a two-sided Wiener process.

\begin{lemma}\label{lem:stoch-conv-Psi-diffb-under-int}
	Let $t_0 \in [-\infty, \infty)$ be such that
	$\Psi(\,\cdot\,)Q^\frac{1}{2}\in L^2(0, \infty; \mathscr L_2(U))$ and
	set $\bbT := [t_0, \infty)$ if $t_0 \in \bbR$ or $\bbT := \R$ if $t_0 = -\infty$.
	Then the stochastic convolution
	$(\int_{t_0}^t \Psi(t-s)\rd W^Q(s))_{t\in \bbT}$ is mean-square continuous.
	
	If $\Psi(\,\cdot\,)Q^\frac{1}{2}\in H^1_{0,\{0\}}(0, \infty; \mathscr L_2(U))$,
	then  
	$(\int_{t_0}^t \Psi(t-s)\rd W^Q(s))_{t\in \bbT}$ is mean-square differentiable on $\bbT$
	with 
	\begin{align}
		\forall
		t \in \bbT : \quad
		\frac{\rd}{\rd t}\int_{t_0}^{t} \Psi(t-s)\rd W^Q(s) 
		= 
		\int_{t_0}^t \partial_t \Psi(t-s) \rd W^Q(s),
		\quad 
		\bbP\text{-a.s.}
		\label{eq:stoch-conv-Psi-diffb-under-int-sign}
	\end{align}
\end{lemma}
\begin{proof} 
	Adapting the proof of~\cite[Proposition~3.18]{KW2022}
	amounts to establishing that
	\[
	\lim_{h \to 0} 
	\int_{0 \vee (-h)}^\infty 
	\| [\Psi(u + h) - \Psi(u)] Q^\frac{1}{2}\|^2_{\mathscr L_2(U)} \rd u
	=
	0,
	\]
	which holds by the continuity of translations in $L^2(0,\infty; \mathscr L_2(U))$, 
	see for instance~\cite[Lemma~A.4]{KW2022} for this specific form.
	Similarly, for the mean-square differentiability
	we note that compared to the proof of~\cite[Proposition~3.21]{KW2022}, only the terms
	$I_1^{h^\pm}$ change when passing from $t_0 = 0$ to a general $t_0$. 
	For $t \in \bbT$ and $h \in (0, \infty)$, we have
	\begin{equation}
		I^{h^+}_1 = \int_{t_0}^{t} 
		\biggl[ \frac{\Psi(t+h-s)-\Psi(t-s)}{h} - \Psi'(t-s)\biggr] \rd W^Q(s),
	\end{equation}
	whereas for $t \in (t_0, \infty)$ and $h \in [t_0 - t, 0)$ (or $h \in (-\infty, 0)$ if $t_0 = -\infty$),
	\begin{equation}
		I^{h^-}_1 = \int_{t_0}^{t+h} 
		\biggl[ \frac{\Psi(t+h-s)-\Psi(t-s)}{h} - \Psi'(t-s)\biggr] \rd W^Q(s).
	\end{equation}
	Using the It\^o isometry~\eqref{eq:ito-R} and applying the change of variables $u := t - s$:
	\begin{align}
		\bigl\| I^{h^{\pm}}_1 \bigr\|^2_{L^2(\Omega; U)}
		=
		\int_{0 \vee (-h)}^{t-t_0} \,
		\biggl\| \biggl[\frac{\Psi(u+h)-\Psi(u)}{h} - \Psi'(u)\biggr] Q^\frac{1}{2}\biggr\|^2_{\mathscr L_2(U)} \rd u.
	\end{align}
	The proof is finished upon noting that replacing $t$ by $t - t_0$ does not affect the remaining step of the argument.
\end{proof}

\section{Fractional powers of the parabolic operator}\label{app:frac-powers}
Let $A \from \mathsf D(A) \subseteq U \to U$ be a linear operator
on a real Hilbert space $U$.

\begin{definition}\label{def:frac-power-Hille-Phillips}
	Under Assumption~\ref{ass:standing}, we define
	the negative fractional power operator $A^{-\alpha}$ of order $\alpha \in (0,\infty)$ 
	as the $\mathscr L(U)$-valued Bochner integral
	\begin{equation}\label{eq:def-frac-power-Hille-Phillips}
		A^{-\alpha}
		:=
		\frac{1}{\Gamma(\alpha)}
		\int_0^\infty 
		t^{\alpha-1} S(t) \rd t.
	\end{equation}
	Then $A^{-\alpha}$ is injective and we define
	$A^\alpha x := (A^{-\alpha})^{-1} x$ for 
	$x \in \mathsf D(A^\alpha) := \mathsf R(A^{-\alpha})$.
	For $\alpha = 0$ we set $A^0 := \id_U$.
\end{definition}
See~\cite[Section~2.6]{Pazy1983} for more details regarding
this definition of fractional powers.

\begin{lemma}\label{lem:Bochner-space-operator}
	Let $(S, \mathscr A, \mu)$ be a measure space
	such that $L^2(\Sigma; \R)$ is nontrivial and consider the linear
	operator $\cA_S$ on $L^2(S; U)$
	defined by~\eqref{eq:Bochner-counterpart}.
	Then the following statements hold:
	\begin{enumerate}[leftmargin=1cm, label={\normalfont(\alph*)}]
		\item\label{lem:Bochner-space-operator-a} 
		$\mathcal A_S \in \mathscr L(L^2(S; U))$ if and only if 
		$A \in \mathscr L(U)$, with equal operator norms;
		\item\label{lem:Bochner-space-operator-b}
		$\mathcal A_S$ is closed if and only if $A$ is;
		\item \label{lem:Bochner-space-operator-c}
		If $A$ generates a $C_0$-semigroup $(T(t))_{t\ge0}$ on $U$, 
		then
		$\cA_S$ generates the $C_0$-semigroup $(\cT_S(t))_{t\ge0}$ on 
		$L^2(S; U)$.
	\end{enumerate}
\end{lemma}

\begin{proof}
	For~\ref{lem:Bochner-space-operator-a}--\ref{lem:Bochner-space-operator-b} 
	one can argue as in the proof of~\cite[Lemma~A.1]{KW2022}, 
	replacing $T^{-\frac{1}{2}} \mathbf 1_{(0,T)}$
	by 
	another function with unit $L^2(S; \R)$-norm.
	The proof of~\ref{lem:Bochner-space-operator-c} is entirely analogous to
	that of~\cite[Proposition~A.3]{KW2022}.
\end{proof}

The next two statements and their proofs are analogous
to Propositions~A.5 and~3.2 of \cite{KW2022}, respectively.

\begin{proposition}
	For $t \in \R$, let the operator $\cT(t) \in \mathscr L(L^2(\R; U))$
	be defined by 
	\begin{equation}\label{eq:def-translationgroup}
		\cT(t) f := f(\,\cdot\, - t), \quad f \in L^2(\R; U).
	\end{equation}
	The family $(\cT(t))_{t\in\R}$ is a $C_0$-group whose infinitesimal
	generator is given by $-\partial_t$, where $\partial_t$ 
	is the Bochner--Sobolev weak derivative on 
	$\mathsf D(\partial_t) = H^1(\R; U)$.
\end{proposition}
\begin{proposition}\label{prop:parabolic-operator-semigroup}
	Suppose that Assumption~\textup{\ref{ass:standing}} holds.
	The closure $\cB := \clos{\partial_t + \cA_\R}$
	of the sum operator $\partial_t + \cA_\R$ exists and
	$-\cB$ generates the $C_0$-semigroup 
	$(\cS_\R(t)\cT(t))_{t\ge0}$ on $L^2(\R; U)$, which satisfies
	\[
	\forall t \in \R \colon 
	\quad 
	\| \cS_\R(t) \cT(t) \|_{L^2(\R; U)}
	=
	\| \cT(t) \cS_\R(t) \|_{L^2(\R; U)}
	=
	\| S(t) \|_{\mathscr L(U)},
	\]
	where $(\cS_\R(t))_{t\ge0}$ and $(\cT(t))_{t\ge0}$ are defined
	as in~\eqref{eq:Bochner-counterpart} and~\eqref{eq:def-translationgroup},
	respectively.
\end{proposition}
We thus find that $(\cS_\R(t)\cT(t))_{t\ge0}$ inherits the exponential
stability of $(S(t))_{t\ge0}$, so that fractional powers
of $\cB$ can be defined using Definition~\ref{def:frac-power-Hille-Phillips}. Therefore,
under Assumption~\ref{ass:standing}, equation~\eqref{eq:def-frac-power-Hille-Phillips} implies
\[
\cB^{-\gamma} f(t) 
= 
\frac{1}{\Gamma(\gamma)} 
\int_0^\infty r^{\gamma-1} \cS_\R(r) \cT(r) f(t) \rd r
=
\frac{1}{\Gamma(\gamma)} 
\int_0^\infty r^{\gamma-1} S(r) f(t-r) \rd r
\]
for $\gamma \in (0,\infty), f \in L^2(\R; U)$ and almost all $t \in \R$.
We conclude that $\cB^{-\gamma} = \mathfrak I^\gamma$ for 
all $\gamma \in [0,\infty)$, where the latter is defined by~\eqref{eq:def-Igamma}.

\bibliographystyle{siam}
\bibliography{kw-spde-markov.bib}

\end{document}